\documentclass[10pt]{article}
\usepackage{amsbsy,amssymb,amsthm,amsmath,latexsym}
\usepackage{syntonly}
\usepackage{enumerate, mathrsfs, bbm}
\usepackage{graphicx, tikz}
\usepackage{authblk, setspace, cite}
\usepackage[top=1 in, bottom=1 in, left=1.1 in, right=1.1 in]{geometry}
\usepackage[pdftex,linkcolor=red,citecolor=blue,backref=page]{hyperref}

\usepackage{color}
\usepackage{epsfig,graphics}

\newtheorem{thm}{Theorem}[section]
\newtheorem{lem}[thm]{Lemma}
\newtheorem{defi}[thm]{Definition}

\newtheorem{cor}[thm]{Corollary}
\newtheorem{pro}[thm]{Proposition}
\newtheorem{rem}[thm]{Remark}
\newtheorem{cla}[thm]{Claim}

\numberwithin{equation}{section}

\makeatletter
\def\blfootnote{\xdef\@thefnmark{}\@footnotetext}
\makeatother

\newcommand{\R}{\mathbb R}

\def\m{\mathbb}  \def\mcal{\mathcal}  \def\mb{\mathbbm}		
\def\eps{\epsilon}  \def\a{\alpha}  \def\b{\beta}  \def\lam{\lambda}	   \def\si {\sigma}

\def\p{\partial}  \def\ls{\lesssim}	\def\gs{\gtrsim}
\def\la{\langle}  \def\ra{\rangle}  \def\l{\left} \def\r{\right} 
\def\t{\tilde}	\def\wh{\widehat}	\def\wt{\widetilde}
	\def\i{\int\limits}		
\def\be{\begin{equation}}     \def\ee{\end{equation}}
\def\bp{\begin{pmatrix}}	\def\ep{\end{pmatrix}}

\def\H{\mathcal{H}}

\title{ Non-homogeneous boundary value problems for coupled  KdV-KdV systems posed on the half line}
\author[]{Shenghao Li, Min Chen, Xin Yang and Bing-Yu Zhang}
\date{}

\begin{document}

\maketitle

\begin{abstract}
In this article, we study an initial-boundary-value problem of a coupled  KdV-KdV system on the half line $ \mathbb{R}^+ $ with non-homogeneous boundary conditions:
\begin{equation*}
	\begin{cases}
		u_t+v_x+u u_x+v_{xxx}=0, \\
		v_t+u_x+(vu)_x+u_{xxx}=0, \\
		u(x,0)=\phi (x),\quad v(x,0)=\psi (x),\\
		u(0,t)=h_1(t),\quad v(0,t)=h_2(t),\quad v_x(0,t)=h_3(t),
	\end{cases} \qquad x,\,t>0.
\end{equation*}
It is shown that the problem is locally unconditionally well-posed  in $H^s(\mathbb{R}^+)\times H^s(\mathbb{R}^+)$ for  $s> -\frac34 $ with initial data $(\phi,\psi)$ in  $H^s(\mathbb{R}^+)\times H^{s}(\mathbb{R}^+)$ and boundary data
$(h_1,h_2,h_3) $ in $H^{\frac{s+1}{3}}(\mathbb{R}^+)\times H^{\frac{s+1}{3}}(\mathbb{R}^+)\times H^{\frac{s}{3}}(\mathbb{R}^+)$. The approach developed in this paper can also be applied to study more general KdV-KdV systems posed on the half line. 
		
\end{abstract}
	


\begin{center}
	\tableofcontents
\end{center}

\section{Introduction}
The coupled KdV-KdV system 
\begin{equation} \label{ckdv0}
	\begin{cases}
		u_t + v_x + u u_x + v_{xxx} = 0, \\ 
		v_t + u_x + (vu)_x + u_{xxx} = 0, 
	\end{cases} 
\end{equation}
is a special case of a broad class of Boussinesq systems or the so-called abcd systems derived by Bona, Chen and Saut in \cite{BCS02,BCS04} from the two-dimensional Euler equations. Compared to those uni-directional models, such as the KdV equation, the Boussinesq equation and the BBM equation, the bi-directional systems (\ref{ckdv0}) provide wider range of applications in reality.  The Cauchy problem of the system (\ref{ckdv0})  posed either on $\R$ or on the  periodic domain  $\mathbb{T}$  has been well studied. In particular,  it is  known to be analytically well-posed  in the space $H^s (\R)$ for any $s>-\frac34$, but related bilinear estimates fail in $H^s (\R)$ for any $s<-\frac34$. (See e.g. \cite{AC08, BGK10, YZ22a}). In this paper, we are concerned with the well-posedness of the initial-boundary-value problem (IBVP) for the coupled KdV-KdV system posed on the half line $ \R^+ $:
\begin{equation}\label{ckdv}
	\begin{cases}
		u_t+v_x+u u_x+v_{xxx}=0, \\
		v_t+u_x+(vu)_x+u_{xxx}=0, \\
		u(x,0)=\phi (x),\quad v(x,0)=\psi (x),\\
		u(0,t)=h_1(t),\quad v(0,t)=h_2(t),\quad v_x(0,t)=h_3(t).
	\end{cases} \qquad x, \, t>0.
\end{equation}

Let's first briefly review some highly related literatures. The IBVP for the single KdV equation
\begin{equation}\label{KdV}
	\begin{cases} 
		w_t + w_x + w_{xxx} = 0, \\
		w(x,0) = \phi(x), \quad w(0,t) = h(t).
	\end{cases} \qquad x, \, t>0.
\end{equation}
posed on the half line has been investigated extensively.  Bona, Sun and Zhang\cite{BSZ02} developed a method by using the Laplace transform to establish the well-posedness of (\ref{KdV}) in the space $ H^{s}(\R^+) $ for $ s>\frac34 $. Later, they improved the range of $ s $ to be $ s>-\frac34 $ in \cite{BSZ06} by taking advantage of the Fourier restriction spaces introduced by Bourgain\cite{Bou93a}. During the same period, Colliander and Kenig\cite{CK02}  introduced the Duhamel boundary forcing operator to establish the well-posedness of (\ref{KdV}) in the space $ H^{s}(\R^+) $ for $ s\geq 0 $. After that, Holmer\cite{Hol06} generalized this method to improve the range of the index $ s $ to be $ s>-\frac34 $ by constructing an analytic family of boundary forcing operators. These two methods automatically establish locally analytical well-posedness since both of them took advantage of a certain fixed point argument. As a result, the threshold $ -\frac34 $ is optimal concerning the analytical well-posedness, see \cite{CCT03}. These two methods were also applied to study the well-posedness of the IBVP for more general equations, such as the Boussinesq equations \cite{Xue08, CT17, LCZ18}, the BBM equation \cite{BCH14}, the nonlinear Schodinger equations \cite{Hol05, ET16, Cav17, BSZ18, GW20}, the Klein-Gordon-Schodinger system \cite{CT19}, the Schodinger-KdV systems \cite{CC19}
and many more. But few works have been devoted to investigate the low regularity well-posedness of the IBVP for KdV-KdV systems such as (\ref{ckdv}). Part of the reason is due to the complicated structure of the systems. Actually, it was discovered in \cite{Oh09} and \cite{YZ22a} that the threshold index for the well-posedness of the Cauchy problem for KdV-KdV systems on $ \R $ may be as large as $ 0 $ or $ \frac34 $ instead of $ -\frac34 $, depending on the dispersion coefficients and how $ u $ and $ v $ are interacted in the systems. The goal of this paper is to study the IBVP for the particular KdV-KdV system (\ref{ckdv}) on the half line $ \R^+ $ and establish the well-posedness beyond the threshold $ -\frac34 $.


In order to adopt results and techniques obtained in the study of KdV-type equations, the new variables $U := \frac14(u+v)$ and $V := \frac14(u-v)$ are introduced so that the linear parts of (\ref{ckdv}) on $ U $ and $ V $ are decoupled with the KdV type. More precisely, the new system reads as 
\[\begin{cases}
	U_t+U_x+U_{xxx}+3UU_x+(UV)_x-VV_x=0, \vspace{0.03in}\\
	V_t-V_x-V_{xxx}-UU_x+(UV)_x+3VV_x=0, \vspace{0.03in}\\
	U(x,0)=\frac14(\phi +\psi ),\quad V(x,0)=\frac14(\phi -\psi ), \vspace{0.03in}\\
	U(0,t)=\frac14(h_1+h_2),\quad V(0,t)=\frac14(h_1-h_2),\quad V_x(0,t)-U_x(0,t)=-\frac12 h_3.
\end{cases} \qquad x,\,t>0.\]
For ease of notations, we still denote $(U,V)$ as $(u,v)$, and write $\left ( \frac14(\phi +\psi ),\frac14(\phi -\psi ) \right )$ as $(\phi, \psi)$ and $\left (\frac14(h_1+h_2),\frac14(h_1-h_2),-\frac12 h_3 \right )$ as $(h_1, h_2, h_3)$. Then it reduces to study the following IBVP:
\begin{equation}\label{ckdv-1}
	\begin{cases}
		u_t+u_{xxx}+u_{x}+3uu_x+(uv)_x-vv_x=0,\\
		v_t-v_{xxx}-v_{x}-uu_x+(uv)_x+3vv_x=0,\\
		u(x,0)=\phi (x),\quad v(x,0)=\psi(x),\\
		u(0,t)=h_1(t),\quad v(0,t)=h_2(t),\quad v_x(0,t)-u_x(0,t)=h_3 (t),
	\end{cases} \qquad x,t>0.
\end{equation}
which is equivalent to the IBVP (\ref{ckdv}) for the well-posedness in the space $H^s (\R^+)\times H^s(\R^+)$. 

For any $ s\in\R $, we define
\[ {\cal H}_x^s(\R^+) := H^s(\R^+)\times H^s (\R^+), \quad {\cal H}^s_t (\R^+) := H^{\frac{s+1}{3}} (\R^+) \times H^{\frac{s+1}{3}} (\R^+) \times H^{\frac{s}{3}}(\R^+). \]
When $ s\geq 3 $, the well-posedness of (\ref{ckdv-1}) in the space $H^s (\R^+)\times H^s(\R^+)$ can be established concerning the classical solutions $ (u,v)\in C([0,T]; {\cal H}^s_x (\R^+)) $. So in the remaining of this paper, we will focus on the case $ s\leq 3 $. In addition, for lower regularity solutions, there will be  a serious uniqueness issue. For more details, the readers can refer to Remark \ref{Re, uwp} in section 2. In this situation, the type of solutions of (\ref{ckdv-1}) that we are interested in is the so-called mild solutions which can be viewed as limits of strong solutions. The precise definition is given below.

\begin{defi}[Mild solutions] \label{Def, m-s}
	Let $s\leq 3 $ and $T>0$ be given. 
	\begin{itemize} 
		\item[(i)]  A function pair $(u, v)\in C([0,T]; {\cal H}^3_x (\R^+)) \cap C^1([0,T]; {\cal H}^0_x (\R^+)) $ is said to be a strong solution of the IBVP (\ref{ckdv-1}) in the space $ C([0,T]; {\cal H}^s_x (\R^+)) $ if the equations in (\ref{ckdv-1})  hold for a.e. $(x,t)\in \R^+\times (0,T) $ and $ (u,v, v_x) |_{x=0} \in {\cal H}^s_t (0,T)$.
		
		\item[(ii)] A function pair $(u, v)\in C([0,T]; {\cal H}^s_x (\R^+)) $ is said to be a mild solution of  the IBVP (\ref{ckdv-1})  if  there exists a sequence of strong solutions  $(u_n, v_n) \in C([0,T]; {\cal H}^3_x (\R^+))$  such that
		$\lim\limits _{n\to \infty} (u_n, v_n) \to (u,v) $ in $ C([0,T]; {\cal H}^s_x (\R^+))$ and 
		$\lim\limits _{n\to 0} (h_{1,n}, h_{2,n} , h_{3, n} ) \to (h_1, h_2, h_3)$  in  $ {\cal H}^s_t (0,T)$, where $ (h_{1,n}, h_{2,n} , h_{3, n} ) := \left. (u_n, v_n, \partial _x v_n)\right |_{x=0}$.
	\end{itemize}
\end{defi}

After the concept of solutions is fixed, we can discuss the (unconditional) local  well-posedness of (\ref{ckdv-1}) which is understood in the following sense.

\begin{defi}[Local well-posedness]\label{Def, lwp}
	For any $ s\leq 3 $, the IBVP (\ref{ckdv-1}) is said to be locally well-posed in the space ${\cal H}_x^s(\R^+)$ if for any $r>0$,  there exists some time $T>0$, depending only on $r$ and $s$, such that for  any naturally compatible \footnote{For the concept of compatibility, the reader can refer to Definition \ref{Def, compatibility-lin}.}  $(\phi, \psi) \in {\cal H}^s _x(\R^+)$ and $\vec{h}= (h_1, h_2, h_3 ) \in  {\cal H}^s_t (\R^+)$
	with $ \|\phi , \psi )\|_{{\cal H}^s_x (\R^+)} + \|\vec{h} \|_{{\cal H}^s_t (\R^+)}  \le r$, 
	the IBVP (\ref{ckdv-1}) admits a unique mild solution $(u,v) \in C([0,T];{\cal H}^s_x (\R^+))$ and, moreover, the solution map is continuous  in the corresponding spaces. If the solution map is real analytic instead, then the IBVP (\ref{ckdv-1}) is said to be locally analytically well-posed.
\end{defi}

We intend to find  those values of $s$ for which the IBVP (\ref{ckdv-1}) is locally well-posed in the space ${\cal H}^s_x (\R^+)$. The following theorem is the main result of this paper.

\begin{thm}\label{Thm, main}
	The IBVP (\ref{ckdv-1}) is locally analytically  well-posed in the space ${\cal H}^s_x (\R^+) $, with any naturally compatible data $(\phi, \psi) \in {\cal H}^s _x(\R^+)$ and $\vec{h}= (h_1, h_2, h_3 ) \in  {\cal H}^s_t (\R^+) $,
	for any $s> -\frac34$.
\end{thm} 

To prove Theorem \ref{Thm, main} which is a result on local well-posedness, we  use scaling argument to introduce
\[u^{\b}(x,t)=\b u(\b^{\frac12}x,\b^{\frac32}t) \quad\text{and}\quad v^{\b}(x,t)=\b v(\b^{\frac12}x,\b^{\frac32}t),\]
where $\beta\in(0,1]$ is a parameter.
Then (\ref{ckdv-1}) becomes 
\[\begin{cases}
	u^\b_t+u^\b_{xxx}+\b u^\b_{x}+3u^\b u^\b_x+(u^\b v^\b)_x-v^\b v^\b_x=0,\\
	v^\b_t-v^\b_{xxx}-\b v^\b_{x}-u^\b u^\b_x+(u^\b v^\b)_x+3v^\b v^\b_x=0,\\
	u^\b(x,0)=\phi^\b(x),\quad v^\b(x,0)=\psi^\b(x),\\
	u^\b(0,t)=h_1^\b(t),\quad v^\b(0,t)=h_2^\b(t),\quad v^\b_x(0,t)-u^\b_x(0,t)=h_3^\b(t),
\end{cases} \qquad x,t>0,\]
where 
\[\begin{cases}
	\phi ^\b(x) := \b \phi (\b^{\frac12}x),\quad \psi ^\b(x) := \b \psi(\b^{\frac12}x),\\
	h_1^\b(t) := \b h_1(\b^{\frac32}t), \quad h_2^\b(t) := \b h_2(\b^{\frac32}t) ,\quad h_3^\b(t) := \b^{\frac32}h_3(\b^{\frac32} t).
\end{cases}\]
When $ s> -\frac34$, 
\[\|(\phi^\b,\psi ^\b)\|_{\mcal{H}^s_x(\m{R}^+)}\ls \b^{\frac38}\|(\phi ,\psi )\|_{\mcal{H}^s_1(\m{R}^+)},\quad 
\|(h_1^\b,h_2^\b,h_3^\b)\|_{\mcal{H}^s_t(\m{R}^+)}\ls \b^{\frac38}\|(h_1,h_2,h_3)\|_{\mcal{H}^s_t(\m{R}^+)}, \]
which implies 
\[\lim_{\b\to 0^+}\|(\phi ^\b,\psi ^\b)\|_{\mcal{H}^s_x(\m{R}^+)}+\|(h_1^\b,h_2^\b,h_3^\b)\|_{\mcal{H}^s_t(\m{R}^+)}=0.\]
Thus, in order to prove Theorem \ref{Thm, main}, it is sufficient  to establish Theorem \ref{Thm, main-scaling} below for  the following IBVP:
\begin{equation}\label{ckdv-2}
	\begin{cases}
		u_t+u_{xxx}+\b u_{x}+3uu_x+(uv)_x-vv_x=0,\\
		v_t-v_{xxx}-\b v_{x}-uu_x+(uv)_x+3vv_x=0,\\
		u(x,0)=\phi (x),\quad v(x,0)=\psi(x),\\
		u(0,t)=h_1(t),\quad v(0,t)=h_2(t),\quad v_x(0,t)-u_x(0,t)=h_3(t).
	\end{cases} \qquad x,t>0.
\end{equation}

\begin{thm}\label{Thm, main-scaling} 
	For any $s>-\frac34 $ and $T>0$,  there exists an $\epsilon >0$, depending only on $T$ and $s$, such that  for any $\b\in(0,1]$ and  for any naturally compatible data $(\phi, \psi) \in {\cal H}^s _x(\R^+)$ and $\vec{h}= (h_1, h_2, h_3 ) \in  {\cal H}^s_t (\R^+)  $ with
	\[ \|\phi , \psi )\|_{{\cal H}^s_x (\R^+)} + \|\vec{h} \|_{{\cal H}^s_t (\R^+)}  \le \epsilon , \]
	the IBVP (\ref{ckdv-2}) admits a unique mild solution $(u,v) \in C \big( [0,T];{\cal H}^s_x (\R^+) \big)$, and moreover, the solution map is real analytic in the corresponding spaces.
\end{thm}


We remark that the well-posedness obtained in Theorem \ref{Thm, main-scaling} (or equivalently Theorem \ref{Thm, main}) is unconditional (see Remark \ref{Re, uwp}) since the uniqueness is secured in the space $ C\big([0,T]; {\cal H}^s_x(\R^+)\big) $ itself rather than in a smaller subspace. The reason that the unconditional uniqueness can be justified is because we are concerning with the mild solutions. If the distributional solutions, instead of the mild solutions, are adopted in Definition \ref{Def, lwp}, then the unconditional uniqueness is still a challenging open problem.

\section{Outline of the method and applications}

\subsection{Key ingredients in the proof of Theorem \ref{Thm, main-scaling}}

Theorem \ref{Thm, main-scaling} will be proved using the  approach similar to  that  developed in  \cite{BSZ02,BSZ06, CK02, Hol06}  for the KdV equation but with some modifications to handle the bi-directional system \eqref{ckdv-2}. Consider the Cauchy problem for the linear KdV-type equation posed on $\R$, 
\be\label{linear eq-1}
w_t+\alpha w_{xxx}+\beta w_x=0, \qquad  w (x,0) =w_{0}(x) , \qquad  x\in\m{R},\, t\in\m{R}, \ee
where $\alpha\in\m{R}\backslash\{0\}$ and $\b\in\m{R}$.  Let   $W_{R}^{\alpha,\beta}$  be the semigroup operator associated to (\ref{linear eq-1}).  The solution of (\ref{linear eq-1})  is then given  explicitly by 
\be\label{semigroup op-1}
W_{R}^{\alpha,\beta}(t)w_{0}(x)=\int_{\m{R}}e^{i\xi x}e^{i\phi^{\alpha,\beta}(\xi)t}\widehat{w_{0}}(\xi)d\xi,\ee
where $\widehat{w_{0}}(\xi)$ is the Fourier  transform of $w_0 (x)$, and $\phi^{\alpha,\beta}$ is the cubic polynomial defined as below:
\be\label{phase fn}
\phi^{\alpha,\beta}(\xi) = \alpha \xi^{3}-\beta \xi. \ee


\begin{defi}\label{Def, FR space}
	For $\a,\b,s,b\in\R$, the Fourier restriction space $X^{\a,\b}_{s,b}$ is defined to be the completion of the Schwarz class ${\mathcal S}(\R^2)$ under the norm
	\be\label{FR norm}
	\|w\|_{X^{\a,\b}_{s,b}}=\left\|\langle\xi\rangle^s\langle\tau-\phi^{\a,\b}(\xi)
	\rangle^b\widehat{w}(\xi,\tau)\right\|_{L^2_{\xi,\tau}(\R^2)},
	\ee
	where $\phi^{\a,\b}$ is as defined in (\ref{phase fn}) and $\widehat{w}$ represents the space-time Fourier transform of $w$. 
\end{defi}
When $b>\frac12$, it follows directly from the Sobolev embedding that $X^{\a,\b}_{s,b}\subset C_{t}(\m{R};H^{s}(\m{R}))$. But the estimate on the boundary integral operator (see e.g. Lemma \ref{Lemma, pos-kdv-bdr} and \ref{Lemma, neg-kdv-bdr}) forces $b\leq \frac12$, and then the space $X^{\a,\b}_{s,b}$ may not be in $C_{t}(\m{R};H^{s}(\m{R}))$. In addition, the bilinear estimates used in the Cauchy problems will fail. In order to overcome these difficulties, we introduce some modified Fourier restriction spaces  $X^{\a,\b}_{s,b,\sigma}$ and $Y_{s,b,\sigma}^{\a,\b}$.   The spaces $X^{\a,\b}_{s,b,\sigma}$ will be  used to justify the bilinear estimates and $Y_{s,b,\sigma}^{\a,\b}$  will serve as suitable solution spaces. 

\begin{defi}\label{Def, MFR spaces}
For $\a,\b,s,b\in\R$ and $\sigma>\frac12$,  let  $\Lambda^{\a,\b}_{s,\sigma}$  be the completion of the Schwarz class ${\mathcal S}(\R^2)$ under the norm
\be\label{Lambda norm}
\|w\|_{\Lambda^{\a,\b}_{s,\sigma}}=\left\|\mb{1}_{\{e^{|\xi|}\leq 3+|\tau|\}}\la\xi\ra^{s}\la\tau-\phi^{\a,\b}(\xi)\ra^\si\wh{w}(\xi,\tau)\right\|_{L^2_{\xi,\tau}(\R^2)},\ee
where $\mb{1}$ means the characteristic function. In addition, we denote
\[X^{\a,\b}_{s,b,\sigma}=X^{\a,\b}_{s,b}\cap\Lambda^{\a,\b}_{s,\sigma} \quad\text{and}\quad Y_{s,b,\sigma}^{\a,\b} =X^{\a,\b}_{s,b,\si}\cap C_{t}(\m{R},H^{s}(\m{R})).\]
The norm of  the space  $X^{\a,\b}_{s,b,\sigma}$ is defined to be 
\be\label{MFR norm}\begin{split}
	\| w\|_{X^{\a,\b}_{s,b,\sigma}} &=\| w\|_{X^{\a,\b}_{s,b}}+\| w\|_{\Lambda^{\a,\b}_{s,\sigma}}\\
	&\sim \left\|\la \xi\ra ^{s}\big[\la L\ra ^{b}+\mb{1}_{\{e^{|\xi|}\leq 3+|\tau|\}}\la L\ra^{\sigma}\big]\wh{w}(\xi,\tau)\right\|_{L^{2}_{\xi,\tau}(\m{R}^2)},
\end{split}\ee
where $L=\tau-\phi^{\a,\b}(\xi)$, and the norm of the space  $Y_{s,b,\sigma}^{\a,\b}$ is defined to be 
\be\label{Y norm}
\|w\|_{Y^{\a,\b}_{s,b,\si}}=\|w\|_{X^{\a,\b}_{s,b,\si}}+\sup_{t\in\R}\|w(x,t)\|_{H_{x}^s(\R)}.\ee
\end{defi}
	
From the above definition, we see that when $b\leq \frac12$, the space $X^{\a,\b}_{s,b,\sigma}$ provides more regularity in time than the space $X^{\a,\b}_{s,b}$ in the low frequency region. This improvement is essential to establish the bilinear estimates in Proposition  \ref{Prop, bilin}. Note that  for any $b\leq \sigma$, 
\be\label{MFR, ub}\| w\|_{X^{\a,\b}_{s,b,\sigma}} \leq 2\| w\|_{X^{\a,\b}_{s,\sigma}}.\ee
Besides, $e^{|\xi|}\leq 3+|\tau|$ when  $|\xi|\leq 1$. Thus,
\be\label{MEF, lb}\| w\|_{X^{\a,\b}_{s,b,\sigma}}\gs \left\|\la \xi\ra ^{s}\big[\la L\ra ^{b}+\mb{1}_{\{|\xi|\leq 1\}}\la L\ra^{\sigma}\big]\wh{w}(\xi,\tau)\right\|_{L^{2}_{\xi,\tau}(\m{R}^2)},\quad \mbox{where $ L=\tau-\phi^{\a,\b}(\xi) $}.\ee

For technical reasons, when $\sigma>\frac12$, we define the $Z_{s,\sigma-1}^{\a,\b}$ space to be the completion of the Schwarz class $\mcal{S}(\m{R}^2)$ with respect to the following norm.
\be\label{Z norm}
\|w\|_{Z^{\a,\b}_{s,\sigma-1}}:=\left\|\langle\tau\rangle^{\frac{s}{3}+\frac12-\sigma}\la\tau-\phi^{\a,\b}(\xi)\ra^{\sigma-1}\widehat{w}(\xi,\tau)\right\|_{L^2_{\xi,\tau}(\R^2)}.
\ee
This space is an intermediate space which will only be used to estimate the inhomogeneous term in the Duhamel formula, see Lemma \ref{Lemma, space trace for Duhamel}. The definition of this space is a modification of the $Y_{s,\sigma-1}$ space defined in Section 5.3 in \cite{Hol06}. 

Let  $\Omega _T := \R^+ \times (0,T)$ for  given $T>0$. We define a restricted
version of the space $X_{s, b}$ to the domain  $\Omega _T $
as 
$X_{s,b}^{\alpha, \beta}  (\Omega _T ) := \left. X_{s,b}^{\alpha, \beta} \right |_{ \Omega _T}$
with the quotient norm
\[ \| u\| _{X_{s,b}^{\alpha, \beta} (\Omega _T )}=\inf _{w\in X_{s, b}^{\alpha, \beta}}  \{ \|w\|_{X_{s, b}^{\alpha, \beta}}: w = u \,\text{ on }\, \Omega  _T\}. \]
The spaces $X_{s,b, \sigma}^{\alpha, \beta} (\Omega _T),$  $Y_{s,b, \sigma}^{\alpha, \beta} (\Omega _T) $ and $Z_{s,b, \sigma}^{\alpha, \beta -1} ( \Omega _T)$ are defined similarly.

For those nonlinear terms in (\ref{ckdv-2}), we denote
\[ F(u,v):= -3uu_x-(uv)_x+vv_x, \qquad G(u,v):= uu_x-(uv)_x-3vv_x,\]  
then the IBVP (\ref{ckdv-2}) is converted to the following equations
\begin{equation} \label{ckdv-3}
\left \{ \begin{array}{ll} u(x,t)= W^{1, \beta}_{\R^+}  (t) \phi  (x) + \int ^t_0 W^{1, \beta}_{\R^+} (t-\tau ) F(u,v) (\tau ) \,d\tau + W^{1, \beta}_{bdr} \left [ h_1 \right ] (x,t), \\ \quad \\
	v(x,t)= W^{-1, -\beta}_{\R^+ }(t) \psi  (x) + \int ^t_0 W^{-1, -\beta}_{\R ^+ } (t-\tau ) G(u,v) (\tau ) \,d\tau  + W^{-1, -\beta}_{bdr} \left [ h_2, h_3 + u_x (0,t) \right ] (x,t), \end{array} 
\right. \end{equation}
where $ x,t>0 $, $W_{\R^+}^{1,\beta} (t)$ and $W_{\R^+}^{-1, -\beta} (t)$  are  the  $C^0$ semigroups  associated to  the IBVPs of the linear KdV equations posed on $\R ^+$:
\be\label{linear eq-r}
\left\{\begin{array}{ll}
w_t+w_{xxx}+\beta w_x=0,  \\
w|_{x=0}=0, \quad w |_{t=0}=\phi(x) \in H^{s}(\m{R^+}),
\end{array}
\qquad x,t>0,\right.\ee
and
\be\label{linear eq-l}
\left\{\begin{array}{ll}
w_t- w_{xxx} - \beta w_x=0, \\
w|_{x=0}=0, \quad  w_x|_{x=0} = 0, \quad w |_{t=0}=\psi(x) \in H^{s}(\m{R^+}),
\end{array}  
\qquad x,t>0,\right.\ee
respectively,  
$W^{1, \beta}_{bdr} (t)$ and $W^{-1,-\beta}_{bdr} (t)$ are the boundary integral operators associated with he IBVPs of the linear KdV equation posed on $\R^+$:
\be\label{lin-u-1}\begin{cases}
u_t+u_{xxx}+\b u_{x}=0,\\
u(x,0)=0,\\
u(0,t)=h_1(t),
\end{cases}   \qquad x,t>0,\ee
and 
\be\label{lin-v-1}\begin{cases}
v_t-v_{xxx}-\b v_{x}=0,\\
v(x,0)=0,\\
v(0,t)=h_2(t),\quad v_x(0,t)=h_3(t),
\end{cases} \qquad x,t>0,\ee
respectively. Here for simplicity, we  have assumed that $\phi (0)= h_1 (0)= \psi (0) = h_2 (0) = 0$ if $\frac12 < s< \frac32$ and $h_3(0)= \phi'(0)= \psi' (0) =0 $ if  $\frac32< s\leq 3$. 
We will first  establish the following conditional well-posedness.

\begin{thm}[Conditional Well-posedness]  \label{Thm, cwp}

Let $-\frac34 < s\leq 3$, $T>0$  and $0\leq  \beta \leq 1$ be given.  There exist $ \sigma = \sigma(s)>\frac12 $ and $r=r(s,T)>0$  such that   for  any naturally compatible  $(\phi, \psi) \in {\cal H}^s _x(\R^+)$,  
$\vec{h}= (h_1, h_2, h_3 ) \in  {\cal H}^s_t (\R^+)  $  related to the IBVP (\ref{ckdv-2})  with
\[ \|\phi , \psi )\|_{{\cal H}^s_x (\R^+)} + \|\vec{h} \|_{{\cal H}^s_t (\R^+)}  \le r , \]
the system of the integral equations (SIE)  (\ref{ckdv-3}) admits a unique solution  
\[ (u,v) \in \mcal{Y}_{\sigma} (\Omega _T) :=Y^{1,\b}_{s,\frac12,\sigma} (\Omega _T) \times Y^{-1,-\b}_{s,\frac12,\sigma}(\Omega _T)\]
Moreover, the solution map is real analytic    in the corresponding spaces.
\end{thm}

\begin{rem}\label{Re, uwp}
The well-posedness  presented in Theorem \ref{Thm, cwp}   is  conditional in the sense  that 
the solution $(u,v)\in C([0,T]; \H^s_x (\R^+))$   of   (\ref{ckdv-3})    is in fact  the  restriction for  a function defined on   $\R\times \R$ to  the region $(0, T)\times \R^+$, 
and   the uniqueness  of the solution  holds in a subspace $\mcal{Y}_{\sigma} (\Omega _T)$ rather than  in the full space $ C([0,T];  \H^s_x (\R^+))$, 
which  leads to a  serious issue: 
If  a solution of the   IBVP \eqref{ckdv-2}  is found in a different approach, will it be the same as  that presented by  Theorem \ref{Thm, cwp}?
\end{rem}
	
In this paper, we will prove Theorem \ref{Thm, main-scaling} to give a positive answer to the above question. The justification of Theorem \ref{Thm, main-scaling} follows from Theorem \ref{Thm, cwp} and the following two further properties.
\begin{itemize}
\item[(i)]  the solution of SIE (\ref{ckdv-3}) is a mild solution of  the IBVP (\ref{ckdv-2})
\item[(ii)]  for given initial data $(\phi, \psi)$ and  the boundary data $\vec{h}$, the IBVP (\ref{ckdv-2})  admits at most one mild solution.
\end{itemize}
More precise statements about these two properties are given below in Theorem \ref{Thm, exist of ms} and Theorem \ref{Thm, uniq of ms}.

\begin{thm}[Existence of the mild solution]	\label{Thm, exist of ms}
Fix $s\in(-\frac34,3]$ and $r>0$. Then there exist $\sigma=\sigma(s)>\frac12$ and $T=T(s,r)>0$ such that for any compatible data $(p,q)\in \mcal{H}_{x}^{s}(\m{R}^{+})$ and $(a,b,c)\in \mcal{H}_{t}^{s}(\m{R}^{+})$  with 
\[\|(p,q)\|_{\mcal{H}^s_1(\m{R}^+)}+\|(a,b,c)\|_{\mcal{H}^s_2(\m{R}^+)}\leq r,\]
the IBVP (\ref{ckdv-1}) admits a mild solution 
$(u,v)\in C(0,T; H^s(\R^+))\times C(0,T; H^s(\R^+))$
such that
$u\in Y^{1,1}_{s,\frac12,\sigma}(\Omega _T)$,  
$ v\in Y^{-1,-1}_{s,\frac12,\sigma}(\Omega _T)$, and
\[\|(u,v)\|_{{\mathcal{Y}}_s}\leq  \a_{r,s} \l( \|(p,q)\|_{\mcal{H}^s_x(\m{R}^+)}+\|(a,b,c)\|_{\mcal{H}^s_t(\m{R}^+)}\r),\]
where $\a$ is a continuous function from $\R^+$ to $\R^+$ depending only on $r$ and $s$.
\end{thm}

\begin{thm}[Uniqueness of the mild solution] \label{Thm, uniq of ms}
For $s\in (-\frac34,3]$, $(p,q)\in \mcal{H}_{x}^{s}(\m{R}^{+})$ and $(a,b,c)\in \mcal{H}_{t}^{s}(\m{R}^{+})$, the IBVP \eqref{ckdv-1} admits at most one mild solution $ (u,v) $ in the sense of Definition \ref{Def, m-s}.
\end{thm}

Theorem  \ref{Thm, cwp} will be justified by the standard contraction mapping principle.  We   first  need to  investigate the  following two  forced linear IBVPs   whose compatibility will be introduced in Definition \ref{Def, compatibility-lin}.
\begin{equation}\label{lin-u}
\begin{cases}
u_t+u_{xxx}+\b u_{x}=f,\\
u(x,0)=p(x), \\
u(0,t)=a(t),
\end{cases} \qquad x,t>0,
\end{equation}
and
\be\label{lin-v}\begin{cases}
v_t-v_{xxx}-\b v_{x}=g,\\
v(x,0)=q(x),\\
v(0,t)=b_1(t),\quad v_x(0,t)=b_2(t).
\end{cases} \qquad x,t>0.\ee

\begin{defi}\label{Def, compatibility-lin}
Let $s\in\big(-\frac34,3\big]$, $(p,q)\in \mcal{H}_{x}^{s}(\m{R}^{+})$ and $(a,b_1,b_2)\in \mcal{H}_{t}^{s}(\m{R}^{+})$. Then 
\begin{itemize}
\item the data $(p,a)$ is said to be compatible for (\ref{lin-u}) if $p(0)=a(0)$ when $s>\frac12$;

\item the data $(q,b_1,b_2)$ is said to be compatible for (\ref{lin-v-1}) if they satisfy $q(0)=b_1(0)$ when $s>\frac12$ and further satisfy $q'(0)=b_2(0)$ when $s>\frac32$. The compatibility for (\ref{lin-v}) is defined similarly.
\end{itemize}
\end{defi}
The following  two  estimates for their solutions are key tools in proving Theorem \ref{Thm, cwp}. For ease of notation, we denote $ \m{R}^+_0 = \m{R}\cup\{0\} = [0,\infty) $.
\begin{pro}\label{Prop, pos-kdv}
Let $s\in(-\frac34,3]$, $T>0$ and $0<\b\leq 1$. Assume $p\in H^{s}(\m{R}^+)$ and $a\in H^{\frac{s+1}{3}}(\m{R}^+)$ are compatible for (\ref{lin-u}). Then there exists $\sigma_1=\sigma_1(s)>\frac12$ such that for any $\sigma\in \big(\frac12,\sigma_1\big]$ and for any $f\in X^{1,\b}_{s,\sigma-1}\bigcap Z^{1,\b}_{s,\sigma-1}$, the equation (\ref{lin-u}) has a solution up to time $T$. More precisely, there exists a function 
\be\label{Gamma+}
\wt{u}:=\Gamma^{+}_{\b}(f,p,a),\ee
defined on $\m{R}\times\m{R}$, belongs to $Y_{s,\frac12,\sigma}^{1,\b}\bigcap C_{x}^{j}\big(\m{R}^+_0; H_{t}^{\frac{s+1-j}{3}}(\m{R})\big)$ for $j=0,1$, and its restriction $\wt{u}|_{\m{R}^+_0\times[0,T]}$ solves (\ref{lin-u}) on $\m{R}^+_0\times[0,T]$. In addition, $\wt{u}$ satisfies the following estimates with $C=C(s,\sigma)$.
\begin{align}
\big\| \wt{u}\big\|_{Y^{1,\b}_{s,\frac12,\sigma}} &\leq C\Big(\|f\|_{X^{1,\b}_{s,\sigma-1}}+\|f\|_{Z^{1,\b}_{s,\sigma-1}}+\|p\|_{H^{s}}+\|a\|_{H^{\frac{s+1}{3}}}\Big),\label{est for pos lin kdv} \\
\sup_{x\geq 0}\big\|\p_x^j \wt{u}\big\|_{H^{\frac{s+1-j}{3}}_t(\m{R})} &\leq C\Big(\|f\|_{X^{1,\b}_{s,\sigma-1}}+\|f\|_{Z^{1,\b}_{s,\sigma-1}}+\|p\|_{H^{s}}+\|a\|_{H^{\frac{s+1}{3}}}\Big), \quad\,j=0,1. \label{trace est for pos lin kdv}
\end{align}
\end{pro}

\begin{pro}\label{Prop, neg-kdv}
Let $s\in(-\frac34,3]$, $T>0$ and $0<\b\leq 1$. Assume $q\in H^{s}(\m{R}^+)$, $b_1\in H^{\frac{s+1}{3}}(\m{R}^+)$ and $b_2\in H^{\frac{s}{3}}(\m{R}^+)$ are compatible for (\ref{lin-v}). Then there exists $\sigma_2=\sigma_2(s)>\frac12$ such that for any $\sigma\in \big(\frac12,\sigma_2\big]$ and for any $g\in X^{-1,-\b}_{s,\sigma}\bigcap Z^{-1,-\b}_{s,\sigma}$, the equation (\ref{lin-v}) has a solution up to time T. More precisely, there exists a function 
\be\label{Gamma-}
\wt{v}:=\Gamma^{-}_{\b}(g,q,b_1,b_2),\ee
defined on $\m{R}\times\m{R}$, belongs to $Y_{s,\frac12,\sigma}^{-1,-\b}\bigcap C_{x}^{j}\big(\m{R}^+_0; H_{t}^{\frac{s+1-j}{3}}(\m{R})\big)$ for $j=0,1$, and its restriction $\wt{v}|_{\m{R}^+_0\times[0,T]}$ solves (\ref{lin-v}) on $\m{R}^+_0\times[0,T]$. In addition, $\wt{v}$ satisfies the following estimates with $C=C(s,\sigma)$.
\begin{align}
\big\| \wt{v}\big\|_{Y^{-1,-\b}_{s,\frac12,\sigma}} &\leq C\Big(\|g\|_{X^{-1,-\b}_{s,\sigma-1}}+\|g\|_{Z^{-1,-\b}_{s,\sigma-1}}+\|q\|_{H^{s}}+\|b_1\|_{H^{\frac{s+1}{3}}}+\|b_2\|_{H^{\frac{s}{3}}}\Big), \label{est for neg lin kdv}\\
\sup_{x\geq 0}\big\|\p_x^j \wt{v}\big\|_{H^{\frac{s+1-j}{3}}_t(\m{R})} 
&\leq C\Big(\|g\|_{X^{-1,-\b}_{s,\sigma-1}}+\|g\|_{Z^{-1,-\b}_{s,\sigma-1}}+\|q\|_{H^{s}}+\|b_1\|_{H^{\frac{s+1}{3}}}+\|b_2\|_{H^{\frac{s}{3}}}\Big),\quad\,j=0,1. \label{trace est for neg lin kdv}
\end{align}
\end{pro}

Proposition \ref{Prop, pos-kdv} and Proposition \ref{Prop, neg-kdv} will be proved in Section \ref{Sec, lin pb} and they will be used to handle linear estimates in the proof of Theorem \ref{Thm, cwp}, while for the nonlinear part, the following bilinear estimates will be the key ingredient.
\begin{pro}\label{Prop, bilin}
Let $-\frac34<s\leq 3$, $\a\neq 0$ and $|\b|\leq 1$. Then there exists $\sigma_0=\sigma_0(s,\a)>\frac12$ such that for any $\sigma\in(\frac12,\sigma_0]$, the following bilinear estimates hold for any $w_1,w_2$ with $C=C(s,\a,\sigma)$.
\begin{eqnarray*}
\|\p_{x}(w_1 w_2)\|_{X^{\a,\b}_{s,\sigma-1}}+\|\p_{x}(w_1 w_2)\|_{Z^{\a,\b}_{s,\sigma-1}} 
&\leq & C \| w_1\|_{X^{\a,\b}_{s,\frac12,\sigma}}\| w_2\|_{X^{\a,\b}_{s,\frac12,\sigma}},\\
\|\p_{x}(w_1 w_2)\|_{X^{-\a,-\b}_{s,\sigma-1}}+\|\p_{x}(w_1 w_2)\|_{Z^{-\a,-\b}_{s,\sigma-1}}
&\leq & C \| w_1\|_{X^{\a,\b}_{s,\frac12,\sigma}}\| w_2\|_{X^{\a,\b}_{s,\frac12,\sigma}},\\
\|\p_{x}(w_1 w_2)\|_{X^{ \a,\b}_{s,\sigma-1}}+\|\p_{x}(w_1 w_2)\|_{Z^{ \a,\b}_{s,\sigma-1}}
&\leq & C \| w_1\|_{X^{\a,\b}_{s,\frac12,\sigma}}\| w_2\|_{X^{-\a,-\b}_{s,\frac12,\sigma}}.
\end{eqnarray*}
\end{pro}

The verification of Proposition \ref{Prop, bilin} will be presented in Section \ref{Sec, bilin est}. Then one can take advantage of Proposition \ref{Prop, pos-kdv}--Proposition \ref{Prop, bilin} to justify Theorem \ref{Thm, cwp}. Based on Theorem \ref{Thm, cwp}, we will establish the persistence of regularity property in Proposition \ref{Prop, Tmax} which is the key ingredient in the proof of Theorem \ref{Thm, exist of ms} and Theorem \ref{Thm, uniq of ms}.

\subsection{Application to general KdV-KdV systems}

We point out that the  approach developed in this paper  for  the  IBVP (\ref{ckdv}) can also be applied to study the  IBVPs of  general coupled KdV-KdV  systems (\ref{ckdv-general}):
\be\label{ckdv-general}
\bp u_t\\v_t  \ep + A_{1}\bp u_{xxx}\\v _{xxx}\ep + A_{2}\bp u_x\\v_x \ep = A_{3}\bp uu_x\\vv_x \ep + A_{4}\bp u_{x}v\\uv_{x}\ep, 
\ee
where $\{A_{i}\}_{1\leq i\leq 4}$ are $2\times 2$ real constant matrices and $ A_1 $ is diagonalizable: $A_{1}=M\bp a_{1}& 0\\0& a_{2}\ep M^{-1}$ with  $a_{1}a_2 \ne 0$. By regarding $M^{-1}\bp u\\v\ep$ as the new unknown functions (still denoted by $u$ and $v$),  the third order terms $ u_{xxx} $ and $ v_{xxx} $ in system (\ref{ckdv-general}) can be decoupled such that
\be\label{ckdv, coef form}
\left\{\begin{array}{rcl}
u_t+a_{1}u_{xxx}+b_{11}u_x &=& -b_{12}v_x+c_{11}uu_x+c_{12}vv_x+d_{11}u_{x}v+d_{12}uv_{x},\vspace{0.03in}\\
v_t+a_{2}v_{xxx}+b_{22}v_x &=& -b_{21}u_x+c_{21}uu_x+c_{22}vv_x +d_{21}u_{x}v+d_{22}uv_{x},
\end{array}\right.\ee
where the dispersion coefficients $ a_1 $ and $ a_2 $ are the nonzero eigenvalues of the matrix $ A_1 $. Systems (\ref{ckdv-general}) include many practical models, such as the well-known Majda-Biello system\cite{MB03}, the Hirota-Satsuma system \cite{HS81} and the Gear-Grimshaw system \cite{GG84}. 
%
%
The Cauchy problem for general systems (\ref{ckdv-general}) posed either on $\mathbb{R}$ or on the torus $\mathbb{T}$  has been well studied in the literature for its well-posedness in the space $H^s (\mathbb{R})\times H^s (\mathbb{R})$ or $H^s (\mathbb{T})\times H^s (\mathbb{T})$. The interested readers are referred to \cite{YZ22a,YZ22b}  and the references therein for an overall review of this study.  For the IBVPs of systems (\ref{ckdv-general}) posed on $\R^+$,  we can justify the well-posedness results  similar to  that presented in Theorem \ref{Thm, main} for the IBVP (\ref{ckdv}) using the same approach. 

In the following, we list two such results for the IBVPs of the Gear-Grimshaw system (\ref{G-G system}) on $ \R^+ $:
\be\label{G-G system}
\left\{\begin{array}{rcl}
u_{t}+u_{xxx}+\sigma_{3}v_{xxx} &=&-uu_{x}+\sigma_{1}vv_{x}+\sigma_{2}(uv)_{x},\vspace{0.03in}\\
\rho_{1}v_{t}+\rho_{2}\sigma_{3}u_{xxx}+v_{xxx}+\sigma_{4}v_{x} &=& \rho_{2}\sigma_{2}uu_{x}-vv_{x}+\rho_{2}\sigma_{1}(uv)_{x}, 
\end{array}\right.\ee
where  $\sigma_{i}\in\m{R}(1\leq i\leq 4)$ and $\rho_{1},\,\rho_{2}>0$. This system can be written in the format of (\ref{ckdv-general}) with 
\[A_1 = \bp 1& \sigma_{3}\\ \frac{\rho_{2}\sigma_{3}}{\rho_{1}} & \frac{1}{\rho_{1}} \ep, \quad A_2 =  \bp 0 & 0\\ 0 & \frac{\sigma_4}{\rho_{1}} \ep, \quad A_3 = \bp -1 & \sigma_{1}\\ \frac{\rho_{2}\sigma_{2}}{\rho_{1}} & -\frac{1}{\rho_{1}} \ep, \quad A_4 = \bp \sigma_2 & \sigma_2 \\ \frac{\rho_{2}\sigma_{1}}{\rho_{1}} & \frac{\rho_{2}\sigma_{1}}{\rho_{1}} \ep.\]
Note that when $\rho _2 \sigma _3^2 \ne 1$, $ A_1 $ possesses two  nonzero real eigenvalues $ \lam_1 $ and $ \lam_2 $ with $ \lam_1\leq \lam_2 $. Moreover, $\lambda _1 < 0 < \lambda _2$ if  $\rho _2 \sigma _3^2 >1$;
$0<  \lambda _1 < \lambda _2$ if $ \rho _2 \sigma _3^2 <1$; and 
$\lambda _1=\lambda _2 =1$ if $ \sigma _3 =0$ and $\rho _1 =1$. Thus, when $ \rho _2 \sigma _3^2 >1$,  we need to impose three boundary conditions as that in the case of the IBVP (\ref{ckdv}),
\be\label{G-G system-1}
\left\{\begin{array}{rcll}
u_{t}+u_{xxx}+\sigma_{3}v_{xxx} &=&-uu_{x}+\sigma_{1}vv_{x}+\sigma_{2}(uv)_{x}, \vspace{0.03in}  & x, \ t >0,\\
\rho_{1}v_{t}+\rho_{2}\sigma_{3}u_{xxx}+v_{xxx}+\sigma_{4}v_{x} &=& \rho_{2}\sigma_{2}uu_{x}-vv_{x}+\rho_{2}\sigma_{1}(uv)_{x}, \vspace{0.03in}  & x,\ t >0,\\ 
u(x,0)=\phi(x), & \  & v(x,0)= \psi (x), &x>0, \\ 
u(0,t)= h_1 (t), &\ & v(0,t)= h_2 (t), \ v_x (0,t)= h_3 (t) , \ & t>0.
\end{array}\right.\ee
But in the case of $ \rho _2 \sigma _3^2 < 1$, we only need to impose two boundary conditions since $ \lam_1 $ and $ \lam_2 $ have the same sign.
\be\label{G-G system-2}
\left\{\begin{array}{rcll}
u_{t}+u_{xxx}+\sigma_{3}v_{xxx} &=&-uu_{x}+\sigma_{1}vv_{x}+\sigma_{2}(uv)_{x}, \vspace{0.03in}  & x, \ t >0,\\
\rho_{1}v_{t}+\rho_{2}\sigma_{3}u_{xxx}+v_{xxx}+\sigma_{4}v_{x} &=& \rho_{2}\sigma_{2}uu_{x}-vv_{x}+\rho_{2}\sigma_{1}(uv)_{x}, \vspace{0.03in}  & x,\ t >0, \\ 
u(x,0)=\phi(x), & \  & v(x,0)= \psi (x), &x>0,\\ 
u(0,t)= h_1 (t),  &\  &\  v (0,t)= h_2 (t) , \ & t>0.
\end{array}\right.\ee

In Theorem 1.6 in \cite{YZ22a}, the authors obtained a complete list of well-posedness results for the Cauchy problem of the Gear-Grimshaw system (\ref{G-G system}) posed on $\R$ for all possibilities of the parameters $ \rho_1 $, $ \rho_2 $ and $ \sigma_i(1\leq i\leq 4) $. Based on this result and the method we developed to prove Theorem \ref{Thm, main}, we can also establish the following well-posedness results concerning the IBVP of (\ref{G-G system-1}) or (\ref{G-G system-2}) on $ \R^+ $.

\begin{cor}  If $ \rho _1 >0$, $ \rho _2 >0$  and $ \rho _2 \sigma _3^2 >1 $,  then the IBVP (\ref{G-G system-1}) is locally analytically well-posed in the space ${\cal H}^s (\R^+)$, with any naturally compatible data $(\phi, \psi) \in {\cal H}^s _x(\R^+)$ and $\vec{h}= (h_1, h_2, h_3 ) \in  {\cal H}^s_t (\R^+) $,
for any $s> -\frac34$.
\end{cor}

\begin{cor} If $ \rho _1 >0 $, $ \rho _2 >0 $ and $\rho _2 \sigma _3^2 <1$, then the IBVP (\ref{G-G system-2}) is locally analytically well-posed in the space ${\cal H}^s (\R^+)$, with any naturally compatible data $ (\phi, \psi ) \in {\cal H}^s_x (\R^+)$, $(h_1 , h_2 )\in H^{\frac{s+1}{3}}(\R^+)\times H^{\frac{s+1}{3}}(\R^+)$, for any
\begin{enumerate}[(i)]
\item $s>-\frac{3}{4}$ if $\sigma_{3}=0$ and $\rho_{1}=1$;
\item $s\geq 0$  if 
$\rho_{2}\sigma_{3}^{2} <1$ but (\ref{assum-1}) below is not satisfied;
\item $s\geq\frac34 $  if
\begin{equation} \label{assum-1}   \rho_{2}\sigma_{3}^{2}\leq  \frac{9}{25}, \quad  \rho _1= \frac12  \left ( \alpha \pm\sqrt{\alpha ^2 -4 } \right )
\end{equation} 
with
\[ \alpha = \frac{17-25\rho _2 \sigma ^2_3}{4}.\]
\end{enumerate}
\end{cor}

%

Finally, we remark that the current paper only studies the 1D Boussinesq systems, however, the 2D Boussinesq systems are also of great importance and has attracted significant attention, see e.g. \cite{BCL05, BLS08, DMS07, MSZ12, SWX17, LPS12}. These earlier works mainly focus on the well-posedness of the 2D Boussinesq systems on the space-time region $ \m{R}^{2}\times (0,T) $, it might be an interesting future task to study their well-posedness on $ \R^+ \times \R^+$.

\subsection{Outline of the remaining paper}
\begin{itemize}
\item In Section \ref{Sec, pre}, we will present some results for the Cauchy problem of the KdV equation and the reversed KdV equation  posed on the whole line $\R$ which  will play an important role  later in dealing with the initial  boundary value problems of the coupled KdV systems.

\item In Section \ref{Sec, lin pb},  we investigate the linear IBVPs (\ref{lin-u}) and (\ref{lin-v}) to carry out the proofs of Proposition \ref{Prop, pos-kdv} and Proposition \ref{Prop, neg-kdv}.

\item In Section \ref{Sec, bilin est},  various bilinear estimates  presented in Proposition \ref{Prop, bilin}   will be established. 

\item Section \ref{Sec, wp} is devoted to the proof of Theorem \ref{Thm, main-scaling}, the  main result of this paper. Equivalently, we will establish Theorems \ref{Thm, cwp}, \ref{Thm, exist of ms} and \ref{Thm, uniq of ms} in Section \ref{Sec, wp}.

\end{itemize}

\section{Preliminaries}
\label{Sec, pre}

Throughout this paper, $\eta(t)$ is fixed to be a bump function in $C_{0}^{\infty}(\m{R})$ which satisfies $ 0\leq \eta \leq 1 $ and
\be\label{def of eta}
\eta (t)=\left\{\begin{array}{lll}
1 & \text{if} & |t|\leq 1,\\
0 & \text{if} & |t|\geq 2.
\end{array}
\right.\ee
The notation $a+$ denotes $a+\eps$ for any $\eps>0$. Similarly, $a-$ denotes $a-\eps$ for any $\eps>0$. For a function $f$ in both $x$ and $t$ variables, we use $\mcal{F}f$ (or $\wh{f}$) to denote its space-time Fourier transform. Meanwhile, we use $\mcal{F}_{x}f$ or ${\wh{f}}^x$ to denote its Fourier transform with respect to $x$, and use $\mcal{F}_{t}f$ or ${\wh{f}}^t$ to denote its Fourier transform with respect to $t$. If a function $h$ is only in $x$ variable, then we use ${\wh{h}}^x$ (or just $\wh{h}$ when there is no ambiguity) to denote its Fourier transform. Similar notations are applied to functions which are only in $t$ variable.
For any $s\in\m{R}$, we denote $E_s$ to be an extension operator from $H^{s}(\m{R}^+)$ to $H^{s}(\m{R})$ such that 
\be\label{ext for p}\|E_sf\|_{H^{s}(\m{R})}\leq C\|f\|_{H^{s}(\m{R}^+)},\quad\forall\,f\in H^{s}(\m{R}^+),
\ee
where $C$ is a constant which only depends on $s$. 



\begin{lem}\label{Lemma, 0 ext est}
Let $h\in H^{\frac{s+1}{3}}(\m{R}^+)$ and 
\be\label{zero ext}
h^*(x) :=\left\{\begin{array}{lll}
h(x) & \text{if} & x\geq 0,\\
0 & \text{if} & x<0.
\end{array}
\right.\ee
Assume either  $-\frac52<s<\frac12$, or $\frac12<s<\frac72$ and $h(0)=0$. Then $ h^*\in H^{\frac{s+1}{3}}(\m{R})$ and there exists $C=C(s)$ such that 
\be\label{0 ext est}
\| h^*\|_{H^{\frac{s+1}{3}}(\m{R})}\leq C \| h\|_{H^{\frac{s+1}{3}}(\m{R}^+)}.\ee
\end{lem}
The result in Lemma \ref{Lemma, 0 ext est} is standard, see e.g. Lemma 2.1(i)(ii) in \cite{ET16} or \cite{JK95}.  

For any $\b>0$, we define the function $P_{\b}:\m{R}\to\m{R}$ as 
\be\label{def of P_b}
P_{\b}(\mu) =\mu^3-\b\mu.\ee

\begin{lem}\label{Lemma, est on H^s norm}
Let $s\in\m{R}$ and $0<\b\leq 1$. If a function $M:\m{R}\rightarrow\m{C}$ satisfies
\be\label{bdd for m}
|M(\mu)|\leq C |(P_{\b})'(\mu)|^{\frac12}\la\mu\ra^{s+1},\quad\forall\,\mu\in\m{R},\ee
for some constant $C=C(s)$. Then there exists $C_1=C_1(s)$ such that
\be\label{est on H^s norm}
\l\| M(\mu)\wh{f}\big(P_{\b}(\mu)\big)\r\|_{L^2(\m{R})}\leq C_1 \| f\|_{H^{\frac{s+1}{3}}(\m{R})}, \quad \forall\,f\in H^{\frac{s+1}{3}}(\m{R}).\ee
\end{lem}
\begin{proof}
First, it follows from (\ref{bdd for m}) that
\begin{align*}
\l\| M(\mu)\wh{f}\big(P_{\b}(\mu)\big)\r\|_{L^2}^2 &= \i_{\m{R}}|M(\mu)|^2\big|\wh{f}(P_{\b}(\mu))\big|^2\,d\mu
\ls \i_{\m{R}}|(P_{\b})'(\mu)|\la\mu\ra^{2(s+1)}\big|\wh{f}(P_{\b}(\mu))\big|^2\,d\mu.
\end{align*}
Then applying the change of variable $\xi=P_{\b}(\mu)=\mu^3-\b\mu$ yields  (\ref{est on H^s norm}).
\end{proof}

Based on Lemma \ref{Lemma, 0 ext est} and Lemma \ref{Lemma, est on H^s norm}, we immediately obtain the following conclusion.
\begin{cor}\label{Cor, mul est}
Let $a\in H^{\frac{s+1}{3}}(\m{R}^+)$ and denote $a^*$ to be the zero extension of $a$ as defined in (\ref{zero ext}). Assume either $-\frac52<s<\frac12$, or $\frac12<s<\frac72$ and $a(0)=0$. 
Let $M:\m{R}\rightarrow\m{C}$ be a function that satisfies (\ref{bdd for m}). Then there exists $C=C(s)$ such that
\[\l\| M(\mu)\wh{a^*}(P_{\b}(\mu))\r\|_{L^2(\m{R})}\leq C\|a\|_{H^{\frac{s+1}{3}}(\m{R}^+)},\quad\forall\, a\in H^{\frac{s+1}{3}}(\m{R}^+).\]
\end{cor}

Next, we present some results about the Cauchy problem for the KdV equation and the reversed KdV equation  posed on the whole line $\R$ which  will play important roles  later in dealing with the initial  boundary value problems of the coupled KdV system.
\begin{lem}\label{Lemma, lin est}
Let $s\in\m{R}$, $\a\in\m{R}\backslash\{0\}$, $|\b|\leq 1$ and $\frac12<\sigma\leq 1$. Then for any $w_0\in H^s(\m{R})$ and $F\in X^{\a,\b}_{s,\sigma-1}$, there exists $C=C(s,\a,\sigma)$  such that 
\be\label{est for free KdV}
||\eta(t)W_{R}^{\a,\b}(t)w_{0}||_{Y^{\a,\b}_{s,\frac12,\sigma}}\leq C|| w_{0}||_{H^{s}(\m{R})}\ee
and
\be\label{est for Duhamel term}
\Big\Vert \eta(t)\int_{0}^{t}W_{R}^{\a,\b}(t-\tau)F(\cdot,\tau)\,d\tau \Big\Vert_{Y^{\a,\b}_{s,\frac12,\sigma}}\leq C\| F\|_{X^{\a,\b}_{s,\sigma-1}}.\ee
\end{lem}

\begin{proof}
Recalling (\ref{MFR, ub}), the norm $\|\cdot\|_{Y^{\a,\b}_{s,\frac12,\sigma}}$ is bounded by $\|\cdot\|_{X^{\a,\b}_{s,\sigma}}$, so it suffices to prove 
\be\label{est for free KdV, classical}
||\eta(t)W_{R}^{\a,\b}(t)w_{0}||_{X^{\a,\b}_{s,\sigma}}\leq C|| w_{0}||_{H^{s}(\m{R})}\ee
and
\be\label{est for Duhamel term, classical}
\Big\Vert \eta(t)\int_{0}^{t}W_{R}^{\a,\b}(t-\tau)F(\cdot,\tau)\,d\tau \Big\Vert_{X^{\a,\b}_{s,\sigma}}\leq C\| F\|_{X^{\a,\b}_{s,\sigma-1}}.\ee
The proofs for (\ref{est for free KdV, classical}) and (\ref{est for Duhamel term, classical}) follow directly from Lemma 3.1 and Lemma 3.3 in \cite{KPV93Duke}.
\end{proof}

\begin{lem}\label{Lemma, space trace for free kdv}
Let $s\in\m{R}$, $\a\in\m{R}\backslash\{0\}$ and $|\b|\leq 1$. Then for any $w_0\in H^s(\m{R})$, $\eta(t)\big[W_{R}^{\a,\b}(t)w_0\big]$ belongs to  $C_{x}^{j}\big(\m{R}; H_{t}^{\frac{s+1-j}{3}}(\m{R})\big)$, where $ j=0,1 $. In addition, there exists $C=C(s,\a)$ such that 
\[\sup_{x\in \R} \big\|\eta(t) \p_{x}^{j}\big[W_R^{\a,\b} (t)w_0\big]\big\|_{H^{\frac{s+1-j}{3}}_t(\R)} \leq C\|w_0\|_{H^s(\R)},\qquad j=0,1.\]
\end{lem}

\begin{proof}
The proof is standard, see e.g. \cite{KPV91Indiana} or Lemma 5.5 in \cite{Hol06}.
\end{proof} 

\begin{lem}\label{Lemma, space trace for Duhamel}
Let $s\in\m{R}$, $\a\in\m{R}\backslash\{0\}$, $|\b|\leq 1$ and $\frac12<\sigma \leq 1$. Then for any $F\in X^{\a,\b}_{s,\si-1}\cap Z^{\a,\b}_{s,\sigma-1}$, $\eta(t)\int^t_0  W^{\a,\b}_R(t-\tau)F(\cdot, \tau)\,d\tau$ lies in 
$ C_{x}^{j}\big(\m{R}; H_{t}^{\frac{s+1-j}{3}}(\m{R})\big)$ for $j=0,1$. In addition, there exists $C=C(s,\a,\sigma)$ such that 
\be\label{xtrace-linD}
\sup_{x\in \R }\left\| \eta(t)\, \partial_x^j\int^t_0  W^{\a,\b}_R(t-\tau)F(\cdot, \tau)\,d\tau\right\|_{H_t^{\frac{s+1-j}{3}}(\R)} \leq C\big(\|F\|_{X^{\a,\b}_{s,\si-1}}+\|F\|_{Z^{\a,\b}_{s,\sigma-1}}\big), \quad j=0,1. \ee
\end{lem}

Similar results of Lemma \ref{Lemma, space trace for Duhamel} can be found in (\hspace{-0.03in}\cite{Hol06}, Lemma 5.6) or (\hspace{-0.03in}\cite{CK02}, Lemma 5.4 and 5.5), so we omit its proof. But we would like to emphasize that if $ j=0 $ and $-1\leq s\leq 3\sigma-1$, then the term $ \|F\|_{Z^{\a,\b}_{s,\sigma-1}} $ in (\ref{xtrace-linD}) can be dropped, but this term is necessary if $ j=0 $ and $ s>3\sigma-1 $. Similarly, if $ j=1 $ and $0\leq s\leq 3\sigma$, then the term $ \|F\|_{Z^{\a,\b}_{s,\sigma-1}} $ is not needed in (\ref{xtrace-linD}), but this term has to be present if $ j=1 $ and $ s>3\sigma $.

\section{Linear problems}
\label{Sec, lin pb}

In this section, we will study the linear IBVPs (\ref{lin-u}) and (\ref{lin-v}),  present the proofs of Proposition  \ref{Prop, pos-kdv} and Proposition \ref{Prop, neg-kdv}.

\subsection{KdV flow traveling to the right}
\label{Subsec, pos KdV, idea}

Consideration is first given to the linear IBVP (\ref{lin-u}) which is copied below for convenience of readers.
\begin{equation} \label{k-1}
\begin{cases}
u_t+u_{xxx}+\b u_{x}=f,\\
u(x,0)=p(x),\\
u(0,t)=a(t),
\end{cases}   \qquad x,t>0.
\end{equation}
This system describes a linear dispersive wave traveling  to the  right, where $p\in H^{s}(\m{R}^+)$, $a\in H^{\frac{s+1}{3}}(\m{R}^+)$,
and $f$ is a function defined on $\m{R}^2$.  The function space that $f$ lives in will be specified later. In addition, the data $p$ and $a$ are assumed to be compatible as in Definition \ref{Def, compatibility-lin}, that is $p(0)=a(0)$ if $s>\frac12$. The solution $u$ of (\ref{k-1}) can be written as $ u= u_1+u_2 $
with
\be\label{soln to pos-kdv-initial}
u_1=\Phi^{1,\b}_{R}(f,p):=W^{1,\b}_{R}(E_{s}p)+\int_{0}^{t}W^{1,\b}_{R}(t-\tau)f(\cdot,\tau)\,d\tau, \ee
where $E_{s}p$ being the extension of $p$ as \eqref{ext for p}, and 
$ u_2   = W^{1, \beta}_{bdr}(a - q)$, 
where 
$ q(t):= u_1(0,t) $ and $W^{1, \beta} _{bdr} (a)$ denotes the solution operator (also called the boundary integral operator) associated to the following linear IBVP (\ref{pos-kdv-bdr}).
\be\label{pos-kdv-bdr}\begin{cases}
z_t+z_{xxx}+\b z_{x}=0, \\
z(x,0)=0,\\
z(0,t)=a(t),
\end{cases} \qquad x,t>0.
\ee

\begin{lem}\label{Lemma, pos-kdv-initial}
Let $s\in(-\frac34,3]$, $0<\b\leq 1$ and $\frac12<\sigma\leq 1$. For any $p\in H^{s}(\m{R}^+)$ and $f\in X^{1,\b}_{s,\sigma-1}\bigcap Z^{1,\b}_{s,\sigma-1}$, the function $\wt{u_1} := \eta(t) \Phi^{1,\b}_{R}(f,p)$ belongs to 
$ Y_{s,\frac12,\sigma}^{1,\b}\bigcap C_{x}^{j}\big(\m{R}; H_{t}^{\frac{s+1-j}{3}}(\m{R})\big)$ for $j=0,1$. In addition, the following estimates hold with some constant $C=C(s,\sigma)$.
\begin{align}
\big\| \wt{u_1}\big\|_{Y^{1,\b}_{s,\frac12,\sigma}} &\leq  C\big(\|f\|_{X^{1,\b}_{s,\sigma-1}}+\|p\|_{H^{s}(\m{R}^+)}\big),\label{est for pos-kdv-initial} \\
\sup_{x\in\m{R}}\big\|\p_x^{j}\wt{u_1}\big\|_{H^{\frac{s+1-j}{3}}_t(\m{R})} &\leq  C\Big(\|f\|_{X^{1,\b}_{s,\sigma-1}}+\|f\|_{Z^{1,\b}_{s,\sigma-1}}+\|p\|_{H^{s}(\m{R}^+)}\Big),\quad j=0,1. \label{trace est for pos-kdv-initial}
\end{align}
\end{lem}
\begin{proof}
(\ref{est for pos-kdv-initial}) follows from (\ref{ext for p}) and Lemma \ref{Lemma, lin est}. (\ref{trace est for pos-kdv-initial}) follows from   (\ref{ext for p}), Lemma \ref{Lemma, space trace for free kdv} and Lemma \ref{Lemma, space trace for Duhamel}.
\end{proof}

Next, we will provide an explicit formula (in integral form) for the boundary integral operator $W^{1, \beta} _{bdr} (a)$. For any $\b>0$, let  $P_{\b}$  be as  given in (\ref{def of P_b}) and  define the function $R_\b:\m{R}\to\m{C}$ by
\be\label{def of R_b}
R_{\b}(\mu) := \frac{\sqrt{3\mu^2-4\b}}{2} =
\left\{\begin{array}{lll}
i\dfrac{\sqrt{4\b-3\mu^2}}{2} & \mbox{if} & \mu^2\leq \frac{4}{3}\b, \vspace{0.1in}\\
\dfrac{\sqrt{3\mu^2-4\b}}{2} &\mbox{if}& \mu^2>\frac{4}{3}\b. \end{array}\right. \ee

\begin{lem}\label{Lemma, formula for pos-kdv-bdr} The boundary integral operator  $W^{1,\b}_{bdr}$ associated to (\ref{pos-kdv-bdr}) has the following explicit  integral  representation
\be\label{formula for pos-kdv-bdr}
W^{1,\b}_{bdr}(a)(x,t)=\frac{1}{\pi}Re\int_{\sqrt{\b}}^{\infty}e^{iP_{\b}(\mu)t}e^{-i\mu x/2}e^{-R_\b(\mu)x}\wh{a^*}[P_{\b}(\mu)]\,dP_{\b}(\mu),\quad x,t\geq 0,\ee
where  $a^*$ is the zero extension of $a$ from $\R^+ $ to $\R$.
\end{lem}

The proof of this lemma is standard based on the Laplace transform method in \cite{BSZ02} and is therefore omitted. 

By taking advantage of formula (\ref{formula for pos-kdv-bdr}), we will extend $W^{1,\b}_{bdr}(a)$ to the whole plane $\m{R}^{2}$ with some desired properties. According to (\ref{def of R_b}), when $\sqrt{\b}\leq \mu\leq \sqrt{\frac{4}{3}\b}$, $R_\b(\mu)$ is a pure imaginary number, so it will not cause much trouble to handle $W^{1,\b}_{bdr}(a)$ even for $x,t<0$. But when $\mu>\sqrt{\frac{4}{3}\b}$, if we do not modify the definition of $W^{1,\b}_{bdr}(a)$ for $x<0$, then $e^{-R_\b(\mu)x}\to +\infty$ as $R_\b(\mu)x\to -\infty$, which makes it difficult to control $W^{1,\b}_{bdr}(a)$. 

To resolve this issue, our first strategy is to multiply a cut-off function (see (\ref{K})) to eliminate the part where $R_\b(\mu)x\leq -1$, this idea comes from \cite{ET16}. Let $\psi \in C^{\infty}(\R)$ such that $ 0\leq \psi\leq 1 $ and
\be\label{def of psi}
\psi(x)=\left\{\begin{array}{lll}
1 & \text{if} & x\geq 0,\\
0& \text{if} & x\leq -1.
\end{array}
\right.\ee
Define $\Phi^{1,\b}_{bdr}(a):\m{R}\times\m{R}\to\m{R}$ by
\begin{align}
\Phi^{1,\b}_{bdr}(a)(x,t)=\frac{1}{\pi}Re\int_{\sqrt{\b}}^{\infty}e^{iP_{\b}(\mu)t}e^{-i\mu x/2}K_{\b}(x,\mu)\wh{a^*}[P_{\b}(\mu)]\,dP_{\b}(\mu),\label{pos-bdr-op}
\end{align}
where 
\be\label{K}
K_{\b}(x,\mu):=\left\{\begin{array}{lll}
e^{-R_{\b}(\mu)x} & \mbox{if} & \sqrt{\b}\leq \mu\leq \sqrt{\frac{4}{3}\b},\\
e^{-R_{\b}(\mu)x}\psi\big(R_{\b}(\mu)x\big) & \mbox{if} & \mu>\sqrt{\frac{4}{3}\b}.
\end{array}\right.\ee
It is easily seen that $\Phi^{1,\b}_{bdr}(a)$ is well-defined on $\m{R}\times\m{R}$. In addition, when $x\geq 0$, $K_{\b}(x,\mu)=e^{-R_{\b}(\mu)x}$, so $\Phi^{1,\b}_{bdr}(a)$ matches $W^{1,\b}_{bdr}(a)$ on $\m{R}^{+}_0\times \m{R}^{+}_0$. Hence, $\Phi^{1,\b}_{bdr}(a)$ is an extension of $W^{1,\b}_{bdr}(a)$ to $\m{R}\times\m{R}$. Moreover, if $s>-\frac12$,  it enjoys some good properties after being localised in time. 

\begin{lem}\label{Lemma, pos-kdv-bdr}
Let $s\in(-\frac12,3]$, $0<\b\leq 1$ and $\frac12<\sigma\leq \min\big\{1,\frac{2s+7}{12}\big\}$. For any $a\in H^{\frac{s+1}{3}}(\m{R}^+)$ that is compatible with (\ref{pos-kdv-bdr}), the function 
\[\wt{u_2} := \eta(t) \Phi^{1,\b}_{bdr}(a)\]
equals $W^{1,\b}_{bdr}(a)$ on $\m{R}^+_0\times [0,1]$ and belongs to $Y_{s,\frac12,\sigma}^{1,\b}\bigcap C_{x}^j\big(\m{R}; H_{t}^{\frac{s+1-j}{3}}(\m{R})\big)$ for $j=0,1$. In addition, the following estimates hold with some constant $C=C(s,\sigma)$. 
\begin{align}
\big\| \wt{u_2}\big\|_{Y^{1,\b}_{s,\frac12,\sigma}} &\leq  C\|a\|_{H^{\frac{s+1}{3}}(\m{R}^+)},\label{est for pos-kdv-bdr} \\
\sup_{x\in\m{R}}\big\|\p^j_x \wt{u_2}\big\|_{H^{\frac{s+1-j}{3}}_t(\m{R})} &\leq  C\|a\|_{H^{\frac{s+1}{3}}(\m{R}^+)}, \quad j=0,1. \label{trace est for pos-kdv-bdr}
\end{align}
\end{lem}

Although the extension $\Phi_{bdr}^{1,\b}$ has a very simple expression, the estimate (\ref{est for pos-kdv-bdr}) fails for $s<-\frac12$, see Remark \ref{Re, -1/2nec}. So in order to deal with the case when $-\frac34<s\leq -\frac12$, we have to use another extension, written as $\Psi_{bdr}^{1,\b}(a)$, for $W_{bdr}^{1,\b}(a)$. In fact, we will combine the extension $\Phi_{bdr}^{1,\b}(a)$ and the construction in \cite{BSZ06} to define $\Psi_{bdr}^{1,\b}(a)$. Since the construction in \cite{BSZ06} is too technical to be written as an explicit expression, we will first state the conclusion and then provide more details later in the proof.
\begin{lem}\label{Lemma, pos-kdv-bdr, low-reg}
Let $-\frac34<s\leq -\frac12$, $0<\b\leq 1$ and $a\in H^{\frac{s+1}{3}}(\m{R}^+)$. Then there  exist $\sigma_1=\sigma_1(s)>\frac12$ and an extension $\Psi_{bdr}^{1,\b}(a)$ of $W_{bdr}^{1,\b}(a)$ such that for any $\sigma\in\big(\frac12,\sigma_1\big]$, the function 
\[\wt{u_2}:=\eta(t)\Psi_{bdr}^{1,\b}(a) \]
equals $W_{bdr}^{1,\b}(a)$ on $\m{R}^+_0\times[0,1]$ and belongs to $Y_{s,\frac12,\sigma}^{1,\b}\bigcap C_{x}^j\big(\m{R}^+_0; H_{t}^{\frac{s+1-j}{3}}(\m{R})\big)$ for $j=0,1$. In addition, the following estimates hold with some constant $C=C(s,\sigma)$. 
\begin{align}
\big\| \wt{u_2}\big\|_{Y^{1,\b}_{s,\frac12,\sigma}} &\leq  C\|a\|_{H^{\frac{s+1}{3}}(\m{R}^+)},\label{est for pos-kdv-bdr, low-reg} \\
\sup_{x\geq 0}\big\|\p^j_x \wt{u_2}\big\|_{H^{\frac{s+1-j}{3}}_t(\m{R})} &\leq  C\|a\|_{H^{\frac{s+1}{3}}(\m{R}^+)}, \quad j=0,1. \label{trace est for pos-kdv-bdr, low-reg}
\end{align}
\end{lem}

\begin{rem}\label{Remark, ext trace deri}
We want to point out a slight difference between Lemma \ref{Lemma, pos-kdv-bdr} and Lemma \ref{Lemma, pos-kdv-bdr, low-reg}. On the one hand, the extension function $\eta(t)\Phi_{bdr}^{1,\b}(a)$ in Lemma \ref{Lemma, pos-kdv-bdr} is shown to live in $C_{x}^j\big(\m{R}; H_{t}^{\frac{s+1-j}{3}}(\m{R})\big)$ and the estimate (\ref{trace est for pos-kdv-bdr}) is valid for any $x\in\m{R}$. On the other hand, the extension function $\eta(t)\Psi_{bdr}^{1,\b}(a)$ in Lemma \ref{Lemma, pos-kdv-bdr, low-reg} is only proved to live in $C_{x}^j\big(\m{R}^+_0; H_{t}^{\frac{s+1-j}{3}}(\m{R})\big)$ and the estimate (\ref{trace est for pos-kdv-bdr, low-reg}) is only justified for $x\geq 0$.
\end{rem}

We will first  carry out the proof of Proposition \ref{Prop, pos-kdv} by assuming  Lemma \ref{Lemma, pos-kdv-bdr} and Lemma \ref{Lemma, pos-kdv-bdr, low-reg} hold. After that, these two lemmas will be verified.

\begin{proof}[{\bf Proof of Proposition  \ref{Prop, pos-kdv} }]
Without loss of generality, we assume $T=1 $. We first deal with the case when $s\in(-\frac12,3]$. Choose $\sigma_1(s)=\min\big\{1,\frac{2s+7}{12}\big\}$ and consider any $\sigma\in(\frac12,\sigma_1(s)]$. 
Based on the operator $ \Phi^{1,\b}_{R} $ in Lemma \ref{Lemma, pos-kdv-initial}
and $ \Phi^{1,\b}_{bdr} $ in Lemma \ref{Lemma, pos-kdv-bdr}, we define 
\[\wt{u}=\Gamma^+_\b(f,p,a):=\eta(t)\Phi^{1,\b}_{R}(f,p)+\eta(t)\Phi^{1,\b}_{bdr}\Big(a-\eta(t)\Phi^{1,\b}_{R}(f,p)\big|_{x=0}\Big).\]
Then $\wt{u}$ is defined on $\m{R}\times\m{R}$ and solves (\ref{lin-u}) on $\m{R}^+_0\times[0,1]$. 
Furthermore, it follows from Lemma \ref{Lemma, pos-kdv-initial} and Lemma \ref{Lemma, pos-kdv-bdr} that $\wt{u}$ belongs to $Y_{s,\frac12,\sigma}^{1,\b}\bigcap C_{x}^j\big(\m{R}^+_0; H_{t}^{\frac{s+1-j}{3}}(\m{R})\big)$ for $j=0,1$. 
Meanwhile, we infer from (\ref{est for pos-kdv-initial}) and (\ref{est for pos-kdv-bdr}) that there exists a constant $C=C(s,\sigma)$ such that
\[\big\|\Gamma^{+}_{\b}(f,p,a)\big\|_{Y^{1,\b}_{s,\frac12,\sigma}}\leq C\Big(\|f\|_{X^{1,\b}_{s,\sigma-1}}+\|p\|_{H^{s}}+\big\|a-\eta(t)\Phi^{1,\b}_{R}(f,p)\big|_{x=0}\big\|_{H_{t}^{\frac{s+1}{3}}}\Big).\]
By (\ref{trace est for pos-kdv-initial}) with $j=0$,
$\big\| \eta(t)\Phi^{1,\b}_{R}(f,p) \big|_{x=0} \big\|_{ H_{t}^{ \frac{s+1}{3} } } \leq  C\big( \|f\|_{X^{1,\b}_{s,\sigma-1}} + \|f\|_{Z^{1,\b}_{s,\sigma-1}} + \|p\|_{H^{s}} \big)$. 
Combining the above two estimates yields (\ref{est for pos lin kdv}). Next, by similar argument and using (\ref{trace est for pos-kdv-initial}) and (\ref{trace est for pos-kdv-bdr}), we can justify (\ref{trace est for pos lin kdv}) as well. 

Then we treat the case when $s\in(-\frac34,-\frac12]$. The proof for this case is almost the same as the above case except we replace $\Phi_{bdr}^{1,\b}$ by $\Psi_{bdr}^{1,\b}$ and replace Lemma \ref{Lemma, pos-kdv-bdr} by Lemma \ref{Lemma, pos-kdv-bdr, low-reg}.   The proof of Proposition  \ref{Prop, pos-kdv} is thus complete. $\Box$
\end{proof}

Now we are ready to deal with Lemma \ref{Lemma, pos-kdv-bdr} and \ref{Lemma, pos-kdv-bdr, low-reg}. 

\begin{proof}[{\bf Proof of Lemma \ref{Lemma, pos-kdv-bdr} }] 
We will first establish estimates (\ref{est for pos-kdv-bdr}) and (\ref{trace est for pos-kdv-bdr}) for $s\in(-\frac12,3]\backslash\{\frac12\}$. Then the corresponding estimates in the case of $ s=\frac12 $ can be deduced by interpolating between the case $ s=0 $ and the case $ s=3 $. Fix $s\in(-\frac12,3]\backslash\{\frac12\}$ and $0<\b\leq 1$. Define $h_\b$ such that $\wh{h_\b}(\mu)=(P_{\b})'(\mu)\wh{a^*}[P_{\b}(\mu)]$. Then 
\begin{align*}
\Phi^{1,\b}_{bdr}(a)(x,t)=\frac{1}{\pi}Re \int_{\sqrt{\b}}^{\infty}e^{iP_{\b}(\mu)t}e^{-i\mu x/2}K_{\b}(x,\mu)\wh{h_\b}(\mu)\,d\mu,
\end{align*}
where $K_\b$ is as defined in (\ref{K}). Since $ |(P_{\b})'(\mu)|=|3\mu^2-\b|\ls \la\mu\ra^2$, then
\[\la\mu\ra^{s}|\wh{h_\b}(\mu)|\ls |(P_{\b})'(\mu)|^{1/2}\la\mu\ra^{s+1}\wh{a^*}[P_{\b}(\mu)].\]
Since $a\in H^{\frac{s+1}{3}}(\m{R}^+)$, it then follows from Corollary \ref{Cor, mul est} that $h_\b\in H^{s}(\m{R})$ and $\|h_\b\|_{H^{s}(\m{R})}\ls \|a\|_{H^{\frac{s+1}{3}}(\m{R}^+)}$. Inspired by the above observation, we define an operator $\mcal{T}_{\b}$ as 
\be\label{T}
[\mcal{T}_{\b}(h)](x,t):=\int_{\sqrt{\b}}^{\infty}e^{iP_{\b}(\mu)t}e^{-i\mu x/2}K_{\b}(x,\mu)\wh{h}(\mu)\,d\mu.\ee
Recalling $Y^{1,\b}_{s,\frac12,\sigma}=X^{1,\b}_{s,\frac12}\cap \Lambda^{1,\b}_{s,\sigma}\cap C_{t}(\m{R},H^s(\m{R}))$, so the proof of Lemma \ref{Lemma, pos-kdv-bdr} reduces to establishing the following estimates  for the operator $\mcal{T}_{\b}$.
\begin{align}
\big\|\eta(t)\mcal{T}_\b(h)\big\|_{X^{1,\b}_{s,\frac12}} &\leq  C\|h\|_{H^s(\m{R})},\label{pos-kdv-bdr, FR} \\
\big\|\eta(t)\mcal{T}_\b(h)\big\|_{\Lambda^{1,\b}_{s,\sigma}} &\leq  C\|h\|_{H^s(\m{R})},\label{pos-kdv-bdr, Lambda} \\
\sup_{t\in\m{R}}\big\|\eta(t)\mcal{T}_\b(h)\big\|_{H^{s}_{x}(\m{R})} &\leq  C\|h\|_{H^s(\m{R})},\label{pos-kdv-bdr, time-trace} \\
\sup_{x\in\m{R}}\big\|\p^j_x \big[\eta(t)\mcal{T}_\b(h)\big]\big\|_{H^{\frac{s+1-j}{3}}_t(\m{R})} &\leq  C\|h\|_{H^s(\m{R})}, \quad j=0,1. \label{pos-kdv-bdr, space-trace}
\end{align}

Before showing those estimates hold,  we introduce some notations. Recalling the formula (\ref{K}), we have  
\[K_\b(x,\mu)=\left\{\begin{array}{lll}
e^{-i\sqrt{4\b-3\mu^2}x/2} & \mbox{if} & \sqrt{\b}\leq \mu\leq \sqrt{\frac{4}{3}\b}, \vspace{0.1in}\\
k\big(R_\b(\mu)x\big) &\mbox{if}& \mu>\sqrt{\frac{4}{3}\b},
\end{array}\right.\]
where 
\be\label{k}
k(y) :=e^{-y}\psi(y),\quad\forall\, y\in\m{R}.\ee
It is easily seen that $k$ is a real-valued Schwarz function on $\m{R}$. We will first justify (\ref{pos-kdv-bdr, space-trace}).

\begin{proof}[{\bf Proof of (\ref{pos-kdv-bdr, space-trace})}]
We first consider the case when $j=0$. For any fixed $x\in\m{R}$, taking the Fourier transform of $ \eta \mcal{T}_\b(h) $ (see (\ref{T}) for the expression of $ \mcal{T}_\b(h) $) with respect to $t$ leads to
\[\mcal{F}_{t}[\eta\mcal{T}_{\b}(h)](x,\tau)=\int_{\sqrt{\b}}^{\infty}\wh{\eta}\big(\tau-P_\b(\mu)\big)e^{-i\mu x/2}K_{\b}(x,\mu)\wh{h}(\mu)\,d\mu.\]
Since $|e^{-i\mu x/2}K_{\b}(x,\mu)|\leq C$ for a universal constant $C$ and 
$\la\tau\ra^{\frac{s+1}{3}} \leq \la\tau-P_\b(\mu)\ra^{\frac{|s+1|}{3}} \la P_\b(\mu)\ra^{\frac{s+1}{3}}$,
\begin{align*}
\la\tau\ra^{\frac{s+1}{3}} \big|\mcal{F}_{t}[\eta\mcal{T}_{\b}(h)](x,\tau)\big|
& \ls \int_{\sqrt{\b}}^{\infty} \la\tau-P_\b(\mu)\ra^{\frac{|s+1|}{3}} \l| \wh{\eta}\big(\tau-P_\b(\mu)\big) \r| \la P_\b(\mu)\ra^{\frac{s+1}{3}} |\wh{h}(\mu)| \,d\mu\\
&= \int_{\sqrt{\b}}^{\infty} \wh{f_s}\big(\tau-P_\b(\mu)\big)\la P_\b(\mu)\ra^{\frac{s+1}{3}}|\wh{h}(\mu)|\,d\mu,
\end{align*}
where $f_s$ is defined such that $\wh{f_s}(\cdot)=\la\cdot\ra^{\frac{|s+1|}{3}}|\wh{\eta}(\cdot)|$. Then by splitting the above integral domain $ [\sqrt{\b}, \infty] $ into $ [\sqrt{\b}, 4] $
and $ [4,\infty] $, we are able to 
%
complete the proof of (\ref{pos-kdv-bdr, space-trace}) for the case $j=0$.

Next, for the case $j=1$, 
\[\mcal{F}_{t}\big(\p_x[\eta\mcal{T}_{\b}(h)]\big)(x,\tau)=\int_{\sqrt{\b}}^{\infty}\wh{\eta}\big(\tau-P_\b(\mu)\big)\p_{x}\big[e^{-i\mu x/2}K_{\b}(x,\mu)\big]\wh{h}(\mu)\,d\mu.\]
Noticing there exists a universal constant $C$ such that
$\l| \p_{x}\big[e^{-i\mu x/2}K_{\b}(x,\mu)\big]\r|\ls C\mu$ for any $\mu\geq 4$, this leads to
\[\big|\mcal{F}_{t}\big(\p_x[\eta\mcal{T}_{\b}(h)])(x,\tau)\big|\ls \int_{\sqrt{\b}}^{\infty}\l|\wh{\eta}\big(\tau-P_\b(\mu)\big)\r|\mu|\wh{h}(\mu)|\,d\mu.\]
Define $h_1$ such that $\wh{h_1}(\mu) = \mb{1}_{\{\mu>\sqrt{\b}\}} \mu |\wh{h}(\mu)|$. Then by following the same argument as that for the case $j=0$, we obtain
\[\sup_{x\in\m{R}}\big\|\p_x \big[\eta\mcal{T}_\b(h)\big]\big\|_{H^{\frac{s}{3}}_t(\m{R})} \ls \|h_1\|_{H^{s-1}(\m{R})}\ls \|h\|_{H^s(\m{R})}. \]
\end{proof}

Next, in order to prove (\ref{pos-kdv-bdr, FR})--(\ref{pos-kdv-bdr, time-trace}), we decompose $\mcal{T}_{\b}(h)$ as $\mcal{T}_\b(h)=\mcal{T}_{\b,1}(h)+\mcal{T}_{\b,2}(h)$, where 
\begin{align}
[\mcal{T}_{\b,1}(h)](x,t) &:=\i_{\sqrt{\b}}^{\sqrt{\frac43 \b}}e^{iP_{\b}(\mu)t}e^{-i(\mu+\sqrt{4\b-3\mu^2})x/2}\,\wh{h}(\mu)\,d\mu, \label{def of T1}\\
[\mcal{T}_{\b,2}(h)](x,t) &:=\i_{\sqrt{\frac43 \b}}^{\infty}e^{iP_{\b}(\mu)t}e^{-i\mu x/2}k\big(R_\b(\mu)x\big)\,\wh{h}(\mu)\,d\mu. \label{def of T2}
\end{align}

\begin{proof}[{\bf Proof of (\ref{pos-kdv-bdr, FR})}]
Firstly, we discuss the estimate for $\mcal{T}_{\b,1}(h)$ and we will actually prove a stronger result:
\be\label{T1 est-FR norm}
\|\eta(t)\mcal{T}_{\b,1}(h)\|_{X^{1,\b}_{3,1}}\ls \| h\|_{H^s(\m{R})}.\ee
Since the integral range for $ \mu $ in $ \mcal{T}_{\b,1}(h) $ is a finite interval $ [\sqrt{\beta}, \sqrt{4\beta/3}] $, the verification of (\ref{T1 est-FR norm}) is very straightforward and therefore is omitted.
Next we discuss the estimate for $\mcal{T}_{\b,2}(h)$ which is further decomposed into two parts: $\mcal{T}_{\b,2}(h)=\mcal{T}_{\b,3}(h)+\mcal{T}_{\b,4}(h)$, where 
\begin{align}
[\mcal{T}_{\b,3}(h)](x,t) &:=\i_{\sqrt{\frac43 \b}}^{4}e^{iP_{\b}(\mu)t}e^{-i\mu x/2}k\big(R_\b(\mu)x\big)\,\wh{h}(\mu)\,d\mu, \label{def of T3} \\
[\mcal{T}_{\b,4}(h)](x,t) &:=\i_{4}^{\infty}e^{iP_{\b}(\mu)t}e^{-i\mu x/2}k\big(R_\b(\mu)x\big)\,\wh{h}(\mu)\,d\mu. \label{def of T4}
\end{align}
The reason of introducing such a decomposition is to make sure that $R_\b(\mu)$ has a positive lower bound in $\mcal{T}_{\b,4}(h)$. For example, with the choice of 4, we have $R_{\b}(\mu)>3$ for any $\mu\geq 4$ since $0<\b\leq 1$. The choice of the cut-off number 4 is flexible, as long as it is away from $\sqrt{\frac43 \b}$. For the estimate on $\mcal{T}_{\b,3}$, we will also prove a stronger result:
\be\label{T3 est-FR norm}
\|\eta(t)\mcal{T}_{\b,3}(h)\|_{X^{1,\b}_{3,1}} \ls \| h\|_{H^{s}(\m{R})}.\ee
Again, since the integral region for $ \mu $ in $ \mcal{T}_{\b,3}(h) $ is a finite interval $ [\sqrt{4\beta}/3, 4] $, the inequality (\ref{T3 est-FR norm}) can be built without much difficulty. 

So the focus is turned to the estimate on $\|\eta(t)\mcal{T}_{\b,4}(h)\|_{X^{1,\b}_{s,\frac12}}$. 
By direct computation, 
\[ \mcal{F}[\eta(t)\mcal{T}_{\b,4}(h)](\xi,\tau) = \int_{4}^{\infty} \wh{\eta}\big(\tau-P_{\b}(\mu)\big) \frac{1}{R_\b(\mu)} \,\wh{k}\bigg(\frac{\xi+\mu/2}{R_\b(\mu)}\bigg) \wh{h}(\mu)\,d\mu.\]
Since $k$ is a Schwarz function, so is $\wh{k}$, hence, 
\[ \l| \wh{k}\l( \frac{\xi+\mu/2}{R_\b(\mu)} \r) \r| \ls \frac{R_\b^6(\mu)}{R_\b^6(\mu) + |\xi + \mu/2|^6} \sim \frac{\mu^6}{\mu^6 + |2\xi + \mu|^6}.\]
As a result, 
\be\label{T4, FT-est}
\big|\mcal{F}[\eta(t)\mcal{T}_{\b,4}(h)](\xi,\tau)\big|\ls \int_{4}^{\infty}\l|\wh{\eta}\big(\tau-P_\b(\mu)\big)\r|\,\frac{\mu^5|\wh{h}(\mu)|}{\mu^6+|2\xi+\mu|^6}\,d\mu.\ee
For any $\mu\geq 4$, $\la\tau-\phi^{1,\b}(\xi)\ra \leq \la\tau-P_\b(\mu)\ra\la P_\b(\mu)-\phi^{1,\b}(\xi)\ra \ls \la\tau-P_\b(\mu)\ra(\mu+|\xi|)^3$, so it follows from (\ref{T4, FT-est}) that
\be\label{T4, X norm-est1}\begin{split}
& \la\xi\ra^s\la\tau-\phi^{1,\b}(\xi)\ra^{\frac12}\l| \mcal{F}[\eta(t)\mcal{T}_{\b,4}(h)](\xi,\tau)\r|\\
\ls\quad &\int_{4}^{\infty}\la\tau-P_\b(\mu)\ra^{\frac12}\l|\wh{\eta}\big(\tau-P_\b(\mu)\big)\r|\,\frac{\la\xi\ra^s(\mu+|\xi|)^{\frac32}}{\mu^6+|2\xi+\mu|^6}\,\mu^5|\wh{h}(\mu)|\,d\mu.
\end{split}\ee
It remains to estimate the $L^2_{\xi,\tau}$ norm of the right hand side of (\ref{T4, X norm-est1}). First, for any fixed $\mu\geq 4$, by dividing $\m{R}$ into $\{\xi:|\xi|\geq \mu\}$ and $\{\xi:|\xi|<\mu\}$, and by utilizing the assumption $-\frac12<s\leq 3$, we attain
\be\label{T4, X norm-key est}
\bigg\| \frac{\la\xi\ra^s(\mu+|\xi|)^{\frac32}}{\mu^6 + |2\xi+\mu|^6}  \bigg\|_{L^2_\xi} \ls C\mu^{s-4}.\ee
Thus, by (\ref{T4, X norm-est1}) and Minkowski's inequality, 
\[\begin{split}
& \l\|\la\xi\ra^s\la\tau-\phi^{1,\b}(\xi)\ra^{\frac12} \mcal{F}[\eta(t)\mcal{T}_{\b,4}(h)](\xi,\tau)\r\|_{L^2_{\xi}}\\
\ls\quad &\int_{4}^{\infty}\la\tau-P_\b(\mu)\ra^{\frac12}\l|\wh{\eta}\big(\tau-P_\b(\mu)\big)\r|\mu^{s+1}|\wh{h}(\mu)|\,d\mu.
\end{split}\]
This leads to 
\be\label{T4, X norm-est2}
\|\eta(t)\mcal{T}_{\b,4}(h)\|_{X^{1,\b}_{s,\frac12}} \ls \l \| \int_{4}^{\infty}\la\tau-P_\b(\mu)\ra^{\frac12}\l|\wh{\eta}\big(\tau-P_\b(\mu)\big)\r|\mu^{s+1}|\wh{h}(\mu)|\,d\mu \r\|_{L^2_\tau}. \ee
Note that the function $P_\b(\mu) = \mu^3 - \beta\mu$ is increasing on $[4,\infty)$, so its inverse function is well-defined which is denoted as $P_{\b}^{-1}$. By making the change of variable $y=P_{\b}(\mu)$, we have $ y \sim \mu^3 $ and
\be\label{T4, X norm-est3}
\|\eta(t)\mcal{T}_{\b,4}(h)\|_{X^{1,\b}_{s,\frac12}} \ls \l\| \int_{P_\b(4)}^{\infty}\la\tau-y\ra^{\frac12}|\wh{\eta}(\tau-y)|\,y^{\frac{s-1}{3}}\,\big|\wh{h}\big(P_{\b}^{-1}(y)\big)\big|\,dy \r\|_{L^2_\tau} .\ee
Since $\eta$ is a Schwarz function, $\la\cdot\ra^{\frac12}\wh{\eta}(\cdot)$ is integrable. Then we apply Young's inequality to (\ref{T4, X norm-est3}) to find
\be\label{T4, X norm-est4}
\|\eta(t)\mcal{T}_{\b,4}(h)\|_{X^{1,\b}_{s,\frac12}} \ls \bigg(\int_{P_\b(4)}^{\infty}y^{\frac{2(s-1)}{3}}\l|\wh{h}\big(P_{\b}^{-1}(y)\big)\r|^2\,dy\bigg)^{\frac12}.\ee
Finally, making the change of variable $\mu=P_{\b}^{-1}(y)$ yields 
$\|\eta(t)\mcal{T}_{\b,4}(h)\|_{X^{1,\b}_{s,\frac12}} \ls \|h\|_{H^s(\m{R})}$. 
\end{proof}

\begin{proof}[{\bf Proof of (\ref{pos-kdv-bdr, Lambda})}]
By taking advantage of the above proof, it is not difficult to justify (\ref{pos-kdv-bdr, Lambda}). Next, we will sketch its proof with a focus on illustrating why the restriction $ \sigma\leq \frac{2s+7}{12} $ is needed in Lemma \ref{Lemma, pos-kdv-bdr}. Firstly, since $ s\leq 3 $ and $ \sigma\leq 1 $,
then $\|f\|_{\Lambda^{1,\b}_{s,\sigma}} \ls \|f\|_{X^{1,\b}_{3,1}}$ for any function $f$. Combining this fact with (\ref{T1 est-FR norm}) and (\ref{T3 est-FR norm}), we deduce
\begin{align*}
&\|\eta(t)\mcal{T}_{\b,1}(h)\|_{\Lambda^{1,\b}_{s,\sigma}}+\|\eta(t)\mcal{T}_{\b,3}(h)\|_{\Lambda^{1,\b}_{s,\sigma}} \ls  \|\eta(t)\mcal{T}_{\b,1}(h)\|_{X^{1,\b}_{3,1}}+\|\eta(t)\mcal{T}_{\b,3}(h)\|_{X^{1,\b}_{3,1}}\ls \|h\|_{H^{s}(\m{R})}.
\end{align*}
It thus remains to estimate $\|\eta(t)\mcal{T}_{\b,4}(h)\|_{\Lambda^{1,\b}_{s,\sigma}}$. On the one hand, it follows from (\ref{T4, FT-est}) that 
\[\big|\mcal{F}[\eta(t)\mcal{T}_{\b,4}(h)](\xi,\tau)\big|\ls \int_{4}^{\infty}\l|\wh{\eta}\big(\tau-P_\b(\mu)\big)\r|\,\mu^{-1}|\wh{h}(\mu)|\,d\mu.\]
On the other hand, it is easily seen that for any $\eps_1>0$,
\[\Big\|\mb{1}_{\{e^{|\xi|}\leq 3+|\tau|\}}\la\xi\ra^{s}\la\tau-\phi^{\a,\b}(\xi)\ra^{\sigma}\Big\|_{L^{2}_{\xi}} \ls C_{\eps_1} \la\tau\ra^{\sigma+\eps_1},\]
where $ C_{\eps_1} $ denotes a constant that depends on $\eps_1$. The specific value of $\eps_1$ will be chosen later in (\ref{T4, L-norm, eps_1}) and it only depends on $s$. Hence, 
\be\label{T4, L-norm-est2}\begin{split}
\big\|\eta(t)\mcal{T}_\b(h)\big\|_{\Lambda^{1,\b}_{s,\sigma}} &= \l\|\mb{1}_{\{e^{|\xi|}\leq 3+|\tau|\}}\la\xi\ra^{s}\la\tau-\phi^{\a,\b}(\xi)\ra^{\sigma} \mcal{F}[\eta(t)\mcal{T}_{\b,4}(h)](\xi,\tau)\r\|_{L^2_{\xi,\tau}}\\
&\ls \l\| \int_{4}^{\infty}\la\tau\ra^{\sigma+\eps_1}\l|\wh{\eta}\big(\tau-P_\b(\mu)\big)\r|\mu^{-1}|\wh{h}(\mu)|\,d\mu \r\|_{L^2_\tau}.
\end{split}\ee
Noticing 
$\la\tau\ra^{\sigma+\eps_1}\ls \la\tau-P_\b(\mu)\ra^{\sigma+\eps_1}\la\mu\ra^{3(\sigma+\eps_1)}$,
so 
\[\big\|\eta(t)\mcal{T}_\b(h)\big\|_{\Lambda^{1,\b}_{s,\sigma}} \ls \l\| \int_{4}^{\infty} \la\tau-P_\b(\mu)\ra^{\sigma+\eps_1} \l|\wh{\eta}\big(\tau-P_\b(\mu)\big)\r| \mu^{3(\sigma+\eps_1)-1} |\wh{h}(\mu)| \,d\mu \r\|_{L^2_\tau}.\]
Comparing to (\ref{T4, X norm-est2}), as long as $3(\sigma+\eps_1)-1\leq s+1$, that is $\sigma\leq \frac{s+2}{3}-\eps_1$, we can finish the proof by similar argument as that after (\ref{T4, X norm-est2}). Choosing 
\be\label{T4, L-norm, eps_1} 
\eps_1=\frac{2s+1}{12}
\ee
fulfills this requirement thanks to the restriction $\sigma\leq  \frac{2s+7}{12}$.
\end{proof}

\begin{proof}[{\bf Proof of (\ref{pos-kdv-bdr, time-trace})}]
Again, based on the proof of (\ref{pos-kdv-bdr, FR}), it will not take significantly more effort to verify (\ref{pos-kdv-bdr, time-trace}), but we will still carry out its proof with the purose of revealing the necessity of the assumption $ s>-\frac12 $ in Lemma \ref{Lemma, pos-kdv-bdr}. Firstly, it follows from the 1-D Sobolev embedding and the condition $ s\leq 3 $ that  $\|f\|_{L^{\infty}_{t}H^{s}_{x}}\ls \|f\|_{X^{1,\b}_{s,1}}\ls \|f\|_{X^{1,\b}_{3,1}}$. Then we use (\ref{T1 est-FR norm}) and (\ref{T3 est-FR norm}) again to deduce
\begin{align*}
\|\eta(t)\mcal{T}_{\b,1}(h)\|_{L^{\infty}_{t} H^{s}_{x}} + \|\eta(t)\mcal{T}_{\b,3}(h)\|_{L^{\infty}_{t} H^{s}_{x}} \ls \|\eta(t)\mcal{T}_{\b,1}(h)\|_{X^{1,\b}_{3,1}} + \|\eta(t)\mcal{T}_{\b,3}(h)\|_{X^{1,\b}_{3,1}} \ls \|h\|_{H^{s}(\m{R})}.
\end{align*}
It thus remains to estimate $\|\eta(t)\mcal{T}_{\b,4}(h)\|_{L^{\infty}_{t}H^{s}_{x}}$. We first consider the case when $t=0$ which reduces to the proof of $\|[\mcal{T}_{\b,4}(h)](x,0)\|_{H^s_x}\ls \|h\|_{H^s(\m{R})}$. By direct computation, 
\[\l\|\big[\mcal{T}_{\b,4}(h)\big](x,0)\r\|_{H^s_x} =\l\|\la\xi\ra^{s}\int_{4}^{\infty}\wh{k}\l(\frac{\xi+\mu/2}{R_\b(\mu)}\r)\frac{\wh{h}(\mu)}{R_\b(\mu)}\,d\mu \r\|_{L^2_{\xi}}.\]
According to duality, 
\begin{align}\label{t-trace, T4-est1}
\l\|\big[\mcal{T}_{\b,4}(h)\big](x,0)\r\|_{H^s_x} &= \sup_{\|g\|_{L^2_\xi}=1}\int_{\m{R}}\int_{4}^{\infty}g(\xi)\la\xi\ra^{s}\wh{k}\l(\frac{\xi+\mu/2}{R_\b(\mu)}\r)\frac{\wh{h}(\mu)}{R_\b(\mu)}\,d\mu\,d\xi.
\end{align}
Exchanging the order of integration and shifting $\xi$, we obtain
\be\begin{split}\label{t-trace, T4-est2}
\int_{\m{R}}\int_{4}^{\infty} g(\xi) \la\xi\ra^{s} \wh{k}\Big(\frac{\xi+\mu/2}{R_\b(\mu)}\Big) \frac{\wh{h}(\mu)}{R_\b(\mu)} d\mu d\xi 
= 
\int_{4}^{\infty} \int_{\m{R}} g\Big(\xi-\frac{\mu}{2}\Big )\l\la\xi-\frac{\mu}{2}\r\ra^{s} \wh{k} \Big( \frac{\xi}{R_\b(\mu)} \Big) \frac{\wh{h}(\mu)}{R_\b(\mu)} d\xi d\mu.
\end{split}\ee
Then it suffices to control the RHS of (\ref{t-trace, T4-est2}). We split the integral domain $ [4,\infty)\times \R $ as $ B_1 \cup B_2 $, where 
$B_{1} = \l\{(\mu,\xi)\in\m{R}^2: \mu\geq 4, \big|\xi-\frac{\mu}{2}\big|\ls 1\r\}$ and $B_{2} = \l\{(\mu,\xi)\in\m{R}^2: \mu\geq 4, \big|\xi-\frac{\mu}{2}\big|\gg 1\r\}$.

Since $ \xi $ is restricted to a small region for each fixed $ \mu $ in $ B_1 $, the contribution of the integral on $ B_1 $ can be handled easier, so the focus will be put on the contribution on $ B_2 $. 
By the change of variable $\xi=R_\b(\mu)y$, 
\be\label{t-trace, T4-B2}
\mbox{Contribution on $B_2$} \ls \int_{\m{R}}\int_{\m{R}}\mb{1}_{B_{3}}(\mu,y)\Big|g\Big(R_\b(\mu)y-\frac{\mu}{2}\Big)\Big|\l\la R_\b(\mu)y-\frac{\mu}{2}\r\ra^{s}|\wh{k}(y)||\wh{h}(\mu)|\,dy\,d\mu,\ee
where 
$B_3:=\l\{(\mu,y):\mu\geq 4, \l |R_\b(\mu)y-\frac{\mu}{2}\r|\gg 1\r\}$.
Switching the order of integration yields
\be\label{t-trace, T4-B2-est1}
\mbox{RHS of (\ref{t-trace, T4-B2})}\leq\int_{\m{R}}|\wh{k}(y)|\bigg(\int_{\m{R}}\mb{1}_{B_{3}}(\mu,y)\Big| g\Big(R_\b(\mu)y-\frac{\mu}{2}\Big)\Big| \l\la R_\b(\mu)y-\frac{\mu}{2}\r\ra^{s}|\wh{h}(\mu)|\,d\mu\bigg)\,dy.\ee
For any $(\mu,y)\in B_3$, $\mu\geq 4$ and $\big|R_\b(\mu)y-\frac{\mu}{2}\big|\gg 1$. Then we find
\be\label{t-trace, T4-B2-key est}
\Big| R_\b(\mu)y-\frac{\mu}{2}\Big|\sim \Big|\mu\Big(y-\frac{1}{\sqrt{3}}\Big)\Big| \quad\mbox{and}\quad \Big| R_{\b}'(\mu)y-\frac12\Big|\sim \Big| y-\frac{1}{\sqrt{3}}\Big| .\ee
Since $\|g\|_{L^2}=1$ and $\frac{d}{d\mu}\big[R_\b(\mu)y-\frac{\mu}{2}\big]=R_\b'(\mu)y-\frac12$, it follows from H\"older's inequality that
\be\label{t-trace, T4-B2-est2}\begin{split}
&\int_{\m{R}}\mb{1}_{B_{3}}(\mu,y) \Big| g\Big(R_\b(\mu)y-\frac{\mu}{2}\Big)\Big| \l\la R_\b(\mu)y-\frac{\mu}{2}\r\ra^{s}|\wh{h}(\mu)|\,d\mu\\
\ls\quad & \bigg(\int_{\m{R}}\frac{\mb{1}_{B_{3}}(\mu,y)}{\big|R_\b'(\mu)y-\frac{1}{2}\big|}\l\la R_\b(\mu)y-\frac{\mu}{2}\r\ra^{2s}|\wh{h}(\mu)|^2\,d\mu\bigg)^{\frac12}.
\end{split}\ee
Based on (\ref{t-trace, T4-B2-key est}),
\begin{align*}
\mbox{RHS of (\ref{t-trace, T4-B2-est2})} \ls \bigg(\int_{\m{R}}\mb{1}_{B_{3}}(\mu,y) \Big|y-\frac{1}{\sqrt{3}}\Big|^{2s-1}\mu^{2s}|\wh{h}(\mu)|^2\,d\mu\bigg)^{\frac12} \ls \Big|y-\frac{1}{\sqrt{3}}\Big|^{s-\frac12}\|h\|_{H^s(\m{R})}.
\end{align*}
Then we plug this inequality into (\ref{t-trace, T4-B2-est1}) to deduce
\begin{align*}
\mbox{RHS of (\ref{t-trace, T4-B2})} \ls \bigg(\int_{\m{R}}|\wh{k}(y)|\Big|y-\frac{1}{\sqrt{3}}\Big|^{s-\frac12}\,dy\bigg)\|h\|_{H^s(\m{R})} \ls \|h\|_{H^s(\m{R})},
\end{align*}
where the last inequality is due to the assumption $s>-\frac12$. 


Having established the case when $t=0$, the general case follows immediately. Actually, for any $t$, by defining $f$ such that $\wh{f}(\mu)=e^{iP_\b(\mu)t}\wh{h}(\mu)$, we have $\l[\mcal{T}_{\b,4}(f)\r](x,0)=\l[\mcal{T}_{\b,4}(h)\r](x,t)$. So we infer from the $t=0$ case that 
\[\l\|\big[\mcal{T}_{\b,4}(h)\big](x,t)\r\|_{H^{s}_{x}}=\l\|\big[\mcal{T}_{\b,4}(f)\big](x,0)\r\|_{H^{s}_{x}}\ls \|f\|_{H^s}=\|h\|_{H^s}.\]
\end{proof}

Thus, the proof of Lemma \ref{Lemma, pos-kdv-bdr} is completed.
\end{proof}

\begin{rem}\label{Re, -1/2nec}
We emphasize that the above proof for (\ref{pos-kdv-bdr, time-trace}) can not be extended to $s< -\frac12$. Actually, if we choose 
$g(\xi)=\mb{1}_{\{|\xi|\leq 1\}}$ in (\ref{t-trace, T4-est1}), then 
$\mbox{RHS of (\ref{t-trace, T4-est1})}\sim \int_{4}^{\infty}\frac{\wh{h}(\mu)}{R_\b(\mu)}\,d\mu\sim \int_{4}^{\infty}\frac{\wh{h}(\mu)}{\mu}\,d\mu$.
In order to prove 
$\int_{4}^{\infty}\frac{\wh{h}(\mu)}{\mu}\,d\mu\ls \|h\|_{H^s}$ for any $h\in H^{s}(\m{R})$,
$s$ has to be at least $-\frac12$. In fact, by the scaling $\wh{h_\lam}:=\wh{h}(\lam\mu)$ and sending $\lam\rightarrow 0^+$, we can see $s\geq -\frac12$.
\end{rem}

Next, we will take advantage of the proof of Lemma \ref{Lemma, pos-kdv-bdr} to justify Lemma \ref{Lemma, pos-kdv-bdr, low-reg}.

\begin{proof}[{\bf Proof of Lemma \ref{Lemma, pos-kdv-bdr, low-reg}}]
Recall the definition of  $W^{1,\b}_{bdr}$ in (\ref{formula for pos-kdv-bdr}), we write
\begin{align*}
W_{bdr}^{1,\b}(a)(x,t)=& \frac{1}{\pi}Re\int_{\sqrt{\b}}^{4}e^{iP_{\b}(\mu)t}e^{-i\mu x/2}e^{-R_{\b}(\mu)x}\wh{a^*}[P_{\b}(\mu)]\,dP_{\b}(\mu)\\
&+\frac{1}{\pi}Re\int_{4}^{\infty}e^{iP_{\b}(\mu)t}e^{-i\mu x/2}e^{-R_{\b}(\mu)x}\wh{a^*}[P_{\b}(\mu)]\,dP_{\b}(\mu)\\
:=&I_1(x,t)+I_2(x,t), \qquad\forall\, x,t\geq 0.
\end{align*}
Notice that one can extend $I_1(x,t)$ for all $x$ and $t$ with its extension chosen as
\be
\mcal{E}I_1(x,t):=\frac{1}{\pi}Re\int_{\sqrt{\b}}^{4}e^{iP_{\b}(\mu)t}e^{-i\mu x/2}K_{\b}(x,\mu)\wh{a^*}[P_{\b}(\mu)]\,dP_{\b}(\mu), \qquad\forall\,x,t\in\m{R}\label{pos-bdr-opI1}
\ee
where  $K_\b$ is as defined in \eqref{K}. For this part, it is very similar to the operator $\Phi_{bdr}^{1,\b}$ in Lemma \ref{Lemma, pos-kdv-bdr} but with $\mu$ restricted on $[\sqrt{\b},4]$. On this small interval $ [\sqrt{\b},4] $, the proof of Lemma \ref{Lemma, pos-kdv-bdr} also works for $-\frac34<s\leq -\frac12$ and the extension for this part satisfies all the desired estimates.
Unfortunately, the same extension fails to work if $\mu$ is large. So we need a more subtle extension for $I_2(x,t)$. For the convenience of statement and the latter generalization in Section \ref{Sec, neg KdV}, we introduce the following operator $\mcal{I}_{\b,m}$ for $0<\b\leq 1$ and $m\in\m{R}$. 
\be\label{def of I_b}
[\mcal{I}_{\b,m}(f)](x,t)=Re \int_{4}^{\infty}e^{iP_{\b}(\mu)t}e^{-im\mu x}e^{-R_\b(\mu)x}\wh{f}\big(P_\b(\mu)\big)\,dP_\b(\mu),\quad \forall\, x\geq 0,\,t\in\m{R}.\ee
Then 
$I_2(x,t)=\frac{1}{\pi}\big[\mcal{I}_{\b,\frac12}(a^*)\big](x,t)$ for any $x,t\geq 0$.
So in order to extend $I_2$, it suffices to extend $\mcal{I}_{\b,\frac12}(a^*)$. 
By the following Claim \ref{Prop, kdv-bdr-low-reg}, we can choose the extension to be $ \frac{1}{\pi} \mcal{EI}_{\b,\frac12}(a^*) $ to prove Lemma \ref{Lemma, pos-kdv-bdr, low-reg}. Therefore, the justification of Lemma \ref{Lemma, pos-kdv-bdr, low-reg} reduces to establishing the following result.
\begin{cla}
\label{Prop, kdv-bdr-low-reg}
Let $-\frac34<s\leq -\frac12$, $0<\b\leq 1$, $m\neq 0$ and $f\in H^{\frac{s+1}{3}}(\m{R})$. Then there exist $\sigma_1=\sigma_1(s)>\frac12$ and an extension $\mcal{EI}_{\b,m}(f)$ of $\mcal{I}_{\b,m}(f)$ from $\m{R}^+_0\times\m{R}$ to $\m{R}\times\m{R}$ such that for any $\sigma\in\big(\frac12,\sigma_1\big]$ and for $j=0,1$,
\[\eta(t)\mcal{EI}_{\b,m}(f)\in Y_{s,\frac12,\sigma}^{1,\b}\bigcap C_{x}^{j}\big(\m{R}^+_0; H_{t}^{\frac{s+1-j}{3}}(\m{R})\big).\]
In addition, the following estimates hold with some constant $C=C(s,\sigma,m)$. 
\begin{align}
\big\|\eta(t)\mcal{EI}_{\b,m}(f)\big\|_{Y^{1,\b}_{s,\frac12,\sigma}} &\leq  C\|f\|_{H^{\frac{s+1}{3}}(\m{R})},\label{est for kdv-bdr-low reg} \\
\sup_{x\geq 0}\Big\|\p^j_x \big[\eta(t)\mcal{EI}_{\b,m}(f)\big]\Big\|_{H^{\frac{s+1-j}{3}}_t(\m{R})} &\leq  C\|f\|_{H^{\frac{s+1}{3}}(\m{R})}, \quad j=0,1. \label{trace est for kdv-bdr-low reg}
\end{align}
\end{cla}

\begin{proof}[{\bf Proof of Claim \ref{Prop, kdv-bdr-low-reg}}]

Without loss of generality, we assume $m=\frac12$, the general case is similar. First, we notice that when $x\geq 0$, any extension of $\mcal{I}_{\b,\frac12}(f)$ from $\m{R}^+_0\times\m{R}$ to $\m{R}\times\m{R}$ is equal to $\mcal{I}_{\b,\frac12}(f)$ on $\m{R}^+_0\times\m{R}$, so the LHS of (\ref{trace est for kdv-bdr-low reg}) is the same for all extensions.  Define $ h $ such that 
$\wh{h}(\mu) = \wh{f} \big( P_\b(\mu) \big)(P_{\b})'(\mu)\mb{1}_{\mu>4}$. 
Then $ \| h \|_{H^s(\R)} \ls \| f \|_{H^{\frac{s+1}{3}}(\R)} $. Meanwhile, it follows from (\ref{def of I_b}) that 
\[
[\mcal{I}_{\b,\frac12}(f)](x,t) = Re \int_{4}^{\infty}e^{iP_{\b}(\mu)t}e^{-i\mu x/2}e^{-R_\b(\mu)x} \wh{h}(\mu) \,d\mu, \quad \forall\, x\geq 0,\,t\in\m{R}.
\]
Recalling the definition (\ref{def of T4}) for the operator $\mcal{T}_{\b,4}(h)$, we see that $ \mcal{I}_{\b,\frac12}(f) $ agrees with $ Re(\mcal{T}_{\b,4}(h)) $ when $ x\geq 0 $. So $ Re(\mcal{T}_{\b,4}(h)) $ is an extension of $\mcal{I}_{\b,\frac12}(f)$ from $\m{R}^+_0\times\m{R}$ to $\m{R}\times\m{R}$. As a result, for any other extension $\mcal{EI}_{\b,\frac12}(f)$,
\[\begin{split}
\sup_{x\geq 0} \big\| \p^j_x \big[ \eta(t)\mcal{EI}_{\b,\frac12}(f) \big] \big\|_{H^{\frac{s+1-j}{3}}_t(\R)} &= \sup_{x\geq 0}\big\|\p^j_x \big[\eta(t) Re(\mcal{T}_{\b,4}(h)) \big] \big\|_{H^{\frac{s+1-j}{3}}_t(\R)} \\
& \leq \sup_{x\geq 0} \big\| \p^j_x \big[ \eta(t) \mcal{T}_{\b,4}(h) \big] \big\|_{H^{\frac{s+1-j}{3}}_t(\R)}
\end{split} 
,\qquad j=0,1.\]
Although the value of $s$ is required to be greater than $-1/2$ in Lemma \ref{Lemma, pos-kdv-bdr}, the proof of (\ref{pos-kdv-bdr, space-trace}) is valid for any $s$ in $(-\frac34,-\frac12]$. So we conclude $\eta(t)\mcal{T}_{\b,4}(h)\in C_{x}^{j}\big(\m{R}^+_0; H_{t}^{\frac{s+1-j}{3}}(\m{R})\big)$ and
\[\sup_{x\geq 0} \big\| \p^j_x \big[ \eta(t)\mcal{T}_{\b,4}(h) \big] \big\|_{H^{\frac{s+1-j}{3}}_t(\m{R})} \leq C \| h \|_{H^s(\R)} \leq C\|f\|_{H^{\frac{s+1}{3}}(\m{R})}.\]
Then it remains to find an extension $\mcal{EI}_{\b,\frac12}(f)$ which belongs to $Y_{s,\frac12,\sigma}^{1,\b}$ and satisfies (\ref{est for kdv-bdr-low reg}).

In \cite{BSZ06}, the authors provided an extension for the integral $I_2(x,t)$, see page 16 in Section 2.2 in \cite{BSZ06}. Essentially, the term $I_2(x,t)$ in \cite{BSZ06} can be rewritten as 
\[\begin{split}
I_2(x,t) &=Re \int_{4}^{\infty}e^{i(\mu^3-\mu)t}e^{-\big(\sqrt{3\mu^2-4}+i\mu\big)\frac{x}{2}}\wh{h^*}\big(\mu^3-\mu\big)\,d(\mu^3-\mu),\\
&=\big[\mcal{I}_{1,\frac12}(h^*)\big](x,t), \qquad \forall\,x\geq 0,\,t\in\m{R},
\end{split} \]
where $h\in H^{\frac{s+1}{3}}(\m{R}^+)$ and $h^*$ is the zero extension of $h$ from $\m{R}^+$ to $\m{R}$. By Lemma \ref{Lemma, 0 ext est}, $h^*\in H^{\frac{s+1}{3}}(\m{R})$ and this is the only property needed in \cite{BSZ06} to carry out the extension which satisfies the property (\ref{est for kdv-bdr-low reg}), see the extension $\mcal{B}I_{m1}$ on page 21, Theorem 3.1 on page 22, and Lemma 3.10 on page 36 in \cite{BSZ06}. 
After replacing $h^*$ by $f$, the method also applies to extend $\mcal{I}_{1,\frac12}(f)$ to satisfy (\ref{est for kdv-bdr-low reg}). Thus, Claim \ref{Prop, kdv-bdr-low-reg} is justified when $\b=1$ and $m=\frac12$. For general $\b\in(0,1)$, it follows from $\mu\geq 4$ that $P_\b(\mu)\sim \mu^{3}$ and $R_{\b}(\mu)\sim \mu$, so Claim \ref{Prop, kdv-bdr-low-reg} can also be justified in a similar way for any $ 0<\beta\leq 1 $.
\end{proof}

Hence,  the proof of Lemma \ref{Lemma, pos-kdv-bdr, low-reg} is finished.
\end{proof}

\subsection{KdV flow traveling to the left}
\label{Sec, neg KdV}

In this section,  we  will consider  the linear IBVP (\ref{lin-v}) which is copied below for convenience of readers.
\begin{equation} \label{k-2}
\begin{cases}
v_t-v_{xxx}-\b v_{x}=g,\\
v(x,0)=q(x),\\
v(0,t)=b_1(t),\quad v_x(0,t)=b_2(t),
\end{cases} \qquad x,t>0.
\end{equation}
This system describes a linear dispersive wave traveling  to the left, where $q\in H^{s}(\m{R}^+)$, $b_{1}\in H^{\frac{s+1}{3}}(\m{R}^+)$, $b_{2}\in H^{\frac{s}{3}}(\m{R}^+)$, and $g$ is a function defined on $\m{R}^2$. The function space that $g$ lives in will be specified later. In addition, the data $q$, $b_1$ and $b_2$ are assumed to be compatible as in Definition \ref{Def, compatibility-lin}, that is they satisfy $q(0)=b_1(0)$ if $s>\frac12$ and further satisfy $q'(0)=b_2(0)$ if $s>\frac32$.
The solution $v$ of (\ref{k-2}) can be decomposed into $v = v_1 + v_2$ 
with
\be\label{soln to neg-kdv-initial}
v_1=\Phi^{-1,-\b}_{R}(g,q):=W^{-1,-\b}_{R}(E_{s}q)+\int_{0}^{t}W^{-1,-\b}_{R}(t-\tau)g(\cdot,\tau)\,d\tau.\ee
where $E_{s}q$ is defined as in \eqref{ext for p}, 
and $ v_2 = W^{-1, -\beta}_{bdr} (b_1- q_1, b_2-q_2 )$, 
where  $ q_1 := v_1 (0,t) $, $ q_2 (t) := \partial _x v_1 (0,t) $, and 
$W^{-1, -\beta}_{bdr} (b_1,b_2) $  denotes the solution operator (also called the boundary integral operator)  associated to the following IBVP (\ref{neg-kdv-bdr}). 
\be\label{neg-kdv-bdr}\begin{cases}
z_t - z_{xxx} - \b z_{x}=0,\\
z(x,0)=0,\\
z(0,t)=b_1(t),\quad z_{x}(0,t)=b_2(t),
\end{cases} \quad x,t>0.
\ee


\begin{lem}\label{Lemma, neg-kdv-initial}
Let $s\in(-\frac34,3]$, $0<\b\leq 1$ and $\frac12<\sigma\leq 1$. For any $q\in H^{s}(\m{R}^+)$ and $g\in X^{-1,-\b}_{s,\sigma-1}\bigcap Z^{-1,-\b}_{s,\sigma-1}$, the function 
$\wt{v_1}:=\eta(t)\Phi^{-1,-\b}_{R}(g,q)$
belongs to $Y_{s,\frac12,\sigma}^{-1,-\b}\bigcap C_{x}^{j}\big(\m{R}; H_{t}^{\frac{s+1-j}{3}}(\m{R})\big)$ for $j=0,1$. In addition, the following estimates hold with some constant $C=C(s,\sigma)$.
\begin{align}
\big\| \wt{v_1}\big\|_{Y^{-1,-\b}_{s,\frac12,\sigma}} &\leq  C \big( \|g\|_{X^{-1,-\b}_{s,\sigma-1}} + \|q\|_{H^{s}(\R^+)} \big), 
\label{est for neg-kdv-initial} \\
\sup_{x\in\m{R}} \big\| \p_x^{j} \wt{v_1} \big\|_{H^{\frac{s+1-j}{3}}_t(\m{R})} &\leq  C \Big( \|g\|_{X^{-1,-\b}_{s,\sigma-1}} + \|g\|_{Z^{-1,-\b}_{s,\sigma-1}} + \|q\|_{H^{s}(\R^+)} \Big), \quad j=0,1.
\label{trace est for neg-kdv-initial}
\end{align}
\end{lem}

\begin{proof}
(\ref{est for neg-kdv-initial}) follows from (\ref{ext for p}) and Lemma \ref{Lemma, lin est}. (\ref{trace est for neg-kdv-initial}) follows  from(\ref{ext for p}), Lemma \ref{Lemma, space trace for free kdv} and Lemma \ref{Lemma, space trace for Duhamel}.
\end{proof}


Parallel to Lemma \ref{Lemma, formula for pos-kdv-bdr}, we can use the Laplace transform method in \cite{BSZ02} to derive an explicit formula for the boundary integral operator $ W^{-1,-\b}_{bdr}(b_1,b_2) $. For any $\b>0$, recall the definitions (\ref{def of P_b}) and (\ref{def of R_b}) for $P_\b$ and $R_\b$.

\begin{lem}\label{Lemma, formula for neg-kdv-bdr}
The boundary integral operator $W^{-1,-\b}_{bdr}(b_1,b_2) $ associated to  (\ref{neg-kdv-bdr}) has the explicit integral respresnetation
\be\label{formula for neg-kdv-bdr}\begin{split}
W^{-1,-\b}_{bdr}(b_1,b_2)(x,t) &=\frac{1}{\pi}Re\int_{\sqrt{\b}}^{\infty}e^{iP_{\b}(\mu)t}\Big[e^{-i\mu x}A-e^{\frac{i\mu x}{2}}e^{-R_{\b}(\mu)x}B\Big]\,dP_{\b}(\mu),\quad x,t\geq 0,
\end{split}\ee	
where 
\be\label{A and B}\left\{\begin{array}{l}
A = A(\mu) = \dfrac{R_\b(\mu)-i\mu/2}{R_\b(\mu)-3\mu i/2}\,\wh{b_1^*}[P_{\b}(\mu)]+\dfrac{1}{R_\b(\mu)-3\mu i/2}\,\wh{b_2^*}[P_{\b}(\mu)], \vspace{0.1in}\\
B = B(\mu) = \dfrac{i\mu}{R_\b(\mu)-3\mu i/2}\, \wh{b_1^*}[P_{\b}(\mu)]+\dfrac{1}{R_\b(\mu)-3\mu i/2}\,\wh{b_2^*}[P_{\b}(\mu)],
\end{array}\right.\ee
$b_1^* $ and $b_2^*$ are the zero extensions of $b_1$ and $b_2 $ from $\R^+_0$ to $\R$, respectively.
\end{lem}

Similar to the discussion after Lemma \ref{Lemma, formula for pos-kdv-bdr}, we intend to extend $W^{-1,-\b}_{bdr}(b_1,b_2)$ to the whole plane $\m{R}^2$ with some specific properties. Recalling the function $K_{\b}(x,\mu)$ in (\ref{K}), we define $\Phi^{-1,-\b}_{bdr}(b_1,b_2):\m{R}\times\m{R}\to\m{R}$ by
\be\label{neg-bdr-op1}\begin{split}
\Phi^{-1,-\b}_{bdr}(b_1,b_2)(x,t) &=\frac{1}{\pi}Re\int_{\sqrt{\b}}^{\infty}e^{iP_{\b}(\mu)t}\Big[e^{-i\mu x}A-e^{\frac{i\mu x}{2}}K_{\b}(x,\mu)B\Big]\,dP_{\b}(\mu),
\end{split}\ee
where $A$ and $B$ are as defined in (\ref{A and B}). It is easily seen that 
when $x,t\geq 0$, $K_{\b}(x,\mu) = e^{-R_{\b}(\mu)x}$, so $\Phi^{-1,-\b}_{bdr}(b_1,b_2)$ agrees with $W^{-1,-\b}_{bdr}(b_1,b_2)$ on $\m{R}^{+}_0\times \m{R}^{+}_0$. Hence, $\Phi^{-1,-\b}_{bdr}(b_1,b_2)$ is an extension of $W^{-1,-\b}_{bdr}(b_1,b_2)$ to $\m{R}\times\m{R}$. Moreover, if $s>-\frac12$, then it satisfies some desired properties as shown below.

\begin{lem}\label{Lemma, neg-kdv-bdr}
Let $s\in(-\frac12,3]$, $0<\b\leq 1$ and $\frac12<\sigma\leq \min\big\{1,\frac{2s+7}{12}\big\}$. For any $b_1\in H^{\frac{s+1}{3}}(\m{R}^+)$ and $b_2\in H^{\frac{s}{3}}(\m{R}^+)$ that are compatible with (\ref{neg-kdv-bdr}), the function 
\[\wt{v_2}:=\eta(t)\Phi^{-1,-\b}_{bdr}(b_1,b_2),\] 
equals $W^{-1,-\b}_{bdr}(b_1,b_2)$ on $\m{R}^+_0\times [0,1]$ and belongs to $Y_{s,\frac12,\sigma}^{-1,-\b}\bigcap C_{x}^{j}\big(\m{R}; H_{t}^{\frac{s+1-j}{3}}(\m{R})\big)$ for $j=0,1$. In addition, the following estimates hold with some constant $C=C(s,\sigma)$.
\begin{align}
\big\| \wt{v_2}\big\|_{Y^{-1,-\b}_{s,\frac12,\sigma}} &\leq  C\Big(\|b_1\|_{H^{\frac{s+1}{3}}(\m{R}^+)}+\|b_2\|_{H^{\frac{s}{3}}(\m{R}^+)}\Big),\label{est for neg-kdv-bdr} \\
\sup_{x\in\m{R}}\big\|\p^j_x \wt{v_2}\big\|_{H^{\frac{s+1-j}{3}}_t(\m{R})} &\leq  C\Big(\|b_1\|_{H^{\frac{s+1}{3}}(\m{R}^+)}+\|b_2\|_{H^{\frac{s}{3}}(\m{R}^+)}\Big), \quad j=0,1. \label{trace est for neg-kdv-bdr}
\end{align}
\end{lem}

\begin{proof}
We will first establish the estimates (\ref{est for neg-kdv-bdr}) and (\ref{trace est for neg-kdv-bdr}) for $s\in(-\frac12,3]\backslash\{\frac12 ,\frac32\}$. Then the corresponding estimates in the case of $ s=\frac12 $ or $ s=\frac32 $ can be deduced by interpolating between the case $ s=0 $ and the case $ s=3 $.
Fix $s\in(-\frac12,3]\backslash\{\frac12,\frac32\}$ and $0<\b\leq 1$. Define $f_{\b}$ and $g_{\b}$ such that $\wh{f_\b}(\mu)=(P_{\b})'(\mu)A$ and $\wh{g_\b}(\mu)=(P_{\b})'(\mu)B$, where $A$ and $B$ are as defined in (\ref{A and B}). Then 
\begin{align*}
\Phi^{-1,-\b}_{bdr}(b_1,b_2)(x,t)=\int_{\sqrt{\b}}^{\infty}e^{iP_{\b}(\mu)t}\Big[e^{-i\mu x}\wh{f_\b}(\mu)-e^{i\mu x/2}K_{\b}(x,\mu)\wh{g_\b}(\mu)\Big]\,d\mu.
\end{align*}
Since $0<\b\leq 1$ and $(P_{\b})'(\mu)=3\mu^2-\b$, then $| (P_\b)'(\mu) | \sim \mu$ and $ | R_\b(\mu) -3\mu i/2 | \sim \mu$. Therefore, 
\be\label{Mul, neg}\left\{\begin{array}{ll}
\l |(P_{\b})'(\mu)\dfrac{R_\b(\mu)-i\mu/2}{R_\b(\mu)-3\mu i/2}\r | &\ls |(P_{\b})'(\mu)|^{1/2}\la \mu\ra,\vspace{0.08in}\\
\l |(P_{\b})'(\mu)\dfrac{i\mu}{R_\b(\mu)-3\mu i/2}\r | &\ls |(P_{\b})'(\mu)|^{1/2}\la \mu\ra, \vspace{0.08in}\\
\l |(P_{\b})'(\mu)\dfrac{1}{R_\b(\mu)-3\mu i/2}\r | &\ls |(P_{\b})'(\mu)|^{1/2}.
\end{array}\right.\ee
Hence, it follows from (\ref{Mul, neg}) and (\ref{A and B}) that
\[
|\wh{f_\b}(\mu)|+|\wh{g_\b}(\mu)| \ls |(P_{\b})'(\mu)|^{1/2}\Big(\la\mu\ra\wh{b_1^*}[P_\b(\mu)]+\wh{b_2^*}[P_\b(\mu)]\Big).
\]
Therefore, 
$\la\mu\ra^{s} \big( |\wh{f_\b}(\mu)|+|\wh{g_\b}(\mu)| \big) \ls |(P_{\b})'(\mu)|^{1/2} \big( \la\mu\ra^{s+1}\wh{b_1^*}[P_\b(\mu)] + \la\mu\ra^{s}\wh{b_2^*}[P_\b(\mu)] \big)$.
Since $b_{1}\in H^{\frac{s+1}{3}}(\m{R}^+)$ and $b_{2}\in H^{\frac{s}{3}}(\m{R}^+)$ are compatible with (\ref{neg-kdv-bdr}), it then follows from Corollary \ref{Cor, mul est} that $f_{\b},g_{\b}\in H^{s}(\m{R})$ and 
\[\|f_\b\|_{H^{s}(\m{R})}+\|g_\b\|_{H^{s}(\m{R})}\ls \|b_1\|_{H^{\frac{s+1}{3}}(\m{R}^+)}+ \|b_2\|_{H^{\frac{s}{3}}(\m{R}^+)}.\]
Inspired by the above observation, we define an operator $\mcal{L}_{\b}$ as 
\be\label{L}
\big[\mcal{L}_{\b}(f,g)\big](x,t):=\int_{\sqrt{\b}}^{\infty}e^{iP_{\b}(\mu)t}\Big[e^{-i\mu x}\wh{f}(\mu)-e^{i\mu x/2}K_{\b}(x,\mu)\wh{g}(\mu)\Big]\,d\mu.\ee
Recalling $Y^{-1,-\b}_{s,\frac12,\sigma}=X^{-1,-\b}_{s,\frac12}\cap \Lambda^{-1,-\b}_{s,\sigma}\cap C_{t}(\m{R},H^s(\m{R}))$, the proof of Lemma \ref{Lemma, pos-kdv-bdr} reduces to establishing the following estimates for the operator $\mcal{L}_{\b}$.
\begin{align}
\big\|\eta(t)\mcal{L}_\b(f,g)\big\|_{X^{-1,-\b}_{s,\frac12}} &\leq  C(\|f\|_{H^s}+\|g\|_{H^s}),\label{neg-kdv-bdr, FR} \\
\big\|\eta(t)\mcal{L}_\b(f,g)\big\|_{\Lambda^{-1,-\b}_{s,\sigma}} &\leq  C(\|f\|_{H^s}+\|g\|_{H^s}),\label{neg-kdv-bdr, Lambda}\\  
\sup_{t\in\m{R}}\big\|\eta(t)\mcal{L}_\b(f,g)\big\|_{H^{s}_{x}(\m{R})} &\leq  C(\|f\|_{H^s}+\|g\|_{H^s}),\label{neg-kdv-bdr, time-trace} \\
\sup_{x\in\m{R}}\big\|\p^j_x \big[\eta(t)\mcal{L}_\b(f,g)\big]\big\|_{H^{\frac{s+1-j}{3}}_t(\m{R})} &\leq  C(\|f\|_{H^s}+\|g\|_{H^s}), \quad j=0,1. \label{neg-kdv-bdr, space-trace}
\end{align}
To justify these estimates, we wrtie $\mcal{L}_{\b}(f,g)=\mcal{L}_{\b,1}(f)-\mcal{L}_{\b,2}(g)$, where
\[\begin{split}
&[\mcal{L}_{\b,1}(f)](x,t):=\int_{\sqrt{\b}}^{\infty}e^{iP_{\b}(\mu)t}e^{-i\mu x}\wh{f}(\mu)\,d\mu, \\
&[\mcal{L}_{\b,2}(g)](x,t):=\int_{\sqrt{\b}}^{\infty}e^{iP_{\b}(\mu)t}e^{i\mu x/2}K_{\b}(x,\mu)\wh{g}(\mu)\,d\mu.\end{split}\]

\begin{itemize}
\item Firstly, we define $h$ such that $\wh{h}(\mu)=\mb{1}_{\{\mu>\sqrt{\b}\}}\wh{f}(\mu)$. Then $\| h \|_{H^s}\leq \|f\|_{H^s}$ and it follows from (\ref{semigroup op-1}) that
\[[\mcal{L}_{\b,1}(f)](x,t)=\big[W^{-1,-\b}_{R}(-t) h \big](-x), \quad\forall\,x,t\in\m{R}.\] Noticing that for any function $F(x,t)$, its $X^{-1,-\b}_{s,\frac12}$, $\Lambda^{-1,-\b}_{s,\sigma}$, $L^{\infty}_{t}H^{s}_{x}$ and $L^{\infty}_{x}H^{\frac{s+1-j}{3}}_{t}$ norms are preserved under the operation $F(x,t)\rightarrow F(-x,-t)$, so the estimates (\ref{neg-kdv-bdr, FR})-(\ref{neg-kdv-bdr, space-trace}) for $\mcal{L}_{\b,1}(f)$ follow from Lemma \ref{Lemma, lin est}, Lemma \ref{Lemma, space trace for free kdv} and the fact $\| h \|_{H^s}\leq \|f\|_{H^s}$.

\item Secondly, the operator $\mcal{L}_{\b,2}$ is very similar to $\mcal{T}_{\b}$ (see (\ref{T})). So by analogous arguments, the estimates (\ref{neg-kdv-bdr, FR})-(\ref{neg-kdv-bdr, space-trace}) for $\mcal{L}_{\b,2}(g)$ can also be established.
\end{itemize}
Therefore, Lemma \ref{Lemma, neg-kdv-bdr} is established.
\end{proof}

Analogous to Lemma \ref{Lemma, pos-kdv-bdr}, the estimate (\ref{est for neg-kdv-bdr}) fails for $s<-\frac12$, so in order to handle the case when $-\frac34<s\leq -\frac12$, we need another extension, denoted as $  \Psi_{bdr}^{-1,-\b}(b_1,b_2) $, for $W_{bdr}^{-1,-\b}(b_1,b_2)$. Again, we will combine the above extension $\Phi_{bdr}^{-1,-\b}(b_1,b_2)$ as in (\ref{neg-bdr-op1}) and the construction in \cite{BSZ06} to find it. 

\begin{lem}\label{Lemma, neg-kdv-bdr, low-reg}
Let $-\frac34<s\leq -\frac12$, $0<\b\leq 1$, $b_1\in H^{\frac{s+1}{3}}(\m{R}^+)$ and $b_2\in H^{\frac{s}{3}}(\m{R}^+)$. Then there exist $\sigma_2=\sigma_2(s)>\frac12$ and an extension $\Psi_{bdr}^{-1,-\b}(b_1,b_2)$ of $W_{bdr}^{-1,-\b}(b_1,b_2)$ such that for any $\sigma\in\big(\frac12,\sigma_2\big]$, the function 
\[\wt{v_2}:=\eta(t) \Psi_{bdr}^{-1,-\b}(b_1,b_2)\]
equals $W_{bdr}^{-1,-\b}(b_1,b_2)$ on $\m{R}^+_0\times[0,1]$ and belongs to $Y_{s,\frac12,\sigma}^{-1,-\b}\bigcap C_{x}^j\big(\m{R}^+_0; H_{t}^{\frac{s+1-j}{3}}(\m{R})\big)$ for $j=0,1$. In addition, the following estimates hold with some constant $C=C(s,\sigma)$. 
\begin{align}
\big\| \wt{v_2}\big\|_{Y^{-1,-\b}_{s,\frac12,\sigma}} &\leq  C\Big(\|b_1\|_{H^{\frac{s+1}{3}}(\m{R}^+)}+\|b_2\|_{H^{\frac{s}{3}}(\m{R}^+)}\Big),\label{est for neg-kdv-bdr, low-reg} \\
\sup_{x\geq 0}\big\|\p^j_x \wt{v_2}\big\|_{H^{\frac{s+1-j}{3}}_t(\m{R})} &\leq  C\Big(\|b_1\|_{H^{\frac{s+1}{3}}(\m{R}^+)}+\|b_2\|_{H^{\frac{s}{3}}(\m{R}^+)}\Big), \quad j=0,1. \label{trace est for neg-kdv-bdr, low-reg}
\end{align}
\end{lem}

\begin{proof}[\bf Proof of Lemma  \ref{Lemma, neg-kdv-bdr, low-reg}]
Recall the definition of  $W^{-1,-\b}_{bdr}$ in (\ref{formula for neg-kdv-bdr}) and write
\begin{align*}
W_{bdr}^{-1,-\b}(b_1,b_2)(x,t) = & \frac{1}{\pi} Re \int_{\sqrt{\b}}^{\infty} e^{iP_{\b}(\mu)t} e^{-i\mu x} A\,dP_\b(\mu) - \frac{1}{\pi} Re\int_{\sqrt{\b}}^{4} e^{iP_{\b}(\mu)t} e^{i\mu x/2} e^{-R_{\b}(\mu)x} B\,dP_{\b}(\mu) \\
& - \frac{1}{\pi} Re\int_{4}^{\infty} e^{iP_{\b}(\mu)t} e^{i\mu x/2} e^{-R_{\b}(\mu)x} B\,dP_{\b}(\mu) \\
:= & I_1(x,t) - I_2(x,t) - I_{3}(x,t), \qquad\forall\, x,t\geq 0,
\end{align*}
where $A$ and $B$ are defined as in (\ref{A and B}). For $I_1(x,t)$, it is actually well defined for any $x,t\in\m{R}$, so the trivial extension 
\[\mcal{E}I_1(x,t)=\frac{1}{\pi}Re \int_{\sqrt{\b}}^{\infty}e^{iP_{\b}(\mu)t}e^{-i\mu x}A\,dP_\b(\mu),\qquad\forall\, x,t\in\m{R},\]
satisfies (\ref{est for neg-kdv-bdr, low-reg}) and (\ref{trace est for neg-kdv-bdr, low-reg}), see the estimates for the operator $ \mcal{L}_{\b,1} $ in
the proof of Lemma \ref{Lemma, neg-kdv-bdr}. For $I_2(x,t)$, we define the extension to be 
\[\mcal{E}I_2(x,t)=\frac{1}{\pi}Re \int_{\sqrt{\b}}^{4}e^{iP_{\b}(\mu)t}e^{i\mu x/2}K_{\b}(x,\mu) B \,d P_\b(\mu),\qquad\forall\, x,t\in\m{R},\]
where $ K_\b $ is as defined in (\ref{K}). Then again, by similar argument as that in the proof of Lemma \ref{Lemma, neg-kdv-bdr}, this extension works. For $I_3(x,t)$, due to the formula (\ref{A and B}) for $B$, we define $h$ such that 
\[\wh{h}(\lam)=\mb{1}_{\{\lam\geq P_\b(4)\}}\bigg(\dfrac{i\mu}{R_\b(\mu)-3\mu i/2}\, \wh{b_1^*}(\lam)+\dfrac{1}{R_\b(\mu)-3\mu i/2}\,\wh{b_2^*}(\lam) \bigg),\]
where $\mu=P_\b^{-1}(\lam)$ for $\lam\geq P_\b(4)$. This choice of $h$ implies $\wh{h}\big(P_\b(\mu)\big)=\mb{1}_{\{\mu\geq 4\}}B $. So $I_3(x,t)$ can be rewritten as 
\[I_3(x,t)=\frac{1}{\pi}Re\int_{4}^{\infty}e^{iP_{\b}(\mu)t}e^{i\mu x/2}e^{-R_{\b}(\mu)x}\wh{h}\big(P_\b(\mu)\big)\,dP_{\b}(\mu)=\frac{1}{\pi}\big[\mcal{I}_{\b,-\frac12}(h)\big](x,t),\]
where the operator $\mcal{I}_{\b,-\frac12}(h)$ is defined as in (\ref{def of I_b}). Then it follows from Claim  \ref{Prop, kdv-bdr-low-reg} that the extension $\mcal{E}I_3(x,t):=\frac{1}{\pi}\big[\mcal{EI}_{\b,-\frac12}(h)\big](x,t)$ satisfies (\ref{est for neg-kdv-bdr, low-reg}) and (\ref{trace est for neg-kdv-bdr, low-reg}). 
Hence, Lemma \ref{Lemma, neg-kdv-bdr, low-reg} is verified.
\end{proof}

We now at the stage to present the poof of  Proposition  \ref{Prop, neg-kdv}.

\begin{proof}[\bf Proof of Proposition  \ref{Prop, neg-kdv}]

Without loss of generality, we assume $T=1$. We then first deal with the case when $s\in(-\frac12,3]$. Choose $\sigma_2(s)=\min\big\{1,\frac{2s+7}{12}\big\}$ and consider any $\sigma\in(\frac12,\sigma_2(s)]$. According to the operators $ \Phi_R^{-1,-b} $ and $ \Phi_{bdr}^{-1,-\b} $ in Lemma \ref{Lemma, neg-kdv-initial} and Lemma \ref{Lemma, neg-kdv-bdr} respectively, we define 
\[\begin{split}
\wt{v}&=\Gamma^-_\b(g,q,b_1,b_2)\\
&:=\eta(t)\Phi^{-1,-\b}_{R}(g,q)+\eta(t)\Phi^{-1,-\b}_{bdr}\Big(b_1-\eta(t)\Phi^{-1,-\b}_{R}(g,q)\big|_{x=0},\, b_2-\,\p_x\big[\eta(t)\Phi^{-1,-\b}_{R}(g,q)\big]\big|_{x=0}\Big).
\end{split}\]
The rest proof can be carried out by a parallel argument as that for Proposition \ref{Prop, pos-kdv} except we apply Lemmas \ref{Lemma, neg-kdv-initial},  \ref{Lemma, neg-kdv-bdr} and \ref{Lemma, neg-kdv-bdr, low-reg} instead of Lemmas \ref{Lemma, pos-kdv-initial}, \ref{Lemma, pos-kdv-bdr} and \ref{Lemma, pos-kdv-bdr, low-reg} respectively.

\end{proof}

\section{Bilinear Estimates}
\label{Sec, bilin est}

This section is intended to prove Proposition \ref{Prop, bilin}. Since $b\leq\frac12$ is required to deal with the boundary integral operators, it brings two issues. First, the bilinear estimate may not be justified in the space $X^{\a,\b}_{s,b}$ via the classical methods as in \cite{Bou93b,KPV93Duke, KPV96}. Secondly, $X^{\a,\b}_{s,b}$ is not a subspace of $C_t\big(\m{R}; H^{s}_{x}\big)$, so living in $X^{\a,\b}_{s,b}$ may not guarantee the continuity of the solution flow. In order to overcome these two barriers, we will take advantage of the modified Fourier restriction spaces as defined in (\ref{MFR norm}) and (\ref{Y norm}). On the other hand, as it is well-known that the resonance function plays an essential role in the bilinear estimates, so we first specify its definition, see e.g. \cite{Tao01}.

\begin{defi}[\hspace{-0.03in}\cite{Tao01}]\label{Def, res fcn}
Let $\big((\a_1,\b_1),(\a_2,\b_2),(\a_3,\b_3)\big)$ be a triple in $((\m{R}\backslash\{0\})\times\m{R})^3$. Define the resonance function $H$ associated to this triple by 
\be\label{res fcn}
H(\xi_1,\xi_2,\xi_3)=\sum_{i=1}^3 \phi^{\a_i,\b_i}(\xi_i),\quad\forall\, (\xi_1,\xi_2,\xi_3)\in\m{R}^3\,\,\text{with}\,\, \sum_{i=1}^3 \xi_i=0,\ee
where $\phi$ is defined as in (\ref{phase fn}).
\end{defi}

Since Proposition \ref{Prop, bilin} contains many inequalities, we split its proof into three parts for the sake of clarity, see the following Proposition \ref{Prop, bilin, gc}--Proposition \ref{Prop, bilin, d2}. 

\begin{pro}\label{Prop, bilin, gc}
Let $-\frac34<s\leq 3$, $\a\neq 0$ and $|\b|\leq 1$. Then there exists $\sigma_0=\sigma_0(s,\a)>\frac12$ such that for any $\sigma\in(\frac12,\sigma_0]$, the following bilinear estimates 
\begin{eqnarray}
\|\p_{x}(w_1 w_2)\|_{X^{\a,\b}_{s,\sigma-1}} &\leq & C \| w_1\|_{X^{\a,\b}_{s,\frac12,\sigma}}\| w_2\|_{X^{\a,\b}_{s,\frac12,\sigma}} \label{bilin, gc}\\
\|\p_{x}(w_1 w_2)\|_{Z^{\a,\b}_{s,\sigma-1}} &\leq & C \| w_1\|_{X^{\a,\b}_{s,\frac12,\sigma}}\| w_2\|_{X^{\a,\b}_{s,\frac12,\sigma}} \label{bilin, gc, Z}
\end{eqnarray}
hold for any $w_1,w_2\in X^{\a,\b}_{s,\frac12,\sigma}$  with some constant $C=C(s,\a,\sigma)$.
\end{pro}

\begin{pro}\label{Prop, bilin, d1}
Let $-\frac34<s\leq 3$, $\a\neq 0$ and $|\b|\leq 1$. Then there exists $\sigma_0=\sigma_0(s,\a)>\frac12$ such that for any $\sigma\in(\frac12,\sigma_0]$, the following bilinear estimates
\begin{eqnarray}
\|\p_{x}(w_1 w_2)\|_{X^{-\a,-\b}_{s,\sigma-1}} &\leq & C \| w_1\|_{X^{\a,\b}_{s,\frac12,\sigma}}\| w_2\|_{X^{\a,\b}_{s,\frac12,\sigma}}  \label{bilin, d1}\\
\|\p_{x}(w_1 w_2)\|_{Z^{-\a,-\b}_{s,\sigma-1}} &\leq &  C \| w_1\|_{X^{\a,\b}_{s,\frac12,\sigma}}\| w_2\|_{X^{\a,\b}_{s,\frac12,\sigma}} \label{bilin, d1, Z}
\end{eqnarray}
hold for any $w_1,w_2\in X^{\a,\b}_{s,\frac12,\sigma}$  with some constant $C=C(s,\a,\sigma)$.
\end{pro}

\begin{pro}\label{Prop, bilin, d2}
Let $-\frac34<s\leq 3$, $\a\neq 0$ and $|\b|\leq 1$. Then there exists $\sigma_0=\sigma_0(s,\a)>\frac12$ such that for any $\sigma\in(\frac12,\sigma_0]$, the following bilinear estimates
\begin{eqnarray}
\|\p_{x}(w_1 w_2)\|_{X^{\a,\b}_{s,\sigma-1}} &\leq & C \| w_1\|_{X^{\a,\b}_{s,\frac12,\sigma}}\| w_2\|_{X^{-\a,-\b}_{s,\frac12,\sigma}} \label{bilin, d2}\\
\|\p_{x}(w_1 w_2)\|_{Z^{\a,\b}_{s,\sigma-1}} &\leq &  C \| w_1\|_{X^{\a,\b}_{s,\frac12,\sigma}}\| w_2\|_{X^{-\a,-\b}_{s,\frac12,\sigma}} \label{bilin, d2, Z}
\end{eqnarray}
hold for any $w_1\in X^{\a,\b}_{s,\frac12,\sigma}$ and $w_2\in X^{-\a,-\b}_{s,\frac12,\sigma}$  with some constant $C=C(s,\a,\sigma)$.
\end{pro}

The proofs for these three propositions  are similar while Proposition  \ref{Prop, bilin, d2} is slightly more challenging since $w_1$ and $w_2$ live in different spaces. As a result, we will only prove Proposition  \ref{Prop, bilin, d2}. Before the proof, we first recall some elementary  technical results.

\begin{lem}\label{Lemma, int in tau}
Let $\rho_1\geq \rho_2>\frac12$ and  $\rho_3\in\m{R}$ satisfy
\[\left\{\begin{array}{lcc}
\rho_3=\rho_2 & \text{if} & \rho_1>1,\\
\rho_3<\rho_2 & \text{if} & \rho_1=1,\\
\rho_3=\rho_1+\rho_2-1 & \text{if} & \frac12<\rho_1<1,
\end{array}\right.\]
then there exists $C=C(\rho_1,\rho_2,\rho_3)$ such that for any $a,b\in\m{R}$, 
\be\label{int in tau}
\int_{-\infty}^{\infty}\frac{dx}{\la x-a \ra^{\rho_{1}} \la -x-b \ra^{\rho_{2}}}\leq \frac{C}{\la a+b\ra^{\rho_{3}}}.\ee
\end{lem}

The proof for this lemma is standard (see e.g. \cite{ET16, KPV96}) and therefore omitted, we just want to remark that $\la a+b\ra=\la(x-a)+(-x-b)\ra$.

\begin{lem}\label{Lemma, bdd int}
If $\rho>\frac{1}{2}$, then there exists $C=C(\rho)$ such that for any $\sigma_{i}\in\m{R}$ $(0\leq i\leq 2)$ with $\sigma_{2}\neq 0$,
\be\label{bdd int for quad}
\int_{-\infty}^{\infty}\frac{dx}{\la \sigma_{2}x^{2}+\sigma_{1}x+\sigma_{0}\ra^{\rho}}\leq \frac{C}{|\sigma_{2}|^{1/2}}.\ee
Similarly, if $\rho>\frac{1}{3}$, then there exists $C=C(\rho)$ such that for any $\sigma_{i}\in\m{R}$ $(0\leq i\leq 3)$ with $\sigma_{3}\neq 0$,
\be\label{bdd int for cubic}
\int_{-\infty}^{\infty}\frac{dx}{\la \sigma_{3}x^{3}+\sigma_{2}x^{2}+\sigma_{1}x+\sigma_{0}\ra^{\rho}}\leq \frac{C}{|\sigma_{3}|^{1/3}}.\ee
\end{lem}

\begin{proof}
We refer the reader to the proof of Lemma 2.5 in \cite{BOP97} where (\ref{bdd int for cubic}) was  proved. The similar argument can also be applied to  obtain (\ref{bdd int for quad}).
\end{proof}

%
%

For the proof of the bilinear estimate, it is usually beneficial to reduce it to an estimate of some weighted convolution of $L^{2}$ functions as pointed out in \cite{Tao01, CKSTT03}. The next lemma illustrates this in a general situation. For the convenience of notations, we denote $\vec{\xi}=(\xi_{1},\xi_{2},\xi_{3})$, $\vec{\tau}=(\tau_{1},\tau_{2},\tau_{3})$ and 
\be\label{int domain}
A:=\Big\{(\vec{\xi},\vec{\tau})\in\m{R}^{6}:\sum_{i=1}^{3}\xi_{i}=\sum_{i=1}^{3}\tau_{i}=0\Big\}.\ee

\begin{lem}\label{Lemma, bilin to weighted l2}
Let 
$(\a_i,\b_i)\in(\m{R}\backslash\{0\})\times\m{R}$ for $1\leq i\leq 3$. 
Let $s\in\m{R}$ and $\sigma>\frac12$. Then the bilinear estimate
\[
\|\p_{x}(w_{1}w_{2})\|_{X^{\a_{3},\b_{3}}_{s,\sigma-1}}\leq C\,\|w_{1}\|_{X^{\a_{1},\b_{1}}_{s,\frac12,\sigma}}\,\|w_{2}\|_{X^{\a_{2},\b_{2}}_{s,\frac12,\sigma}}, \quad\forall\, w_1\in X^{\a_{1},\b_{1}}_{s,\frac12,\sigma},\,\, w_2\in X^{\a_{2},\b_{2}}_{s,\frac12,\sigma}, 
\]
is equivalent to
\[
\int\limits_{A}\frac{\xi_{3}\la\xi_{3}\ra^{s}\prod\limits_{i=1}^{3}f_{i}(\xi_{i},\tau_{i})}{\la\xi_{1}\ra^{s}\la\xi_{2}\ra^{s}M_{1}M_{2}\la L_{3}\ra^{1-\sigma}} \leq C\,\prod_{i=1}^{3}\|f_{i}\|_{L^{2}_{\xi\tau}}, \quad\forall\, f_{i}\in L^2(\m{R}\times\m{R}),\quad i=1,2,3,
\]
where $L_{i}=\tau_{i}-\phi^{\a_{i},\b_{i}}(\xi_{i})$ for $ i=1,2,3$, and 
\[
M_1=\la L_{1}\ra^{\frac12}+\mb{1}_{\{e^{|\xi_1|}\leq 3+|\tau_1|\}}\la L_1\ra^{\sigma}, \quad M_2=\la L_{2}\ra^{\frac12}+\mb{1}_{\{e^{|\xi_2|}\leq 3+|\tau_2|\}}\la L_2\ra^{\sigma}.
\]
Similarly, the bilinear estimate 
\[
\|\p_{x}(w_{1}w_{2})\|_{Z^{\a_{3},\b_{3}}_{s,\sigma-1}}\leq C\,\|w_{1}\|_{X^{\a_{1},\b_{1}}_{s,\frac12,\sigma}}\,\|w_{2}\|_{X^{\a_{2},\b_{2}}_{s,\frac12,\sigma}}, \quad \forall\, w_1\in X^{\a_{1},\b_{1}}_{s,\frac12,\sigma},\,\, w_2\in X^{\a_{2},\b_{2}}_{s,\frac12,\sigma}, 
\]
is equivalent to
\[
\int\limits_{A}\frac{\xi_{3}\la\tau_{3}\ra^{\frac{s}{3}+\frac12-\sigma}\prod\limits_{i=1}^{3}f_{i}(\xi_{i},\tau_{i})}{\la\xi_{1}\ra^{s}\la\xi_{2}\ra^{s}M_{1}M_{2}\la L_{3}\ra^{1-\sigma}} \leq C\,\prod_{i=1}^{3}\|f_{i}\|_{L^{2}_{\xi\tau}}, \quad\forall\, f_{i}\in L^2(\m{R}\times\m{R}),\quad i=1,2,3.
\]
\end{lem}

\begin{proof}
The proof is straightforward by using duality and Plancherel theorem. The details are omitted since the argument is standard, see e.g. \cite{Tao01, CKSTT03}.
\end{proof}

Now we are ready to verify Proposition \ref{Prop, bilin, d2}.

\begin{proof}[\bf Proof of Proposition \ref{Prop, bilin, d2}]

The triple $\big((\a_1,\b_1), (\a_2,\b_2), (\a_3,\b_3)\big)$ associated to the bilinear estimates (\ref{bilin, d2}) and (\ref{bilin, d2, Z}) is 
\[(\a_1,\b_1)=(\a,\b),\quad (\a_2,\b_2)=(-\a,-\b),\quad(\a_3,\b_3)=(\a,\b).\] 
Define the set $ A $ as in (\ref{int domain}). For any $(\vec{\xi},\vec{\tau})\in A$,  let 
\be\label{d2, def of L}
L_1:=\tau_1-\phi^{\a,\b}(\xi_1), \quad L_2:=\tau_2-\phi^{-\a,-\b}(\xi_2), \quad L_3:=\tau_3-\phi^{\a,\b}(\xi_3).\ee
According to Definition \ref{Def, res fcn}, the resonance function $H=H(\xi_1,\xi_2,\xi_3)$ has the property that 
$H=-\sum\limits_{i=1}^{3}L_i$
and has the following formula.
\be\label{d2, res fcn}
H(\xi_1,\xi_2,\xi_3)=-\a\xi_2(2\xi_2^2+3\xi_1\xi_2+3\xi_1^2)+2\b\xi_2, \quad\forall\,\sum_{i=1}^3 \xi_i=0.\ee
If $(\xi_3,\tau_3)$ is fixed, then by substituting  $(\xi_2,\tau_2)=-(\xi_1+\xi_3, \tau_1+\tau_3)$, $L_1+L_2$ can be viewed as a function in $\xi_1$: 
\be\label{d2, P}
L_1+L_2 = P(\xi_1) := -2\a\xi_1^3-3\a\xi_3\xi_1^2+(-3\a\xi_3^2+2\b)\xi_1+\phi^{-\a,-\b}(\xi_3)-\tau_3.\ee
Taking the derivative of $P$ with respect to $\xi_1$ yields
\be\label{d2, P'}
P'(\xi_1)=-3\a(2\xi_1^2+2\xi_3\xi_1+\xi_3^2)+2\b.\ee
Similarly, if $(\xi_1,\tau_1)$ is fixed, then by substituting  $(\xi_3,\tau_3)=-(\xi_1+\xi_2,\tau_1+\tau_2)$, $L_2+L_3$ can be viewed as a function in $\xi_2$: 
\be\label{d2, Q}
L_2+L_3 = Q(\xi_2) := 2\a\xi_2^3+3\a\xi_1\xi_2^2+(3\a\xi_1^2-2\b)\xi_2+\phi^{\a,\b}(\xi_1)-\tau_1.\ee
Taking the derivative of $Q$ with respect to $\xi_2$ yields
\be\label{d2, Q'}
Q'(\xi_2)=3\a(2\xi_2^2+2\xi_1\xi_2+\xi_1^2)-2\b.\ee
Finally, if $(\xi_2,\tau_2)$ is fixed, then by substituting  $(\xi_1,\tau_1)=-(\xi_2+\xi_3,\tau_2+\tau_3)$, $L_3+L_1$ can be viewed as a function in $\xi_3$: 
\be\label{d2, R}
L_3+L_1 = R(\xi_3) := 3\a\xi_2\xi_3^2+3\a\xi_2^2\xi_3+\phi^{\a,\b}(\xi_2)-\tau_2.\ee
Taking the derivative of $R$ with respect to $\xi_3$ yields
\be\label{d2, R'}
R'(\xi_3)=3\a\xi_2(2\xi_3+\xi_2).\ee
These formulas (\ref{d2, P})--(\ref{d2, R'}) will play an important role in this section. 

On the other hand, since $|\b|\leq 1$ and $2\xi_2^2+3\xi_1\xi_2^2+3\xi_1^2\geq (\xi_1^2+\xi_2^2)/2$, it follows from formula (\ref{d2, res fcn}) and the fact $\sum_{i=1}^{3}\xi_i=0$ that 
\be\label{d2, large H}
|H(\xi_1,\xi_2,\xi_3)| \gs_{\a} |\xi_2|\sum_{i=1}^{3}\la\xi_i\ra^2
\ee
as long as $\sum_{i=1}^{3}|\xi_i|\geq C_1$ for some constant $C_1=C_1(\a)$.
Analogously, based on (\ref{d2, P'}) and (\ref{d2, Q'}), there exists a constant $C_2=C_2(\a)$ such that if $\sum_{i=1}^{3}|\xi_i|\geq C_2$, then 
\be\label{d2, large P' and Q'}
|P'(\xi_1)|\gs_{\a} \sum_{i=1}^{3}\la\xi_i\ra^2 \quad \mbox{and}\quad |Q'(\xi_2)|\gs_{\a} \sum_{i=1}^{3}\la\xi_i\ra^2.\ee
Define $C^{*}=\max\{4 ,\,C_1,\,C_2\}$
and fix it in the rest of this section. Next, we will prove (\ref{bilin, d2}) first and then justify (\ref{bilin, d2, Z}). For ease of notations, the dependence of constants on $s,\a,\sigma$ may not be explicitly shown.

\begin{proof}[\bf Proof of (\ref{bilin, d2})]

By Lemma \ref{Lemma, bilin to weighted l2}, it suffices to prove for any $\{f_{i}\}_{i=1,2,3}\subset L^2(\m{R}\times\m{R})$,
\[\int\limits_{A}\frac{\xi_{3}\la\xi_{3}\ra^{s}\prod\limits_{i=1}^{3}f_{i}(\xi_{i},\tau_{i})}{\la\xi_{1}\ra^{s}\la\xi_{2}\ra^{s}M_{1}M_{2}\la L_{3}\ra^{1-\sigma}} \leq C\,\prod_{i=1}^{3}\|f_{i}\|_{L^{2}_{\xi\tau}}, \]
where $A$ is defined as in (\ref{int domain}),
\be\label{d2, modified L}
M_1=\la L_{1}\ra^{\frac12}+\mb{1}_{\{e^{|\xi_1|}\leq 3+|\tau_1|\}}\la L_1\ra^{\sigma} \quad\text{and}\quad M_2=\la L_{2}\ra^{\frac12}+\mb{1}_{\{e^{|\xi_2|}\leq 3+|\tau_2|\}}\la L_2\ra^{\sigma}.\ee
Noticing $\la\xi_3\ra\leq \la\xi_1\ra\la\xi_2\ra$, it suffices to consider the case when $s$ is close to $-\frac34$. Without loss of generality, we assume $-\frac34<s\leq -\frac{9}{16}$ and denote $\rho=-s$. Then $\frac{9}{16}\leq \rho<\frac34$ and it reduces to show

\be\label{d2, weighted l2 form}
\int\limits_{A}\frac{|\xi_{3}|\la\xi_{1}\ra^{\rho}\la\xi_{2}\ra^{\rho}\prod\limits_{i=1}^{3}|f_{i}(\xi_{i},\tau_{i})|}{\la\xi_{3}\ra^{\rho}M_{1}M_{2}\la L_{3}\ra^{1-\sigma}} \leq C\,\prod_{i=1}^{3}\|f_{i}\|_{L^{2}_{\xi\tau}}.\ee
Define 
\be\label{sigma_0}
\sigma_0=\frac{7}{12}-\frac{\rho}{9}.\ee
Then (\ref{d2, weighted l2 form}) will be justified for any $\sigma\in\big(\frac12,\sigma_0\big]$. The choice of $\sigma_0$ in (\ref{sigma_0}) implies
\be\label{sigma_0 is good}
\sigma\leq \min\Big\{\frac{\rho+1}{3},\,\frac{7}{12}-\frac{\rho}{9}\Big\}.
\ee

The remaining proof will be divided into three cases depending on the size of the resonance function $ H $. The first case is when $\sum_{i=1}^{3}|\xi_i|$ is small which does not ganrantee the estimate (\ref{d2, large H}). The second case is when $\sum_{i=1}^{3}|\xi_i|$ is large enough but $|\xi_2|$ is small, so although the estimate (\ref{d2, large H}) holds, the lower bound in (\ref{d2, large H}) may be small. It is this step that forces to use the modified Fourier restriction space. The last case is when $\sum_{i=1}^{3}|\xi_i|$ is large enough and $|\xi_2|$ is not too small, so estimate (\ref{d2, large H}) holds with a large lower bound. In the following, we will show more details.

\begin{itemize}
\item {\bf Case 1:} $\sum_{i=1}^{3}|\xi_i|\leq C^*$. 

Although there is no effective lower bound for $ H $ in this case, the LHS of (\ref{d2, weighted l2 form}) can be greatly simplified by dropping all norms of $ \{\xi_i\}_{i=1}^{3} $ and it suffices to prove 
\be\label{d2, case 1}
\int\frac{\prod\limits_{i=1}^{3}|f_{i}(\xi_{i},\tau_{i})|}{\la L_1\ra^{\frac12} \la L_{2}\ra^{\frac12}\la L_{3}\ra^{1-\sigma}} \leq C\,\prod_{i=1}^{3}\|f_{i}\|_{L^{2}_{\xi\tau}}.
\ee
In the following, we will adopt the Cauchy-Schwarz inequality argument as in \cite{KPV96}.
\begin{align}\label{d2, case 1, C-S}
\text{LHS of (\ref{d2, case 1})} &= \iint\frac{|f_3|}{\la L_3\ra^{1-\sigma}}\bigg(\iint\frac{|f_1f_2|}{\la L_1\ra^{\frac12}\la L_2\ra^{\frac12}}\,d\tau_1 d\xi_1 \bigg)\, d\tau_3 d\xi_3 \notag\\
&\ls \iint\frac{|f_3|}{\la L_3\ra^{1-\sigma}}\bigg(\iint\frac{d\tau_1 d\xi_1 }{\la L_1\ra\la L_2\ra}\bigg)^{\frac12}\bigg(\iint f_1^2 f_2^2 \,d\tau_1 d\xi_1 \bigg)^{\frac12}\,d\tau_3 d\xi_3 . 
\end{align}
If 
\be\label{d2, case 1, bdd 1}
\sup_{\xi_3,\tau_3}\,\frac{1}{\la L_3\ra^{2(1-\sigma)}}\iint\frac{\,d\tau_1 d\xi_1}{\la L_1\ra\la L_2\ra}\leq C,\ee
then it follows from (\ref{d2, case 1, C-S}) and  Cauchy-Schwarz inequality that 
\be\label{d2, case 1, holder}
\text{RHS of (\ref{d2, case 1, C-S})} \ls \iint |f_3|\bigg(\iint f_1^2 f_2^2\,d\tau_1 d\xi_1 \bigg)^{\frac12}\,d\tau_3 d\xi_3 \leq  \prod_{i=1}^{3}\|f_{i}\|_{L^{2}}. \ee
So it remains to verify (\ref{d2, case 1, bdd 1}). 
Noticing
\[ L_1 = \tau_1-\phi^{\a,\b}(\xi_1), \quad L_2=\tau_2-\phi^{-\a,-\b}(\xi_2) = -\tau_1-\tau_3-\phi^{\a,\b}(\xi_1+\xi_3),\]
then it follows from Lemma \ref{Lemma, int in tau} (with $ \rho_1=\rho_2=1 $, $\rho_3=2-2\sigma$) that
\be\label{d2, case 1, bdd 2}
\text{LHS of (\ref{d2, case 1, bdd 1})}\ls \sup_{\xi_3,\tau_3}\,\frac{1}{\la L_3\ra^{2(1-\sigma)}}\int\frac{d\xi_1}{\la L_1+L_2\ra^{2-2\sigma}}.\ee
Recalling (\ref{d2, P}), $L_1+L_2=P(\xi_1)$ is a cubic polynomial in $\xi_1$, so it follows from $2-2\sigma>\frac13$ and Lemma \ref{Lemma, bdd int} that the right hand side of (\ref{d2, case 1, bdd 2}) is bounded.

\item {\bf Case 2:} $\sum_{i=1}^{3}|\xi_i|> C^*$ and $|\xi_2|\leq 1$.

In this case, $|\xi_1|\sim |\xi_3|\geq 1$ since $C^{*}\geq 4$. In addition, thanks to the modified Fourier restriction space, we now have $M_2\geq \la L_2\ra^{\sigma}$. So the proof of (\ref{d2, weighted l2 form}) reduces to show 
\be\label{d2, case 2}
\int\frac{|\xi_{3}|\prod\limits_{i=1}^{3}|f_{i}(\xi_{i},\tau_{i})|}{\la L_1\ra^{\frac12} \la L_{2}\ra^{\sigma}\la L_{3}\ra^{1-\sigma}} \leq C\,\prod_{i=1}^{3}\|f_{i}\|_{L^{2}_{\xi\tau}}.\ee

\begin{itemize}

\item {\bf Case 2.1:} $\la L_1\ra\leq \la L_3\ra$.

The key strategy is to adjust the term $ \la L_1 \ra^{\frac12} $ in the denominator in (\ref{d2, case 2}) to be $ \la L_1 \ra^{\sigma} $ whose power $ \sigma $ is larger than $ \frac12 $. Actually, since $\sigma>\frac12$ and $\frac12+(1-\sigma)=\sigma+(\frac32-2\sigma)$, we deduce
\[\dfrac{1}{\la L_1\ra^{\frac12}\la L_3\ra^{1-\sigma}}\leq \dfrac{1}{\la L_1\ra^{\sigma}\la L_3\ra^{\frac32-2\sigma}}.\]
So in order to prove (\ref{d2, case 2}), it 
suffices to establish 
\[\int\frac{|\xi_3||f_3|}{\la L_3\ra^{\frac32-2\sigma}}\frac{|f_1f_2|}{(\la L_1\ra\la L_2\ra)^{\sigma}}\leq C\prod_{i=1}^{3}\|f_i\|_{L^2}.\]
Analogous to the Cauchy-Schwarz argument in Case 1, it remains to verify
\be\label{d2, case 2.1, bdd 1}
\sup_{\xi_3,\tau_3}\,\frac{\xi_3^2}{\la L_3\ra^{3-4\sigma}}\int\frac{d\tau_1\,d\xi_1}{(\la L_1\ra\la L_2\ra)^{2\sigma}}\leq C.\ee
It then follows from $2\sigma>1$ and Lemma \ref{int in tau} that 
\[\int\frac{d\tau_1\,d\xi_1}{(\la L_1\ra\la L_2\ra)^{2\sigma}}\ls \int\frac{d\xi_1}{(\la P(\xi_1)\ra)^{2\sigma}},\]
where $P(\xi_1)=L_1+L_2$ is as defined in (\ref{d2, P}). Thus, it suffices to establish 
\be\label{d2, case 2.1, bdd 2}
\sup_{\xi_3,\tau_3}\,\frac{\xi_3^2}{\la L_3\ra^{3-4\sigma}}\int\frac{d\xi_1}{\la P(\xi_1)\ra^{2\sigma}}\leq C.\ee
Since $\sum_{i=1}^{3}|\xi_i|>C^*$, it follows from (\ref{d2, large P' and Q'}) that $|P'(\xi_1)|\gs \sum_{i=1}^{3}\xi_i^2$. Therefore, 
\begin{eqnarray*}
\text{LHS of (\ref{d2, case 2.1, bdd 1})} = \sup_{\xi_3,\tau_3}\,\frac{1}{\la L_3\ra^{3-4\sigma}}\int\frac{\xi_3^2}{|P'(\xi_1)|}\frac{|P'(\xi_1)|}{\la P(\xi_1)\ra^{2\sigma}}\,d\xi_1 \ls \int\frac{|P'(\xi_1)|}{\la P(\xi_1)\ra^{2\sigma}}\,d\xi_1\leq C.
\end{eqnarray*}

\item {\bf Case 2.2:} $\la L_3\ra \leq \la L_1\ra$. This case can be verified using similar argument as in Case 2.1.
\end{itemize}

\item {\bf Case 3:} $\sum_{i=1}^{3}|\xi_i|>C^*$ and $|\xi_2|>1$.

In this case, $|\xi_2|\sim \la\xi_2\ra$, so it suffices to prove 
\be\label{d2, case 3}
\int\frac{|\xi_{3}|\,\la\xi_1\ra^{\rho}|\xi_2|^{\rho}\prod\limits_{i=1}^{3}|f_{i}(\xi_{i},\tau_{i})|}{\la\xi_3\ra^{\rho}\la L_1\ra^{\frac12} \la L_{2}\ra^{\frac12}\la L_{3}\ra^{1-\sigma}} \leq C\,\prod_{i=1}^{3}\|f_{i}\|_{L^{2}_{\xi\tau}}.\ee
Define 
\be\label{MAX}
\text{MAX}=\max\{\la L_1\ra, \la L_2\ra, \la L_3\ra\}.\ee
Since $H=-\sum\limits_{i=1}^{3}L_i$ and $\sum_{i=1}^{3}|\xi_i|>C^*$, it follows from (\ref{d2, large H}) that 
$\text{MAX}\gs \la H\ra\gs |\xi_2|\sum_{i=1}^{3}\la\xi_i\ra^2$.
Next, we will further decompose the proof into three cases depending on which $L_i$ equals $\text{MAX}$.

\begin{itemize}
\item {\bf Case 3.1:} $\la L_1\ra=\text{MAX}$.

In this case, $\la L_1\ra = \text{MAX} \gs |\xi_2|\sum_{i=1}^{3}\la\xi_i\ra^2$.
The key strategy is again to adjust the terms $ \la L_1 \ra^{\frac12} $ and $ \la L_2 \ra^{\frac12} $ in the denominator in (\ref{d2, case 3}) to be $ \la L_1 \ra^{\sigma} $ and $ \la L_2 \ra^{\sigma} $  whose power $ \sigma $ is larger than $ \frac12 $. Meanwhile, it follows from $\frac12+\frac12+(1-\sigma)=(2-3\sigma)+\sigma+\sigma$ that
\[\frac{1}{\la L_1\ra^{\frac12} \la L_{2}\ra^{\frac12}\la L_{3}\ra^{1-\sigma}}\leq \frac{1}{\la L_1\ra^{2-3\sigma} (\la L_{2}\ra\la L_{3}\ra)^{\sigma}}.\]
Then similar as before, it suffices to establish
\be\label{d2, case 3.1, bdd 1}
\sup_{\xi_1,\tau_1}\,\frac{\la\xi_1\ra^{2\rho}}{\la L_1\ra^{4-6\sigma}}\int\frac{|\xi_2|^{2\rho}\xi_3^2}{\la\xi_3\ra^{2\rho}\la Q(\xi_2)\ra^{2\sigma}}\,d\xi_2\leq C,\ee
where $Q(\xi_2)=L_2+L_3$ is as defined in (\ref{d2, Q}).
Since $\sum_{i=1}^{3}|\xi_i|>C^*$, then $|Q'(\xi_2)|\gs \sum_{i=1}^{3}\la\xi_i\ra^2$. 
Combining with the relation $\la L_1\ra \gs |\xi_2|\sum_{i=1}^{3}\la\xi_i\ra^2$, we find
\begin{eqnarray*}
\frac{\la\xi_1\ra^{2\rho}|\xi_2|^{2\rho}\xi_3^2}{\la \xi_3\ra^{2\rho}} \ls \la L_1\ra^{2-2\rho}|Q'(\xi_2)|.
\end{eqnarray*}
As a result, 
\begin{align}
\text{LHS of (\ref{d2, case 3.1, bdd 1})} \ls \sup_{\xi_1,\tau_1}\frac{\la L_1\ra^{2-2\rho}}{\la L_1\ra^{4-6\sigma}}\int \frac{|Q'(\xi_2)|}{\la Q(\xi_2)\ra^{2\sigma}}\,d\xi_2 \ls  \sup_{\xi_1,\tau_1}\frac{\la L_1\ra^{2-2\rho}}{\la L_1\ra^{4-6\sigma}}  \label{d2, case 3.1, bdd 2}.
\end{align}
Thanks to (\ref{sigma_0 is good}), the right hand side of 
(\ref{d2, case 3.1, bdd 2}) is bounded.

\item {\bf Case 3.2:} $\la L_2\ra=\text{MAX}$.

Similar to Case 3.1, we have $\la L_2\ra\gs |\xi_2|\sum_{i=1}^{3}\la\xi_i\ra^2$ and it suffices to justify
\be\label{d2, case 3.2, bdd 1}
\sup_{\xi_2,\tau_2}\,\frac{|\xi_2|^{2\rho}}{\la L_2\ra^{4-6\sigma}}\int\frac{\la\xi_1\ra^{2\rho}\xi_3^2}{\la\xi_3\ra^{2\rho}\la R(\xi_3)\ra^{2\sigma}}\,d\xi_3\leq C,\ee
where $R(\xi_3)=L_3+L_1$ is as defined in (\ref{d2, R}). 

If $|2\xi_3+\xi_2|\geq \frac{1}{10}\la\xi_1\ra$, then it follows from (\ref{d2, R'}) that $|R'(\xi_3)|\gs |\xi_2|\la\xi_1\ra$, so
\begin{eqnarray*}
\frac{|\xi_2|^{2\rho}\la\xi_1\ra^{2\rho}\xi_3^2}{\la \xi_3\ra^{2\rho}}  \ls \la L_2\ra^{2-2\rho}|R'(\xi_3)|.
\end{eqnarray*}
Consequently, 
\begin{eqnarray}\label{d2, case 3.2, bdd 2}
\text{LHS of (\ref{d2, case 3.2, bdd 1})} \ls \sup_{\xi_2,\tau_2}\frac{\la L_2\ra^{2-2\rho}}{\la L_2\ra^{4-6\sigma}}\int \frac{|R'(\xi_3)|}{\la R(\xi_3)\ra^{2\sigma}}\,d\xi_3 \ls \sup_{\xi_2,\tau_2}\, \frac{\la L_2\ra^{2-2\rho}}{\la L_2\ra^{4-6\sigma}}\leq C,
\end{eqnarray}
where the last inequality is due to (\ref{sigma_0 is good}).
If $|2\xi_3+\xi_2|\leq \frac{1}{10}\la\xi_1\ra$, then $|\xi_1|\sim |\xi_2|\sim |\xi_3|$ and
\[\frac{|\xi_2|^{2\rho}\la\xi_1\ra^{2\rho}\xi_3^2}{\la\xi_3\ra^{2\rho}} \ls |\xi_2|^{\frac12}\la L_2\ra^{\frac{2\rho}{3}+\frac12}.\]
As a result, 
\begin{eqnarray}\label{d2, case 3.2, bdd 3}
\text{LHS of (\ref{d2, case 3.2, bdd 1})} &\ls & \sup_{\xi_2,\tau_2}\,\frac{|\xi_2|^{\frac12}\la L_2\ra^{\frac{2\rho}{3}+\frac12}}{\la L_2\ra^{4-6\sigma}}\int\frac{d\xi_3}{\la R(\xi_3)\ra^{2\sigma}}.
\end{eqnarray}
It then follows from (\ref{d2, R}) and Lemma \ref{Lemma, bdd int} and  that 
$\int\frac{d\xi_3}{\la R(\xi_3)\ra^{2\sigma}}\ls |\xi_2|^{-\frac12}$,
so 
\[\text{RHS of (\ref{d2, case 3.2, bdd 3})}\ls \sup_{\xi_2,\tau_2}\,\frac{\la L_2\ra^{\frac{2\rho}{3}+\frac12}}{\la L_2\ra^{4-6\sigma}}\leq C,\]
where the last inequality is due to (\ref{sigma_0 is good}).

\item {\bf Case 3.3:} $\la L_3\ra=\text{MAX}$. This case can be demonstrated via parallel argument as in Case 3.1.

\end{itemize}

\end{itemize}

\end{proof}

\begin{proof}[\bf Proof of (\ref{bilin, d2, Z})]

Thanks to Lemma \ref{Lemma, bilin to weighted l2}, it suffices to prove 
\be\label{d2, Z, weighted l2 form}
\int\limits_{A}\frac{|\xi_{3}|\la\tau_{3}\ra^{\frac{s}{3}+\frac12-\sigma}\prod\limits_{i=1}^{3}|f_{i}(\xi_{i},\tau_{i})|}{\la\xi_{1}\ra^{s}\la\xi_{2}\ra^{s}M_{1}M_{2}\la L_{3}\ra^{1-\sigma}} \leq C\,\prod_{i=1}^{3}\|f_{i}\|_{L^{2}_{\xi\tau}},\ee
where $M_1$, $M_2$ are as defined in (\ref{d2, modified L}). It suffices to consider the case $-\frac34<s\leq -\frac{9}{16} $ and the case $s=3$ since 
\[\bigg(\frac{\la\tau_3\ra^{\frac13}}{\la\xi_1\ra\la\xi_2\ra}\bigg)^{s}\leq \bigg(\frac{\la\tau_3\ra^{\frac13}}{\la\xi_1\ra\la\xi_2\ra}\bigg)^{-\frac{9}{16}}+\bigg(\frac{\la\tau_3\ra^{\frac13}}{\la\xi_1\ra\la\xi_2\ra}\bigg)^{3}\]
for any $-\frac{9}{16}<s<3$.

For the case of  $-\frac34<s\leq -\frac{9}{16}$, let $\rho=-s$. Then $\frac{9}{16}\leq \rho<\frac34$ and (\ref{d2, Z, weighted l2 form}) becomes 
\be\label{d2, Z, weighted l2 form, small s}
\int\limits_{A}\frac{|\xi_{3}|\la\xi_{1}\ra^{\rho}\la\xi_{2}\ra^{\rho}\prod\limits_{i=1}^{3}|f_{i}(\xi_{i},\tau_{i})|}{\la\tau_{3}\ra^{\frac{\rho}{3}+\sigma-\frac12}M_{1}M_{2}\la L_{3}\ra^{1-\sigma}} \leq C\,\prod_{i=1}^{3}\|f_{i}\|_{L^{2}_{\xi\tau}},\ee
Denote $\sigma_0=\frac{7}{12}-\frac{\rho}{9}$ as in (\ref{sigma_0}). Then (\ref{d2, Z, weighted l2 form, small s}) will be proved for any $\sigma\in\big(\frac12,\sigma_0\big]$. Noticing that this choice of $\sigma_0$ not only implies (\ref{sigma_0 is good}) but also guarantees
\be\label{sigma_0 is good in d2, Z}
\sigma\leq \frac23-\frac{2\rho}{9}.\ee
Firstly, in the region where $\la\tau_3\ra\gs \la\xi_3\ra^3$, we have 
\[\frac{1}{\la\tau\ra^{\frac{\rho}{3}+\sigma-\frac12}}\leq \frac{1}{\la\tau\ra^{\frac{\rho}{3}}}\ls \frac{1}{\la\xi_3\ra^{\rho}},\]
so (\ref{d2, Z, weighted l2 form, small s}) holds as a corollary of (\ref{d2, weighted l2 form}). Then it suffices to consider the region where $\la\tau_3\ra\ll\la\xi_3\ra^3$. Because of this simplification, we can assume 
\be\label{d2, Z, simp}
|\xi_3|\gg 1,\quad |\xi_3|>C^*,\quad |L_3|=|\tau_3-\phi^{\a,\b}(\xi_3)|\sim |\xi_3|^3.\ee
Consequently, it follows from (\ref{d2, large H}) and (\ref{d2, large P' and Q'}) that
\be\label{d2, Z, large H, P' and Q'}
|H|\sim |\xi_2|\sum_{i=1}^{3}\xi_i^2, \quad |P'(\xi_1)|\gs \sum_{i=1}^{3}\xi_i^2, \quad |Q'(\xi_2)|\gs \sum_{i=1}^{3}\xi_i^2.\ee
Thanks to (\ref{d2, Z, simp}) and (\ref{d2, Z, large H, P' and Q'}), we will actually prove (\ref{d2, Z, reduced int est}) which is a stronger estimate than (\ref{d2, Z, weighted l2 form, small s}).
\be\label{d2, Z, reduced int est}
\int\frac{|\xi_{3}|\la\xi_{1}\ra^{\rho}\la\xi_{2}\ra^{\rho}\prod\limits_{i=1}^{3}|f_{i}(\xi_{i},\tau_{i})|}{\la L_{1}\ra^{\frac12}\la L_{2}\ra^{\frac12}\la L_{3}\ra^{1-\sigma}} \leq C\,\prod_{i=1}^{3}\|f_{i}\|_{L^{2}_{\xi\tau}}.
\ee
Since $\la L_3\ra\sim |\xi_3|^3$, the justification of (\ref{d2, Z, reduced int est}) depends essentially on how large $ |\xi_3| $ is. By dividing the integral region into two parts: $|\xi_3|\gs \min\{|\xi_1|, |\xi_2|\}$ and $|\xi_3|\ll \min\{|\xi_1|, |\xi_2|\}$, we can establish (\ref{d2, Z, reduced int est}) via similar argument as that in the proof of (\ref{bilin, d2}).

For  the case of  $s=3$,   (\ref{d2, Z, weighted l2 form}) becomes 
\be\label{d2, Z, weighted l2 form, s=3}
\int\limits_{A}\frac{|\xi_{3}|\la\tau_{3}\ra^{\frac32-\sigma}\prod\limits_{i=1}^{3}|f_{i}(\xi_{i},\tau_{i})|}{\la\xi_{1}\ra^{3}\la\xi_{2}\ra^{3}M_{1}M_{2}\la L_{3}\ra^{1-\sigma}} \leq C\,\prod_{i=1}^{3}\|f_{i}\|_{L^{2}_{\xi\tau}},\ee
where 
\[M_1=\la L_{1}\ra^{\frac12}+\mb{1}_{\{e^{|\xi_1|}\leq 3+|\tau_1|\}}\la L_1\ra^{\sigma} \quad\text{ and} \quad M_2=\la L_{2}\ra^{\frac12}+\mb{1}_{\{e^{|\xi_2|}\leq 3+|\tau_2|\}}\la L_2\ra^{\sigma}\]
are defined as in (\ref{d2, modified L}). Define 
\[\sigma_0=\frac{7}{12}-\frac{1}{9}\cdot\frac{9}{16}=\frac{25}{48}.\]
Then (\ref{d2, weighted l2 form}) can be verified for $\rho=\frac{9}{16}$ (i.e. $s=-\frac{9}{16}$), and therefore also holds for $\rho=-3$ (i.e. $s=3$). Taking advantage of this result, we will justify (\ref{d2, Z, weighted l2 form, s=3}) for any $\sigma\in\big(\frac12,\frac{25}{48}\big]$.
Firstly, if $\la\tau_3\ra\ls \la\xi_3\ra^3$, then (\ref{d2, Z, weighted l2 form, s=3}) holds as a corollary of (\ref{d2, weighted l2 form}), so we assume $\la\tau_3\ra\gg \la\xi_3\ra^3$ in the following. In particular, this assumption implies $\la\tau_3\ra\sim \la L_3\ra$. Then (\ref{d2, Z, weighted l2 form, s=3}) reduces to 
\be\label{d2, Z, s=3}
\int\frac{|\xi_{3}|\la\tau_{3}\ra^{\frac12}\prod\limits_{i=1}^{3}|f_{i}(\xi_{i},\tau_{i})|}{\la\xi_{1}\ra^{3}\la\xi_{2}\ra^{3}M_{1}M_{2}} \leq C\,\prod_{i=1}^{3}\|f_{i}\|_{L^{2}_{\xi\tau}}.\ee
The key observation in the rest proof is 
\be\label{d2, Z, s=3, key}
\la\xi_1\ra^3+\la\xi_2\ra^3+\la L_1\ra+\la L_2\ra\geq \la\tau_3\ra.\ee
Based on this observation, we divide the proof into two cases.
\begin{itemize}
\item Case 1: $\la\xi_1\ra^3\gs |\tau_3|$ or $\la\xi_2\ra^3\gs |\tau_3|$. 

We will only prove for the case $\la\xi_1\ra^3\gs |\tau_3|$ since the other case is similar. Under this assumption, we have  $\la\tau_3\ra^{1/2}\ls \la\xi_1\ra^{3/2}$. Moreover, since $|\xi_3|\ls \la\xi_1\ra\la\xi_2\ra$, it suffices to show 
\be\label{d2, Z, s=3, Case 1, bdd 1}
\int\frac{\prod\limits_{i=1}^{3}|f_{i}(\xi_{i},\tau_{i})|}{\la\xi_{2}\ra^{2}\la L_1\ra^{\frac12}\la L_2\ra^{\frac12}} \leq C\,\prod_{i=1}^{3}\|f_{i}\|_{L^{2}_{\xi\tau}}\ee
in order to prove (\ref{d2, Z, s=3}). Similar as before, it remains to establish 
\be\label{d2, Z, s=3, Case 1, bdd 2}
\sup_{\xi_3,\tau_3}\iint\frac{d\tau_2\,d\xi_2}{\la\xi_2\ra^{4}\la L_1\ra\la L_2\ra}\leq C.\ee
By Lemma \ref{Lemma, int in tau}, 
$\int\frac{1}{\la L_1\ra\la L_2\ra} \,d\tau_2 \ls {\la L_1+L_2\ra}^{-\frac23}\leq 1$,
so 
\[\text{LHS of (\ref{d2, Z, s=3, Case 1, bdd 2})}\ls \sup_{\xi_3,\tau_3}\int\frac{d\xi_2}{\la\xi_2\ra^4}\ls 1.\]

\item Case 2: $\la\xi_1\ra^3\ll |\tau_3|$, $\la\xi_2\ra^3\ll |\tau_3|$ and $\la L_1\ra\geq \la L_2\ra$.

The assumptions in this case and the key observation (\ref{d2, Z, s=3, key}) together imply 
\be\label{d2, Z, s=3, Case 2, large L1}
\la L_1\ra\geq \frac{1}{4}\la\tau_3\ra,\quad  |\tau_1|\sim \la L_1\ra  \quad\text{and}\quad |\tau_1|\gg \la\xi_1\ra^3.\ee

\begin{itemize}
\item Case 2.1: $e^{|\xi_1|}\leq 3+|\tau_1|$. 
Thanks to the definition of $M_1$, this case implies 
$M_1\geq \la L_1\ra^{\sigma}\gs |\tau_3|^{\frac12}\la L_2\ra^{\sigma-\frac12}$.
So 
\[\text{LHS of (\ref{d2, Z, s=3})}\ls \int\frac{|\xi_{3}|\prod\limits_{i=1}^{3}|f_{i}(\xi_{i},\tau_{i})|}{\la\xi_{1}\ra^{3}\la\xi_{2}\ra^{3}\la L_2\ra^{\sigma}}.\]
Since $|\xi_3|\ls \la\xi_1\ra\la\xi_2\ra$, it then remains to establish 
$\sup\limits_{\xi_3,\tau_3} \iint\frac{1}{\la\xi_2\ra^4\la L_2\ra^{2\sigma}}\, d\tau_2\,d\xi_2  \leq C$,
which is obvious due to $\sigma>\frac12$.

\item Case 2.2: $e^{|\xi_1|}> 3+|\tau_1|$. 
In this case, $M_1=\la L_1\ra^{\frac12}\gs |\tau_3|^{\frac12}$ and $|\xi_1|>\ln(3+|\tau_1|)$. In addition, because of (\ref{d2, Z, s=3, Case 2, large L1}), we have 
\be\label{d2, Z, Case 2.2, large xi_1}
|\xi_1|\gs \ln(3+|L_1|)\geq \ln(3+|L_2|).\ee
As a result, 
\[\frac{|\xi_3|\la\tau_3\ra^{\frac12}}{\la\xi_{1}\ra^{3}\la\xi_{2}\ra^{3}M_1}\ls \frac{1}{\la\xi_{1}\ra^{2}\la\xi_{2}\ra^{2}}\ls \frac{1}{\big[\ln(3+|L_2|)\big]^2\la\xi_{2}\ra^{2}}.\]
Hence, 
\[\text{LHS of (\ref{d2, Z, s=3})} \ls \int\frac{\prod\limits_{i=1}^{3}|f_{i}(\xi_{i},\tau_{i})|}{\la\xi_{2}\ra^{2}\la L_2\ra^{\frac12}\big[\ln(3+|L_2|)\big]^2}.\]
Similar as before, it remains to show 
\be\label{d2, Z, case 2.2, bdd 1}
\sup_{\xi_3,\tau_3}\iint \frac{d\tau_2\,d\xi_2}{\la\xi_{2}\ra^{4}\la L_2\ra\big[\ln(3+|L_2|)\big]^4}\leq C.\ee
Noticing 
\[\i_{\m{R}}\frac{d\tau_2}{\la L_2\ra\big[\ln(3+|L_2|)\big]^4}= 2\i_{0}^{\infty}\frac{dx}{(1+x)\big[\ln(3+x)\big]^4}<\infty,\]
so 
\[\text{LHS of (\ref{d2, Z, case 2.2, bdd 1})}\ls \int\frac{d\xi_2}{\la\xi_2\ra^4}\leq C.\]
\end{itemize}

\end{itemize}
\end{proof}

Thus, the proof of Proposition \ref{Prop, bilin, d2} is finished.

\end{proof}

\section{Well-posedness}
\label{Sec, wp}

This section is devoted to the verification of Theorem \ref{Thm, main-scaling}. Equivalently, we will justify Theorem \ref{Thm, cwp}, Theorem \ref{Thm, exist of ms} and Theorem \ref{Thm, uniq of ms}.

\begin{proof}[{\bf Proof of Theorem \ref{Thm, cwp}}]
Without loss of generality, we assume $T=1$. Let  $s\in(-\frac34,3]$ and $\b\in(0,1]$ be given. Define 
\[\sigma=\min\big\{\sigma_1(s), \sigma_2(s), \sigma_0(s,1), \sigma_0(s,-1)\big\},\]
where $\sigma_1(s)$ and $\sigma_2(s)$ are the thresholds as in Proposition  \ref{Prop, pos-kdv} and Proposition  \ref{Prop, neg-kdv}, and $\sigma_0(s,1)$, $\sigma_0(s,-1)$ are the thresholds as in Proposition  \ref{Prop, bilin} when $\a=1$ or $-1$. Then $\sigma>\frac12$ and $\sigma$ only depends on $s$, so in the following, the dependence of any constant on $\sigma$ will be considered as the dependence on $s$. The choice of $r$ will be determined later, see (\ref{main pf, eps}).

Define the space $\mcal{Y}$ to be $Y^{1,\b}_{s,\frac12,\sigma} (\Omega _1) \times Y^{-1,-\b}_{s,\frac12,\sigma} (\Omega _1) $ equipped with the product norm. Denote 
\[E_{0}=\|(p,q)\|_{\mcal{H}^s_x(\m{R}^+)}+\|(a,b,c)\|_{\mcal{H}^s_t(\m{R}^+)}.\]
Then it follows from the assumption that $E_0\leq r$. Define
$\mcal{B}_{C^*} = \big\{(u,v)\in\mcal{Y}: \|(u,v)\|_{\mcal{Y}}\leq C^{*}\big\}$.
We will choose suitable $C^{*}$ and $r$ to guarantee the existence of a solution in the space $\mcal{B}_{C^*}$. For any $(u,v)\in\mcal{B}_{C^*}$, denote 
\be\label{nonlin term}
f(u,v) =-3uu_x-(uv)_x+vv_x, \quad g(u,v) =uu_x-(uv)_x-3vv_x.
\ee
Then we infer from Proposition  \ref{Prop, bilin} that $f\in X^{1,\b}_{s,\sigma-1}(\Omega _1)\cap Z^{1,\b}_{s,\sigma-1}(\Omega _1)$ and $g\in X^{-1,-\b}_{s,\sigma-1}(\Omega _1)\cap Z^{-1,-\b}_{s,\sigma-1}(\Omega _1)$. 

Since the data $(p,q)$ and $(a,b,c)$ are compatible for (\ref{ckdv-2}), then it implies $p$ and $a$ are compatible for (\ref{lin-u}). So according to Proposition \ref{Prop, pos-kdv}, we define $\t{u}=\Gamma^{+}_{\b}(f,p,a)$. By the properties of $\Gamma^{+}_{\b}(f,p,a)$ stated in Proposition  \ref{Prop, pos-kdv}, $\t{u}_{x}(0,t)$ is well-defined and belongs to $H^{\frac{s}{3}}_{t}(\m{R})$. Then again the compatibility of $(p,q)$ and $(a,b,c)$ for (\ref{ckdv-2}) implies the compatibility of $q$ and $b,c+\t{u}_{x}|_{x=0}$. So based on Proposition  \ref{Prop, neg-kdv}, we define $\t{v}=\Gamma^{-}_{\b}(g,q,b,c+\t{u}_{x}|_{x=0})$. Combining $\Gamma^+_{\b}$ and $\Gamma^-_{\b}$ together, we define $\Gamma(u,v)=(\t{u},\t{v})$, where
\be\label{def of Gamma}
\t{u} =\Gamma^+_\b(f,p,a), \quad \t{v} = \Gamma^-_\b\big(g,q,b,c+\t{u}_{x}|_{x=0}\big).
\ee
We will prove $\Gamma$ is a contraction mapping in $\mcal{B}_{C^*}$ so that its fixed point $(u,v)$ are the desired functions in Theorem \ref{Thm, cwp}.

It will be shown first that $\Gamma$ maps the closed ball $\mcal{B}_{C^*}$ into itself for suitable $ C^* $. For any $(u,v)\in \mcal{B}_{C^*}$, it follows from (\ref{est for pos lin kdv}) in Proposition \ref{Prop, pos-kdv} and (\ref{est for neg lin kdv}) in Proposition \ref{Prop, neg-kdv} that 
\begin{align*}
\|\t{u}\|_{Y^{1,\b}_{s,\frac12,\sigma} (\Omega _1)} &\leq C_1\Big(\|f\|_{X^{1,\b}_{s,\sigma-1} (\Omega _1)} +\|f\|_{Z^{1,\b}_{s,\sigma-1} (\Omega _1)} +\|p\|_{H^{s}(\R^+)}+\|a\|_{H^{\frac{s+1}{3}}(\R^+)}\Big),\\
\|\t{v}\|_{Y^{-1,-\b}_{s,\frac12,\sigma} (\Omega _1)} &\leq C_1\Big(\|g\|_{X^{-1,-\b}_{s,\sigma-1} (\Omega _1)} + \|g\|_{Z^{-1,-\b}_{s,\sigma-1} (\Omega _1)}  + \|q\|_{H^{s}(\R^+)} \\
& \qquad\quad +\|b\|_{H^{\frac{s+1}{3}}(\R^+)}+\|c\|_{H^{\frac{s}{3}}(\R^+)}+\|\t{u}_{x}|_{x=0}\|_{H^{\frac{s}{3}}(\R^+)}\Big),
\end{align*}
where $ C_1 $ is a constant that only depends on $ s $. Adding these two together and recalling the definition of $E_0$, we obtain 
\be\label{norm-est-1}
\|(\t{u},\t{v})\|_{\mcal{Y}}\leq C_{1}\Big(E_0+\|f\|_{X^{1,\b}_{s,\sigma-1} (\Omega _1)}+\|f\|_{Z^{1,\b}_{s,\sigma-1} (\Omega _1)}+\|g\|_{X^{-1,-\b}_{s,\sigma-1} (\Omega _1)}+\|g\|_{Z^{-1,-\b}_{s,\sigma-1} (\Omega _1)}+\|\t{u}_{x}|_{x=0}\|_{H^{\frac{s}{3}}(\R^+)}\Big).\ee
By the second estimate (\ref{trace est for pos lin kdv}) in Proposition \ref{Prop, pos-kdv}, 
\[\|\t{u}_{x}|_{x=0}\|_{H^{\frac{s}{3}}(\m{R}^+)}\leq C_1\Big(\|f\|_{X^{1,\b}_{s,\sigma-1} (\Omega _1)}+\|f\|_{Z^{1,\b}_{s,\sigma-1} (\Omega _1)}+\|p\|_{H^{s}(\R^+)}+\|a\|_{H^{\frac{s+1}{3}}(\R^+)}\Big).\]
Plugging this estimate into (\ref{norm-est-1}) leads to
\be\label{norm-est-2}
\|(\t{u},\t{v})\|_{\mcal{Y}}\leq C_{1}(C_{1}+1)\Big(E_0+\|f\|_{X^{1,\b}_{s,\sigma-1} (\Omega _1)}+\|f\|_{Z^{1,\b}_{s,\sigma-1} (\Omega _1)}+\|g\|_{X^{-1,-\b}_{s,\sigma-1} (\Omega _1)}+\|g\|_{Z^{-1,-\b}_{s,\sigma-1} (\Omega _1)}\Big).\ee
Since $f$ and $g$ are defined as in (\ref{nonlin term}), we apply Proposition  \ref{Prop, bilin} to conclude that
\begin{align*}
\|f\|_{X^{1,\b}_{s,\sigma-1} (\Omega _1)}+\|f\|_{Z^{1,\b}_{s,\sigma-1} (\Omega _1)}\leq C_2\big(\|u\|_{X^{1,\b}_{s,\frac12,\sigma} (\Omega _1)}+\|v\|_{X^{-1,-\b}_{s,\frac12,\sigma} (\Omega _1)}\big)^2,\\
\|g\|_{X^{-1,-\b}_{s,\sigma-1} (\Omega _1)}+\|g\|_{Z^{-1,-\b}_{s,\sigma-1} (\Omega _1)}\leq C_2\big(\|u\|_{X^{1,\b}_{s,\frac12,\sigma} (\Omega _1)}+\|v\|_{X^{-1,-\b}_{s,\frac12,\sigma} (\Omega _1)}\big)^2,
\end{align*}
where $ C_2 $ is a constant which only depends on $ s $. Putting these two estimates into (\ref{norm-est-2}), we find 
\[\|(\t{u},\t{v})\|_{\mcal{Y}}\leq C_{3}\Big[E_0+\big(\|u\|_{X^{1,\b}_{s,\frac12,\sigma} (\Omega _1)}+\|v\|_{X^{-1,-\b}_{s,\frac12,\sigma} (\Omega _1)}\big)^2\Big]\leq C_{3}\big(E_0+\|(u,v)\|_{\mcal{Y}}^2\big),\]
where the constant $ C_3 $ only depends on $ s $. Since $(u,v)\in \mcal{B}_{C^*}$, $\|(u,v)\|_{\mcal{Y}}\leq C^{*}$. Hence,
\[\|(\t{u},\t{v})\|_{\mcal{Y}}\leq C_{3}\big[E_0+(C^*)^2\big]\leq C_{3}\big[r + (C^*)^2\big].\]
Choosing $C^* = 8 C_3 r$,  then 
\be\label{norm-est-3}
\|(\t{u},\t{v})\|_{\mcal{Y}}\leq C_3 r + 64 C_3^3 r^2.\ee
Define 
\be\label{main pf, eps}
r = \frac{1}{64 C_3^2}.\ee
Then $r$ only depends on $s$ and it follows from (\ref{norm-est-3}) that
$\|(\t{u},\t{v})\|_{\mcal{Y}}\leq C^{*}/4$, which implies $(\t{u},\t{v})\in \mcal{B}_{C^*}$.

Next for any $(u_j, v_j)\in \mcal{B}_{C^*}$, $j=1,2$, similar argument as above yields 
\[ \| \Gamma (u_1, v_1 )- \Gamma (u_2 , v_2)\|_{\mcal{Y}} \leq \frac12 \| (u_1, v_1)-(u_2, v_2)\|_{\mcal{Y}}.\]
We  have thus shown  that $\Gamma$ is a contraction on $\mcal{B}_{C^*}$, which implies $\Gamma$ has a fixed point
$(u,v)\in \mcal{B}_{C^*}$. By definition of $\Gamma$, $(u,v)$ satisfies 
\[\left\{\begin{aligned}
u &=\Gamma^{+}_{\b}(f,p,a),\\
v &=\Gamma^{-}_{\b}\big(g,q,b,c+u_{x}|_{x=0}\big),
\end{aligned}\right.\]
where $f$ and $g$ are defined as in (\ref{nonlin term}). 
Taking advantage of Proposition \ref{Prop, pos-kdv} and Proposition \ref{Prop, neg-kdv}, we conclude that 
$u\in Y^{1,\b}_{s,\frac12,\sigma} (\Omega _1) \cap C_{x}^{j}\big(\m{R}^+_0; H_{t}^{\frac{s+1-j}{3}}(\m{R}^+)\big)$ and $v\in Y^{-1,-\b}_{s,\frac12,\sigma} (\Omega _1) \cap C_{x}^{j}\big(\m{R}^+_0; H_{t}^{\frac{s+1-j}{3}}(\m{R}^+)\big)$ for $j=0,1$. Hence, the proof of Theorem \ref{Thm, cwp} is completed.
\end{proof}

Now we turn to  consider the unconditional well-posedness for the IBVP (\ref{ckdv-2}) and prove Theorem \ref{Thm, exist of ms} and Theorem \ref{Thm, uniq of ms} . 
Note first that by scaling argument,   Theorem \ref{Thm, cwp}  can be restated as the following result.

\begin{cor}[Conditional Well-posedness]  \label{cwp-1}
Let $-\frac34 < s\leq 3$, $r>0$  and $0\leq  \beta \leq 1$ be given.  There exist $ \sigma = \sigma(s) > \frac12 $ and  $T = T(s,r) >0$  such that for any naturally compatible  $(\phi, \psi) \in {\cal H}^s _x(\R^+)$ and $\vec{h}= (h_1, h_2, h_3 ) \in  {\cal H}^s_t (\R^+) $ related to the IBVP (\ref{ckdv-2})  with
\[ \|\phi , \psi )\|_{{\cal H}^s_x (\R^+)} + \|\vec{h} \|_{{\cal H}^s_t (\R^+)}  \le r , \]
the system of the integral equations (SIE)  (\ref{ckdv-3}) admits a unique solution  
\[ (u,v) \in \mcal{Y}_{\sigma} (\Omega _T) :=Y^{1,\b}_{s,\frac12,\sigma} (\Omega _T) \times Y^{-1,-\b}_{s,\frac12,\sigma}(\Omega _T).\]
Moreover, the solution map is real analytic  in the corresponding spaces.
\end{cor}

%

By standard extension arguments, the solutions  given in Corollary \ref{cwp-1} possess  a  blow-up alternative property as described below.

\begin{lem}[Blow-up alternative] \label{blowup}
Let $-\frac34 <s\leq 3$ be given. For any given
$$(p, q )\in \H_x^s (\R^+), \quad (a,b,c)\in \H_t^s(\R^+), $$ 
there exists a $T^s_{max} >0  $ such that the corresponding solution $(u, v)\in C(0, T^s_{max}; H^s (\R^+))$ and 
\[ \mbox{ either \quad  $T^s_{max} =+\infty $  \quad  or  \quad }
\lim _{t\to T^s_{\max}} \l(\| u(\cdot, t)\|_{H^s (\R^+)}+\|  v(\cdot, t)\|_{H^s (\R^+)}\r) = +\infty .\]
\end{lem}

We then intend to prove Theorem \ref{Thm, exist of ms}, that is to show the solution $(u, v)$ given in Corollary \ref{cwp-1} is a mild solution of the IBVP \eqref{ckdv-1}. According to Definition \ref{Def, m-s}, mild solutions can be regarded as limits of classical solutions. So we first recall a well-known result about classical solutions.

\begin{lem}[Existence and uniqueness of classical  solutions]\label{wp-classical}
Let $r>0$ be given.  There exists a $T>0 $ such that for compatible set
$(p, q )\in \H_x^3 (\R^+)$ and $(a,b,c)\in \H_t^3(\R^+) $
with
\[ \|  ((p, q ))\| _{ \H_x^3 (\R^+)} +\|(a,b,c)\|_{\H_t^3(\R^+) } \leq r, \]
the  IBVP \eqref{ckdv-1}  admits a unique strong solution $ (u,v) $ such that both $ u $ and $ v $ belong to $ C^2 (0,T; L^2 (\R^+))\cap C(0,T; H^3 (\R^+))$.
\end{lem}

\begin{proof}  
The existence of the classical solution $(u, v)$ follows  directly from Corollary \ref{cwp-1} with $s=3$,  and the uniqueness of the  solution can be proved using the standard energy estimate method.
\end{proof}

Based on the notations in Lemma \ref{blowup}, it can be readily  checked that $T_{max}^{s_1}\geq T_{max}^{s_2}$ for $s_1< s_2$.  We  then propose that $T_{max}^{s_1}=T_{max}^{s_2}$.

\begin{pro}[Persistence of regularity] \label{Prop, Tmax} 
For  $-\frac34 < s_1< s_2 \leq 3$, $(p, q )\in \H_x^{s_2} (\R^+)$ and $(a,b,c)\in \H_t^{s_2}(\R^+)$, one has $T^{s_2}_{max} = T^{s_1}_{max}$,
where $ T^{s_2}_{max} $ and $ T^{s_1}_{max} $ are the lifespans of the solutions, corresponding to $ s_2 $ and $ s_1 $ respectively, determined in Lemma \ref{blowup}.
\end{pro}

Before stating the proof of Proposition \ref{Prop, Tmax}, we list some smoothing properties based on the linear estimate in Lemma \ref{Lemma, lin est} and the bilinear estimates in Proposition \ref{Prop, bilin}, the proofs of the following Lemma \ref{Lemma, bilin, gc1}--Lemma \ref{Lemma, bilin, d21} will be given in Appendix \ref{Sec,  proof of bilin smoothing}.

\begin{lem}\label{Lemma, bilin, gc1}
Let $-\frac34<s_1<s_2\leq 3$, $0<T\leq 1$, $\a\neq 0$ and $|\b|\leq 1$. Then there exists $\epsilon_0 = \eps(s_1, s_2, \a) >0$ such that for any $\sigma\in(\frac12, \frac12 + \epsilon_0]$, the following bilinear estimate
\begin{align}
&\quad\, \l\| \eta\Big(\frac{t}{T}\Big) \int^t_0 W_R^{\a, \b} (t-\tau) (w_1 w_2)_x\,d\tau\r\|_{X^{\a,\b}_{s_2,\frac12, \si}} \nonumber\\
& \leq  CT^{\epsilon_0}	\Big( \| w_1\|_{X^{\a,\b}_{s_1,\frac12,\sigma}} \| w_2\|_{X^{\a,\b}_{s_2,\frac12,\sigma}}+\| w_1\|_{X^{\a,\b}_{s_2,\frac12,\sigma}}\| w_2\|_{X^{\a,\b}_{s_1,\frac12,\sigma}}\Big).  \label{bilin, gc1}
\end{align}
holds for any $w_1,w_2\in X^{\a,\b}_{s,\frac12,\sigma}$ with some positive constant $ C = C(s_1, s_2, \a, \eps_0, \sigma)$.
\end{lem}

\begin{lem}\label{Lemma, bilin, d11}
Let $-\frac34<s_1<s_2\leq 3$, $0<T\leq 1$, $\a\neq 0$ and $|\b|\leq 1$. Then there exists $\epsilon_0 = \eps(s_1,s_2,\a) >0$ such that for any $\sigma\in(\frac12, \frac12+\epsilon_0]$, the following bilinear estimate
\begin{align}
&\quad\, \l\| \eta\Big(\frac{t}{T}\Big)  \int^t_0 W_R^{-\a, -\b} (t-\tau) (w_1 w_2)_x\,d\tau\r\|_{X^{-\a,-\b}_{s_2,\frac12, \si}} \nonumber \\
& \leq  CT^{\epsilon_0}	\Big( \| w_1\|_{X^{\a,\b}_{s_1,\frac12,\sigma}} \| w_2\|_{X^{\a,\b}_{s_2,\frac12,\sigma}}+\| w_1\|_{X^{\a,\b}_{s_2,\frac12,\sigma}}\| w_2\|_{X^{\a,\b}_{s_1,\frac12,\sigma}}\Big).  \label{bilin, d11}
\end{align}
holds for any $w_1,w_2\in X^{\a,\b}_{s,\frac12,\sigma}$ with some positive constant $ C = C(s_1, s_2, \a, \eps_0, \sigma)$.
\end{lem}

\begin{lem}\label{Lemma, bilin, d21}
Let $-\frac34<s_1<s_2\leq 3$, $0< T \leq 1$, $\a\neq 0$ and $|\b|\leq 1$. Then there exists $\epsilon_0 = \eps(s_1,s_2,\a) >0$ such that for any $\sigma\in(\frac12, \frac12 + \epsilon_0]$, the following bilinear estimate
\begin{align}
&\quad\, \l\| \eta\Big(\frac{t}{T}\Big) \int^t_0 W_R^{\a, \b} (t-\tau) (w_1 w_2)_x\,d\tau\r\|_{X^{\a,\b}_{s_2,\frac12, \si}} \nonumber\\
&\leq  CT^{\epsilon_0}	\Big( \| w_1\|_{X^{\a,\b}_{s_1,\frac12,\sigma}}\| w_2\|_{X^{-\a,-\b}_{s_2,\frac12,\sigma}}+\| w_1\|_{X^{\a,\b}_{s_2,\frac12,\sigma}}\| w_2\|_{X^{-\a,-\b}_{s_1,\frac12,\sigma}}\Big).  \label{bilin, d21}
\end{align}
holds for any $w_1\in X^{\a,\b}_{s,\frac12,\sigma}$ and $w_2\in X^{-\a,-\b}_{s,\frac12,\sigma}$ with some positive constant $ C = C(s_1, s_2, \a, \eps_0, \sigma)$.
\end{lem}

Now we are ready to justify Proposition \ref{Prop, Tmax}.

\begin{proof}[{\bf Proof of Proposition \ref{Prop, Tmax}}]
For $s=s_1$ or $s_2$, denote 
$\mathcal{Y}_{s}=Y^{1,1}_{s,\frac12,\si}\times Y^{-1,-1}_{s,\frac12,\si}$.
If $t^*:=T^{s_2}_{max}<T^{s_1}_{max}$,  then  by  Lemma \ref{blowup},
\be\label{s1nb} 
r:=\sup_{0 \leq t\leq t^*} \l(\| u(\cdot, t) \|_{H^{s_1}(\R^+)} + \| v(\cdot, t) \|_{H^{s_1}(\R^+)}\r) < +\infty, 
\ee
and 
\be\label{s2nub}
\lim_{t\rightarrow t^*}   \l(\| u(\cdot, t) \|_{H^{s_2}(\R^+)} + \| v(\cdot, t) \|_{H^{s_2}(\R^+)}\r) = +\infty.
\ee
According to Corollary \ref{cwp-1} and its proof,  there exists some $T=T(s_2,r)>0$ such that for any $\delta\in (0,T)$,
\[\|(u,v)\|_{{\mathcal Y}_{s_1}(\R^+\times (t^*-2\delta,t^*-\delta))}\leq \a_{r,s_1}\l( \|(p,q)\|_{\H^{s_1}_x(\R^+)} +\|(a,b,c)\|_{\H_t^{s_1}(\R^+)}\r).\]
In addition, according to Lemma \ref{Lemma, bilin, gc1}- \ref{Lemma, bilin, d21}, one also has, 
\begin{align*}
&\sup_{t^*-2\delta<t<t^*-\delta}\|(u,v)\|_{H^{s_2}(\R^+)\times H^{s_2}(\R^+)}\\
\leq & \  \|(u,v)\|_{{\mathcal Y}_{s_2}(\R^+\times (t^*-2\delta,t^*-\delta))}\\
\leq& \ C_1 \l(\|(p,q)\|_{\H^{s_2}_1(\R^+)} +\|(a,b,c)\|_{\H_2^{s_2}(\R^+)} \r) +  C_2 \delta^{\eps_0} \Big(\|u\|_{X^{1,1}_{s_1,\frac12,\si}} + \|v\|_{X^{-1,-1}_{s_1,\frac12,\si}}\Big) \Big(\|u\|_{X^{1,1}_{s_2,\frac12,\si}} + \|v\|_{X^{-1,-1}_{s_2,\frac12,\si}}\Big),
\end{align*}
where $ \eps_0>0 $ and $ \sigma \in \big(\frac12, \frac12 + \eps_0] $ are some constants which only depend on $ s_1 $ and $ s_2 $. Hence, we can choose $\delta$ small enough such that 
\[C_2 \delta^{\eps_0} \Big(\|u\|_{X^{1,1}_{s_1,\frac12,\si}} +\|v\|_{X^{-1,-1}_{s_1,\frac12,\si}}\Big)\leq \frac12. \]
This yields
\[\sup_{t^*-2\delta<t<t^*-\delta}\|(u,v)\|_{H^{s_2}(\R^+)\times H^{s_2}(\R^+)}\leq 2C_1\l(\|(p,q)\|_{\H^{s_2}_x(\R^+)} +\|(a,b,c)\|_{\H_t^{s_2}(\R^+)} \r),\]
which contradicts with (\ref{s2nub}). 
Therefore, $T^{s_2}_{max}=T^{s_1}_{max}$ and the proof is complete.
\end{proof}

The combination of Corollary \ref{cwp-1} through Proposition \ref{Prop, Tmax} enables one to justify Theorem \ref{Thm, exist of ms}.

\begin{proof}[\bf Proof of Theorem \ref{Thm, exist of ms}]
According to Corollary \ref{cwp-1}, for given  $(p,q)\in \mcal{H}_{x}^{s}(\m{R}^{+})$ and $(a,b,c)\in \mcal{H}_{t}^{s}(\m{R}^{+})$, 
one can obtain a solution $(u,v)\in Y^{1,1}_{s,\frac12,\sigma}(\Omega _T) \times Y^{-1,-1}_{s,\frac12,\sigma}(\Omega _T) $. We will show that such a solution $(u,v)$ is, in fact, a mild solution of IBVP \eqref{ckdv-1}.
We choose  sequences  
$\{(p_n,q_n)\}_{n\geq 1} \subset \mcal{H}_{x}^{3}(\m{R}^{+})$ and $\{(a_n,b_n,c_n)\}_{n\geq 1} \subset \mcal{H}_{t}^{3}(\m{R}^{+})$
such that 
\[\lim_{n\rightarrow \infty}(p_n,q_n)=(p,q)\quad \mbox{in} \quad \H_x^s(\R^+), \qquad \lim_{n\rightarrow \infty}(a_n,b_n,c_n)=(a,b,c)\quad \mbox{in} \quad \H_t^s(\R^+).\]
Since the solution map in Corollary \ref{cwp-1} is continuous, there exist solutions $ \{(u_n, v_n)\}_{n\geq 1} $ under the initial-boundary conditions with data  $\{(p_n,q_n,a_n,b_n,c_n)\}_{n\geq 1}$ such that $ \{(u_n, v_n)\}_{n\geq 1} $ lie in $ C(0,T; H^s(\R^+)) \times C(0,T; H^s(\R^+))$ and  
\[\lim_{n\to \infty} u_n=u, \quad \lim_{n\to \infty} v_n=v \quad \mbox{in } C(0,T; H^s(\R^+)).\]
On the other hand, according to Lemma \ref{blowup}, Lemma \ref{wp-classical} and Proposition \ref{Prop, Tmax}, we know both $u_n$ and $v_n$ belong to $ C(0,T^3_{n,max}; H^3(\R^+))$ for $ n\geq 1$ with $T^3_{n,max}\geq T$. 
This leads to $u_n, v_n \in C(0,T; H^3(\R^+))$, so the proof is now complete.
\end{proof}

Finally, we show that the mild solution of IBVP \eqref{ckdv-1}  is unique, that is to prove Theorem \ref{Thm, uniq of ms}.

\begin{proof}[\bf Proof of Theorem \ref{Thm, uniq of ms}]
For given 
$(p,q)\in \mcal{H}_{x}^{s}(\m{R}^{+})$ and 
$(a,b,c)\in \mcal{H}_{t}^{s}(\m{R}^{+})$,
we assume that there are two mild solutions for the IBVP \eqref{ckdv-1}, denoted as, $(u^{(1)}, v^{(1)})$ and $ (u^{(2)}, v^{(2)})$ 
with
\[u^{(1)}, v^{(1)} \in C(0,T_1; H^s(\R^+)), \qquad  u^{(2)}, v^{(2)} \in C(0,T_2; H^s(\R^+)),\]
for some $T_1, T_2>0$.  Without loss of generality, one can set $T:=T_1\leq T_2$. According to the definition of mild solutions, one can find two sequences of classic solutions of the IBVP \eqref{ckdv-1},
\[(u^{(1)}_n, v^{(1)}_n), (u^{(2)}_n, v^{(2)}_n) \in C(0,T; H^3(\R^+))\times C(0,T; H^3(\R^+)), \quad n\geq 1,\]
with 
\begin{equation*}
\begin{cases}
u^{(1)}_n(x,0)=p^{(1)}_n(x), \quad v^{(1)}_n(x,0)=q^{(1)}_n(x),\\
u^{(1)}_n(0,t)=a^{(1)}_n(t), \quad v^{(1)}_n(0,t)=b^{(1)}_n(t), \quad \partial_x (v^{(1)}_n-u^{(1)}_n)(0,t)=c^{(1)}_n(t),
\end{cases}
\end{equation*}
and
\begin{equation*}
\begin{cases}
u^{(2)}_n(x,0)=p^{(2)}_n(x), \quad v^{(2)}_n(x,0)=q^{(2)}_n(x),\\
u^{(2)}_n(0,t)=a^{(2)}_n(t), \quad v^{(2)}_n(0,t)=b^{(2)}_n(t), \quad \partial_x (v^{(2)}_n-u^{(2)}_n)(0,t)=c^{(2)}_n(t),
\end{cases}
\end{equation*}
such that
\[
\lim_{n\to \infty} (u^{(1)}_n, v^{(1)}_n) = (u^{(1)}, v^{(1)}), \quad
\lim_{n\to \infty} (u^{(2)}_n, v^{(2)}_n) = (u^{(2)}, v^{(2)}), 
\quad \mbox{in} \quad C(0,T;H^s(\R^+)\times H^s(\R^+)),
\]
\[
\lim_{n\to \infty} (p^{(1)}_n,q^{(1)}_n) =(p,q), \quad \lim_{n\to \infty} (p^{(2)}_n,q^{(2)}_n) =(p,q) \quad \mbox{in} \quad \H^s_x(\R^+),
\]
and 
\[
\lim_{n\to \infty} (a^{(1)}_n,b^{(1)}_n,c^{(1)}_n) =(a,b,c), \quad 
\lim_{n\to \infty} (a^{(2)}_n,b^{(2)}_n,c^{(2)}_n) =(a,b,c) \quad 
\mbox{in} \quad \H^s_t(\R^+).
\]
We then denote 
$$(\tilde{u}, \tilde{v})\in \l[C(0,T_3; H^s(\R^+))\r]^2,$$ $$(\widetilde{u}_n^{(1)}, \widetilde{v}_n^{(1)})\in \l[C(0,T^1_{n,max}; H^3(\R^+))\r]^2,$$   $$(\widetilde{u}_n^{(2)}, \widetilde{v}_n^{(2)})\in \l[C(0,T^2_{n,max}; H^3(\R^+))\r]^2,$$
to be the solutions provided by Corollary \ref{cwp-1} with conditions corresponding to 
\[(p,q,a,b,c), \quad (p^{(1)}_n,q^{(1)}_n,a^{(1)}_n,b^{(1)}_n,c^{(1)}_n), \quad (p^{(2)}_n,q^{(2)}_n,a^{(2)}_n,b^{(2)}_n,c^{(2)}_n),\]
respectively.  We define $T^*:= \min\{ T, T_3\} $. According to Lemma \ref{blowup}, we can infer that $T^1_{n, max}\geq T^*$,  $T^2_{n, max}\geq T^*$ and
\[(u^{(1)}_n, v^{(1)}_n)=(\widetilde{u}_n^{(1)}, \widetilde{v}_n^{(1)}), \quad (u^{(2)}_n, v^{(2)}_n)=(\widetilde{u}_n^{(2)}, \widetilde{v}_n^{(2)}), \quad \mbox{in }[ C(0,T^*; H^3(\R^+))]^2,\]
due to the uniqueness of the classic solution of the IBVP \eqref{ckdv-1}. Moreover, according to the continuity of the solution map in Corollary \ref{cwp-1}, one has
\[\lim_{n\to \infty} (\widetilde{u}_n^{(1)}, \widetilde{v}_n^{(1)})=(\tilde{u}, \tilde{v}), \quad \lim_{n\to \infty} (\widetilde{u}_n^{(2)}, \widetilde{v}_n^{(2)})=(\tilde{u}, \tilde{v}), \quad \mbox{in } [C(0,T^*;H^s(\R^+))]^2.\]
Therefore, we deduce that
$(u^{(1)}, v^{(1)})=(\tilde{u}, \tilde{v})=(u^{(2)}, v^{(2)}) $ in $[C(0,T^*;H^s(\R^+))]^2$. 
Finally, we can use continuity property of the solutions $(u^{(1)}, v^{(1)})$, $(u^{(2)}, v^{(2)})$ and $(\tilde{u}, \tilde{v})$  to show $T_1=T_2=T_3$.  
\end{proof}




\appendix


\bigskip

\section{Proofs of Lemma \ref{Lemma, bilin, gc1}--Lemma \ref{Lemma, bilin, d21}}
\label{Sec,  proof of bilin smoothing}

The proofs for Lemma \ref{Lemma, bilin, gc1}, Lemma \ref{Lemma, bilin, d11} and Lemma \ref{Lemma, bilin, d21} are similar, so we will only prove Lemma \ref{Lemma, bilin, d21}. In order to make the argument more clearly, we first present several auxiliary results. We start with the following Lemma \ref{Lemma, bilin smoothing} which is a refinement of Proposition \ref{Prop, bilin, d2} on the temporal regularity of the term $ (uv)_x $.

\begin{lem}\label{Lemma, bilin smoothing}
Let $-\frac34<s\leq 3$, $\a\neq 0$ and $|\b|\leq 1$. Then there exists $\eps_0=\eps_0(s,\a)>0$ such that for any $\sigma\in(\frac12,\frac12+\eps_0]$, 
\be\label{B1}
\|(uv)_{x}\|_{X^{\a,\b}_{s,\sigma-1+\eps_0}}\leq C\|u\|_{X^{\a,\b}_{s,\frac12,\sigma}}\|v\|_{X^{-\a,-\b}_{s,\frac12,\sigma}},
\ee
where $C=C(s,\a,\sigma)$. 
\end{lem}
\begin{proof}
By duality, it suffices to prove 
\[\iint_{\m{R}^2}\xi\,\wh{uv}\wh{w}\,d\xi\,d\tau\leq C\|u\|_{X^{\a,\b}_{s,\frac12,\sigma}}\|v\|_{X^{-\a,-\b}_{s,\frac12,\sigma}}\|w\|_{X^{\a,\b}_{-s,1-\sigma-\eps_0}},\]
for any $u\in X^{\a,\b}_{s,\frac12,\sigma}$,  $v\in X^{-\a,-\b}_{s,\frac12,\sigma}$ and $w\in X^{\a,\b}_{-s,1-\sigma-\eps_0}$. Similar to the proof of Proposition \ref{Prop, bilin, d2}, we denote 
\[L_{1}=\tau_{1}-\phi^{\a,\b}(\xi_1),\quad L_2=\tau_2-\phi^{-\a,-\b}(\xi_2),\quad L_3=\tau_3-\phi^{\a,\b}(\xi_3). \]
and 
\[M_{1}=\la L_1\ra^{\frac12}+\mb{1}_{\{e^{|\xi_1|}\leq 3+|\tau_1|\}}\la L_1\ra^{\sigma},\quad M_{2}=\la L_2\ra^{\frac12}+\mb{1}_{\{e^{|\xi_2|}\leq 3+|\tau_2|\}}\la L_2\ra^{\sigma}.\]
In addition, we define
\[\begin{array}{c}
f_{1}(\xi_1,\tau_1)=\la\xi_1\ra^{s}M_{1}\wh{u}(\xi_1,\tau_1),\quad f_{2}(\xi_2,\tau_2)=\la\xi_2\ra^{s}M_{2}\wh{v}(\xi_2,\tau_2),\\
f_{3}(\xi_3,\tau_3)=\la\xi_3\ra^{-s}\la L_{3}\ra^{1-\sigma-\eps_0}\wh{w}(\xi_3,\tau_3).\end{array}\]
Then it reduces to justify
\be\label{weighted l2 for bilin smoothing}
\int\limits_{A}\frac{|\xi_{3}|\la\xi_{1}\ra^{\rho}\la\xi_{2}\ra^{\rho}\prod\limits_{i=1}^{3}|f_{i}(\xi_{i},\tau_{i})|}{\la\xi_{3}\ra^{\rho}M_{1}M_{2}\la L_{3}\ra^{1-\sigma-\eps_0}} \leq C\,\prod_{i=1}^{3}\|f_{i}\|_{L^{2}_{\xi\tau}},\ee
where $\rho=-s$ and the set $ A $ is as defined in (\ref{int domain}). In the remaining proof, instead of introducing $\sigma_0$ as in (\ref{sigma_0}), we define 
$\eps_0=\frac{1}{16}-\frac{1}{12}\,\rho$.
Then for any $\sigma\in(\frac12,\frac12+\eps_0]$, (\ref{weighted l2 for bilin smoothing}) can be proved in an analogous way as the proof of (\ref{bilin, d2}) in Proposition \ref{Prop, bilin, d2}.

\end{proof}

Next, we will take advantage of Proposition \ref{Lemma, bilin smoothing} to deduce a slightly stronger estimate since the spatial regularity requirement on the right-hand side of (\ref{B2}) is weaker than that of (\ref{B1}).

\begin{lem}\label{Lemma, bilin, high-low reg}
Let $-\frac34<s_1\leq s_2\leq 3$, $\a\neq 0$, $|\b|\leq 1$. Then there exists $\eps_0=\eps_0(s_1,s_2,\a)>0$ such that for any $\sigma\in(\frac12,\frac12+\eps_0]$, 
\be\label{B2}
\|(uv)_{x}\|_{X^{\a,\b}_{s_2,\sigma-1+\eps_0}}\leq C\Big(\|u\|_{X^{\a,\b}_{s_1,\frac12,\sigma}}\|v\|_{X^{-\a,-\b}_{s_2,\frac12,\sigma}}+ \|u\|_{X^{\a,\b}_{s_2,\frac12,\sigma}}\|v\|_{X^{-\a,-\b}_{s_1,\frac12,\sigma}} \Big),
\ee
where $C=C(s_1,s_2,\a,\sigma)$. 
\end{lem}

\begin{proof}
By duality, it suffices to prove 
\[\iint_{\m{R}^2}\xi\,\wh{uv}\wh{w}\,d\xi\,d\tau\leq C\Big(\|u\|_{X^{\a,\b}_{s_1,\frac12,\sigma}}\|v\|_{X^{-\a,-\b}_{s_2,\frac12,\sigma}}+ \|u\|_{X^{\a,\b}_{s_2,\frac12,\sigma}}\|v\|_{X^{-\a,-\b}_{s_1,\frac12,\sigma}} \Big)\|w\|_{X^{\a,\b}_{-s_2,1-\sigma-\eps_0}},\]
for any $u\in X^{\a,\b}_{s_2,\frac12,\sigma}$,  $v\in X^{-\a,-\b}_{s_2,\frac12,\sigma}$ and $w\in X^{\a,\b}_{-s_2,1-\sigma-\eps_0}$.  Denote $L_1, L_2, L_3, M_1, M_2$ in the same way as the above proof for Lemma \ref{Lemma, bilin smoothing}, and define
\[\begin{array}{c}
f_{1}(\xi_1,\tau_1)=\la\xi_1\ra^{s_1}M_{1}\wh{u}(\xi_1,\tau_1),\ \, f_{2}(\xi_2,\tau_2)=\la\xi_2\ra^{s_1}M_{2}\wh{v}(\xi_2,\tau_2),\ \,
f_{3}(\xi_3,\tau_3)=\la\xi_3\ra^{-s_2}\la L_{3}\ra^{1-\sigma-\eps_0}\wh{w}(\xi_3,\tau_3).
\end{array}\]
Denote $\rho=-s_1$ and $r=s_2-s_1\geq 0$. Then similar to (\ref{weighted l2 for bilin smoothing}), it reduces to show 
\be\label{bilin, high-low reg}\begin{split}
&\quad\, \int\limits_{A}\la\xi_3\ra^{r}\frac{|\xi_{3}|\la\xi_{1}\ra^{\rho}\la\xi_{2}\ra^{\rho}|f_{1}(\xi_{1},\tau_{1})f_{2}(\xi_{2},\tau_{2})f_{3}(\xi_{3},\tau_{3})|}{\la\xi_{3}\ra^{\rho}M_{1}M_{2}\la L_{3}\ra^{1-\sigma-\eps_0}} \\
&\leq C\,\Big(\|f_{1}\|_{L^{2}_{\xi\tau}}\|\la\xi\ra^{r}f_{2}\|_{L^{2}_{\xi\tau}}+\|\la\xi\ra^{r}f_{1}\|_{L^{2}_{\xi\tau}}\|f_{2}\|_{L^{2}_{\xi\tau}}\Big)\|f_{3}\|_{L^{2}_{\xi\tau}}.
\end{split}\ee
Since $r\geq 0$ and $\sum_{i=1}^{3}\xi_i=0$, then 
$\la\xi_3\ra^{r}=\la\xi_1+\xi_2\ra^{r}\leq 2^{r}\big(\la\xi_1\ra^{r}+\la\xi_2\ra^{r}\big)$.
Consequently,
$\text{LHS of (\ref{bilin, high-low reg})}\leq 2^{r}(I+II)$,
where 
\begin{align*}
I= \int\limits_{A}\frac{|\xi_{3}|\la\xi_{1}\ra^{\rho}\la\xi_{2}\ra^{\rho}|\la\xi_1\ra^{r}f_{1}(\xi_{1},\tau_{1})|\,|f_{2}(\xi_{2},\tau_{2})f_{3}(\xi_{3},\tau_{3})|}{\la\xi_{3}\ra^{\rho}M_{1}M_{2}\la L_{3}\ra^{1-\sigma-\eps_0}},\\
II=\int\limits_{A}\frac{|\xi_{3}|\la\xi_{1}\ra^{\rho}\la\xi_{2}\ra^{\rho}|f_{1}(\xi_{1},\tau_{1})|\, |\la\xi_2\ra^{r}f_{2}(\xi_{2},\tau_{2})f_{3}(\xi_{3},\tau_{3})|}{\la\xi_{3}\ra^{\rho}M_{1}M_{2}\la L_{3}\ra^{1-\sigma-\eps_0}}.
\end{align*}
According to Lemma \ref{Lemma, bilin smoothing},  
\[
I\leq C\, \|\la\xi\ra^{r}f_{1}\|_{L^{2}_{\xi\tau}} \|f_{2}\|_{L^{2}_{\xi\tau}} \|f_{3}\|_{L^{2}_{\xi\tau}}, \quad 
II\leq C\,\|f_{1}\|_{L^{2}_{\xi\tau}}\|\la\xi\ra^{r} f_{2}\|_{L^{2}_{\xi\tau}} \|f_{3}\|_{L^{2}_{\xi\tau}}.
\]
Thus,  (\ref{bilin, high-low reg}) is verified.
\end{proof}

Recalling the cut-off function $\eta\in C^{\infty}_{0}(\m{R})$ which satisfies $\eta(t)=1$ on $(-1,1)$ and $\text{supp}\, \eta\subset (-2,2)$, the following is a classical estimate on the Fourier restriction norms when the time is localized.

\begin{lem}\label{Lemma, local FR est}
For any $s\in\m{R}$,  $-\frac12<b_1\leq b_2<\frac12$, $\a\neq 0$, $|\b|\leq 1$ and $0<T\leq 1$, there exists a constant $C=C(s,b_1,b_2)$ such that 
\be\label{B3}
\Big\| \eta\Big(\frac{t}{T}\Big)w \Big\|_{X^{\a,\b}_{s,b_1}}\leq C\,T^{b_2-b_1}\| w\|_{X^{\a,\b}_{s,b_2}},\quad\forall\, w\in X^{\a,\b}_{s,b_2}.
\ee
\end{lem}

The proof of this estimate is well-known, see e.g. Lemma 2.11 in \cite{Tao06}.  Combining Lemma \ref{Lemma, bilin, high-low reg} and Lemma \ref{Lemma, local FR est} yields the outcome below.

\begin{lem}\label{Lemma, local bilin, high-low reg}
Let $-\frac34<s_1\leq s_2\leq 3$, $\a\neq 0$, $|\b|\leq 1$ and $0<T\leq 1$. Then there exists $\eps_0=\eps_0(s_1,s_2,\a)>0$ such that for any $\sigma\in(\frac12,\frac12+\eps_0]$, 
\be\label{B4}
\Big\|\eta\Big(\frac{t}{T}\Big) (uv)_{x}\Big\|_{X^{\a,\b}_{s_2,\sigma-1}}\leq CT^{\eps_0}\Big(\|u\|_{X^{\a,\b}_{s_1,\frac12,\sigma}}\|v\|_{X^{-\a,-\b}_{s_2,\frac12,\sigma}}+ \|u\|_{X^{\a,\b}_{s_2,\frac12,\sigma}}\|v\|_{X^{-\a,-\b}_{s_1,\frac12,\sigma}} \Big),
\ee
where $C=C(s_1,s_2,\a,\sigma)$. 
\end{lem}

\begin{proof}
Firstly, we choose $ \eps_0 $ as that in Lemma \ref{Lemma, bilin, high-low reg}. Then for any $\sigma\in(\frac12,\frac12+\eps_0]$, we choose $ w=(uv)_x $, $ s=s_2 $, $ b_1=\sigma-1 $ and $ b_2 = \sigma-1+\eps_0 $ in (\ref{B3}) to obtain 
\[ 
\Big\|\eta\Big(\frac{t}{T}\Big) (uv)_{x}\Big\|_{X^{\a,\b}_{s_2,\sigma-1}} \leq C T^{\eps_0} \| (uv)_x \|_{X^{\a,\b}_{s_2, \sigma - 1}}. 
\]
Then we can apply Lemma \ref{Lemma, bilin, high-low reg} to verify (\ref{B4}).
\end{proof}

Similar to Lemma \ref{Lemma, local FR est}, the estimate for modified Fourier restriction norms of the localized (in time) Duhamel term associated with the semigroup operator $ W^{\a,\b}_R $ can also be established.

\begin{lem}\label{Lemma,  local Duhamel}
Let $s\in\m{R}$, $\a\neq 0$, $|\b|\leq 1$, $\frac12<\si< 1$ and $0<T\leq 1$. Then there exists a constant $C=C(s,\a, \sigma)$ such that
\be\label{local Duhamel}
\left\|\eta\Big(\frac{t}{T}\Big)\int^t_{0} W_R^{\a, \b} (t-\tau)F\,d\tau\right\|_{X^{\a,\b}_{s,\frac12,\sigma}}\leq  C\Big\| \eta\Big(\frac{t}{2T}\Big)F  \Big\|_{X^{\a,\b}_{s,\sigma-1}}.\ee
\end{lem}
\begin{proof}
Define $g(x,t)=\eta\big(\frac{t}{2T}\big)F(x,t)$. Then it is easily seen that 
\begin{align*}
\text{LHS of (\ref{local Duhamel})} &= \left\|\eta\Big(\frac{t}{T}\Big)\int^t_0 W_R^{\a, \b} (t-\tau)g \, d\tau\right\|_{X^{\a,\b}_{s,\frac12,\sigma}}.
\end{align*}
Then similar to Lemma \ref{Lemma, lin est} (also see Lemma 2.1 in \cite{GTV97}),  we have 
\begin{align*}
\left\|\eta\Big(\frac{t}{T}\Big)\int^t_0 W_R^{\a, \b} (t-\tau)g \,d\tau\right\|_{X^{\a,\b}_{s,\frac12,\sigma}} &\leq C\| g \|_{X^{\a,\b}_{s,\sigma-1}}= C\Big\| \eta\Big(\frac{t}{2T}\Big)F\Big\|_{X^{\a,\b}_{s,\sigma-1}}.
\end{align*}
\end{proof}

Now we have developed all the tools needed to justify Lemma \ref{Lemma, bilin, d21}.

\begin{proof}[{\bf Proof of Lemma \ref{Lemma, bilin, d21}}]
Firstly, we choose $ \eps_0 =\eps_0(s_1,s_2, \a) > 0 $ as that in Lemma \ref{Lemma, local bilin, high-low reg}. Then for any $\sigma \in (\frac12, \frac12 + \eps_0]$, we apply Lemma \ref{Lemma,  local Duhamel}, with $ s=s_2 $ and $ F = (w_1 w_2)_x $, to conclude that 
\[\l\| \eta\Big(\frac{t}{T}\Big) \int^t_0 W_R^{\a, \b} (t-\tau)  (w_1 w_2)_x\r\|_{X^{\a,\b}_{s_2,\frac12, \si}}\leq C \Big\|\eta\Big(\frac{t}{2T}\Big)(w_1 w_2)_x\Big\|_{X^{\a,\b}_{s_2,\sigma-1}}.\]
Then it follows from Lemma \ref{Lemma, local bilin, high-low reg} that 
\[  \Big\|\eta\Big(\frac{t}{2T}\Big)(w_1 w_2)_x\Big\|_{X^{\a,\b}_{s_2,\sigma-1}}\leq C T^{\eps_0}\Big(\|w_1\|_{X^{\a,\b}_{s_1,\frac12,\sigma}}\|w_2\|_{X^{-\a,-\b}_{s_2,\frac12,\sigma}}+ \|w_1\|_{X^{\a,\b}_{s_2,\frac12,\sigma}}\|w_2\|_{X^{-\a,-\b}_{s_1,\frac12,\sigma}} \Big),\]
where $C=C(s_1,s_2,\a,\sigma)$.  Combining the above two estimates leads to (\ref{bilin, d21}) in Lemma \ref{Lemma, bilin, d21}.
\end{proof}

\section*{Acknowledgements}
S.  Li is supported by the National Natural Science Foundation of China (no.  12001084 and no.  12071061), and the Applied Basic Research Program of Sichuan Province (no. 2020YJ0264).



{\small 

\newcommand{\etalchar}[1]{$^{#1}$}

}

\bigskip

\thanks{(S. Li) School of Mathematical Sciences, University of Electronic Science and Technology of China, Chengdu, Sichuan 611731, China}

\thanks{Email: lish@uestc.edu.cn}

\medskip

\thanks{(M. Chen) Department of Mathematics, Purdue University, West Lafayette, IN 47907, USA}

\thanks{Email: chen45@purdue.edu}

\medskip

\thanks{(X. Yang) Department of Mathematics, University of California, Riverside, CA, 92521, USA}

\thanks{Email: xiny@ucr.edu}

\medskip

\thanks{(B.-Y. Zhang) Department of Mathematical Sciences, University of Cincinnati, Cincinnati, OH 45221, USA}

\thanks{Email: zhangb@ucmail.uc.edu}

\end{document}